\documentclass[12pt]{scrartcl}    
\usepackage{a4} 
\usepackage{amsmath}    
\usepackage{paralist}
\usepackage{amssymb}  
\usepackage{amsfonts}  
\usepackage{mathrsfs}  
\usepackage{dsfont}
\usepackage{latexsym} 
\usepackage{xcolor}
\usepackage{bbm,exscale}
\definecolor{Myblue}{rgb}{0,0,0.6}  
\usepackage[colorlinks,citecolor=Myblue,linkcolor=Myblue,urlcolor=Myblue,pdfpagemode=None]{hyperref}
\usepackage{amsthm}
\usepackage{accents}
\usepackage[square,numbers,sort&compress]{natbib} 
\usepackage[all,cmtip]{xy}
\usepackage{ifthen} 
\usepackage{bbding}
\usepackage{stmaryrd}  
\usepackage{wasysym}
\usepackage{verbatim}
\usepackage{bbding} 
\usepackage{soul}  
\usepackage[yyyymmdd,hhmmss]{datetime}
\usepackage{booktabs}
\usepackage{enumitem}
\usepackage{color}
\usepackage{tikz}
\usepackage{textcomp}
\usepackage{gensymb}
\usepackage{pdfpages}
\usepackage{tikz-cd}
\usepackage{tikz-3dplot}
\usepackage{pgfplots}
\pgfplotsset{width=7cm,compat=1.8}
\usetikzlibrary{decorations.pathreplacing}
\usetikzlibrary{decorations.markings}
\usetikzlibrary{patterns}
\usepgflibrary{shapes.geometric}

\tikzset{
	string/.style={draw=#1, postaction={decorate}, decoration={markings,mark=at position .51 with {\arrow[color=#1]{>}}}},
	costring/.style={draw=#1, postaction={decorate}, decoration={markings,mark=at position .51 with {\arrow[draw=#1]{<}}}},
	ostring/.style={draw=#1, postaction={decorate}, decoration={markings,mark=at position .47 with {\arrow[draw=#1]{>}}}},
	ustring/.style={draw=#1, postaction={decorate}, decoration={markings,mark=at position .56 with {\arrow[draw=#1]{>}}}},
	oostring/.style={draw=#1, postaction={decorate}, decoration={markings,mark=at position .43 with {\arrow[draw=#1]{>}}}},
	uustring/.style={draw=#1, postaction={decorate}, decoration={markings,mark=at position .59 with {\arrow[draw=#1]{>}}}},
	directed/.style={string=blue!50!black}, 
	odirected/.style={ostring=blue!50!black}, 
	udirected/.style={ustring=blue!50!black}, 
	oodirected/.style={oostring=blue!50!black}, 
	uudirected/.style={uustring=blue!50!black},     
	redirected/.style={costring= blue!50!black},
	redirectedgreen/.style={costring= green!50!black},
	directedgreen/.style={string= green!50!black},
	redirectedlightgreen/.style={costring= green!65!black},
	directedlightgreen/.style={string= green!65!black},
	redirectedred/.style={costring= red!50!black},
	directedred/.style={string= red!50!black}%
}

\tikzset{-dot-/.style={decoration={
			markings,
			mark=at position 0.5 with {\fill circle (1.875pt);}},postaction={decorate}}}

\tikzset{
	Fdot/.style={circle, draw, fill, inner sep=0pt}, 
	Odot/.style={circle, draw, inner sep=0.1pt, minimum size=0.1cm}
}

\def\nicedashedcolourscheme{\shadedraw[top color=blue!22, bottom color=blue!22, draw=gray, dashed]}
\def\nicedashedpalecolourscheme{\shadedraw[top color=blue!12, bottom color=blue!12, draw=gray, dashed]}
\def\nicehalfpalecolourscheme{\shadedraw[top color=blue!22, bottom color=blue!22, draw=white]}
\def\nicenotpalecolourscheme{\shadedraw[top color=blue!32, bottom color=blue!32, draw=white]}
\def\nicecolourscheme{\shadedraw[top color=blue!22, bottom color=blue!22, draw=blue!22]}
\def\nicepalecolourscheme{\shadedraw[top color=blue!12, bottom color=blue!12, draw=white]}
\def\nicenocolourscheme{\shadedraw[top color=gray!2, bottom color=gray!25, draw=white]}
\def\nicereallynocolourscheme{\shadedraw[top color=white!2, bottom color=white!25, draw=white]}
\def\boringcolourscheme{\draw[fill=blue!20, dashed]}

\newcommand\tikzzbox[1]
{#1}

\vfuzz \hfuzz
\makeatletter
\newcommand{\raisemath}[1]{\mathpalette{\raisem@th{#1}}}
\newcommand{\raisem@th}[3]{\raisebox{#1}{$#2#3$}}
\makeatother

\renewcommand{\H}{\mathcal H}
\newcommand{\A}{\mathcal{A}}
\newcommand{\Aca}{\mathcal{A}^{\mathcal{S}}}
\newcommand{\orb}{\mathcal{A}}
\newcommand{\sgn}{\mathrm{sgn}}
\renewcommand{\leq}{\leqslant}
\renewcommand{\geq}{\geqslant}
\newcommand{\rp}[1]{\widetilde{#1}}

\newcommand{\E}{\text{e}} 
\newcommand{\I}{\text{i}}
\newcommand{\B}{\mathcal{B}}
\newcommand{\Borb}{\B_{\mathrm{orb}}}
\newcommand{\Beq}{\B_{\mathrm{eq}}}
\newcommand{\C}{\mathds{C}}
\newcommand{\D}{\mathds{D}}
\newcommand{\Ee}{\mathds{E}}
\newcommand{\K}{\mathds{K}}
\newcommand{\M}{\mathds{M}}
\newcommand{\N}{\mathds{N}}
\newcommand{\Q}{\mathds{Q}}
\newcommand{\R}{\mathds{R}}
\newcommand{\Z}{\mathds{Z}}
\newcommand{\IP}{\mathds{P}}
\def\1{\ifmmode\mathrm{1\!l}\else\mbox{\(\mathrm{1\!l}\)}\fi}
\newcommand{\one}{\mathbbm{1}}
\newcommand{\be}{\begin{equation}}
\newcommand{\ee}{\end{equation}}
\newcommand{\bes}{\begin{equation*}}
\newcommand{\ees}{\end{equation*}}
\newcommand{\cc}[1] {\overline{#1}}
\newcommand{\inv}[0]{{-1}}
\newcommand{\DC}{\mathds{D}^{\mathcal C}}
\newcommand{\cat}{{\mathcal C}}

\newcommand{\inver}[0]{\times}
\newcommand{\chirel}[0]{\chi_{\mathrm{sym}}}
\newcommand{\Dec}[0]{\mathrm{Dec}}
\newcommand{\interior}[1]{%
	{\kern0pt#1}^{\mathrm{o}}%
}
\newcommand{\Strat}[0]{\mathrm{Strat}}
\newcommand{\Stratdef}[0]{\mathrm{Strat}^{\mathrm{def}}}

\newcommand{\sVir}{\mathsf{sVir}}
\newcommand{\MF}{\operatorname{MF}_{\operatorname{bi}}}
\newcommand{\MFW}{\operatorname{MF}_{\operatorname{bi}}(W)}
\newcommand{\MFR}{\operatorname{MF}^\text{R}_{\operatorname{bi}}}
\newcommand{\DG}{\operatorname{DG}_{\operatorname{bi}}}
\newcommand{\DGW}{\operatorname{DG}_{\text{bi}}(W)}
\newcommand{\DGR}{\operatorname{DG}^\text{R}_{\text{bi}}}
\newcommand{\id}{\text{id}}
\newcommand{\KMF}{K_{0}(\operatorname{MF}_{\text{bi}}}
\newcommand{\Ext}{\operatorname{Ext}}
\newcommand{\Hom}{\operatorname{Hom}}
\newcommand{\End}{\operatorname{End}}
\newcommand{\modu}{\operatorname{mod}}
\def\LG{\mathcal{LG}}
\def\LGgr{\mathcal{LG}^{\mathrm{gr}}}
\def\LGgrs{\mathcal{LG}'^{\mathrm{gr}}}
\def\LGgrso{\mathcal{LG}'^{\mathrm{gr}}_{\mathrm{orb}}}
\def\LGs{\mathcal{LG}'}
\def\LGsorb{\mathcal{LG}'_{\mathrm{orb}}}
\def\LGorb{\mathcal{LG}_{\mathrm{orb}}}
\def\LGeq{\mathcal{LG}_{\mathrm{eq}}}
\newcommand{\hmf}{\operatorname{hmf}}
\newcommand{\HMF}{\operatorname{HMF}}
\newcommand{\ev}{\operatorname{ev}}
\newcommand{\eval}{\operatorname{eval}}
\newcommand{\tev}{\widetilde{\operatorname{ev}}}
\newcommand{\coev}{\operatorname{coev}}
\newcommand{\tcoev}{\widetilde{\operatorname{coev}}}
\def\lra{\longrightarrow}
\def\lmt{\longmapsto}
\DeclareMathOperator{\tr}{tr}
\DeclareMathOperator{\str}{str}
\DeclareMathOperator{\Jac}{Jac}
\def\Re{R^{\operatorname{e}}}
\DeclareMathOperator{\Res}{Res}
\newcommand*{\longhookrightarrow}{\ensuremath{\lhook\joinrel\relbar\joinrel\rightarrow}}
\newcommand*{\twoheadlongrightarrow}{\ensuremath{\relbar\joinrel\twoheadrightarrow}}
\newcommand{\Ga}[1]{\Gamma_{\hspace{-2pt}#1}}
\newcommand{\HIA}{\Hom(I,A)}
\newcommand{\ZA}{Z_A(\Hom(I,A))}
\newcommand{\gZA}{\!Z_A^\gamma(\Hom(I,A))}
\newcommand{\gA}{{}_{\gamma_A}A}
\newcommand{\Aginv}{A_{\gamma_A^{-1}}}
\newcommand{\picc}{\pi^{(\text{c,c})}_A}
\newcommand{\pirr}{\pi^{\text{RR}}_A}
\newcommand{\tpirr}{{\widetilde\pi}^{\text{RR}}} 
\newcommand{\im}{\operatorname{im}}
\DeclareMathOperator*{\eq}{=}
\DeclareMathOperator*{\congscript}{\cong}
\newcommand{\specflow}{\mathcal U_{-\frac{1}{2},-\frac{1}{2}}}
\newcommand{\Hil}{\mathcal{H}}
\newcommand{\Hpcc}{\mathcal{H}'_{\text{(c,c)}}}
\newcommand{\Hprr}{\mathcal{H}'_{\text{RR}}}
\newcommand{\Hcc}{\mathcal{H}_{\text{(c,c)}}^A}
\newcommand{\Hrr}{\mathcal{H}_{\text{RR}}^A}
\newcommand{\HccnoA}{\mathcal{H}_{\text{(c,c)}}}
\newcommand{\HrrnoA}{\mathcal{H}_{\text{RR}}}
\newcommand{\Hrrbo}{\Hom_A(X,{}_{\gamma_A}X)}
\newcommand{\buboorb}{\beta_X^{\text{orb}}}
\newcommand{\bobuorb}{\beta^X_{\text{orb}}} 
\def\Cong{C_g}
\def\Centg{N_g}                 
\def\alphaKK{\alpha^{\{K\}}}
\def\alphagKg{\alpha_g^{K(g)}}
\newcommand{\AGC}{A_G^c}

\newcommand{\Bord}{\operatorname{Bord}}
\newcommand{\Bordor}{\operatorname{Bord}_{n}^{\mathrm{or}}}
\newcommand{\Borddef}{\operatorname{Bord}^{\mathrm{def}}}
\newcommand{\Borden}[1] {\widehat{\operatorname{Bord}}{}_{#1}}
\newcommand{\Borddc}{\widehat{\operatorname{Bord}}{}^{\mathrm{def}}_3(\DC)}
\newcommand{\Borddefn}[1] {\operatorname{Bord}^{\mathrm{def}}_{#1}}
\newcommand{\Bordribn}[1] {\operatorname{Bord}^{\mathrm{rib}}_{#1}}
\newcommand{\Bordriben}[1] {\widehat{\operatorname{Bord}}{}^{\mathrm{rib}}_{#1}}
\newcommand{\Borddefen}[1] {\widehat{\operatorname{Bord}}{}^{\mathrm{def}}_{#1}}
\newcommand{\Bordoc}[1] {\operatorname{Bord}^{\mathrm{oc}}_{#1}}
\newcommand{\Bords}{\operatorname{Bord}_{3}^{\APLstar}}
\newcommand{\Sphere}{\operatorname{Sphere}^{\mathrm{def}}}
\newcommand{\Nbh}{\operatorname{\mathcal{N}}}
\newcommand{\Cube}{\operatorname{Cube}}
\newcommand{\Cubed}{\operatorname{Cube}_{3}^{\mathrm{def}}(\mathds{D})}
\newcommand{\Fradj}{\operatorname{\mathcal C_{\mathds D}^{adj}}}
\newcommand{\Bordstrat}{\operatorname{Bord}^{\mathrm{strat}}}
\newcommand{\Bordstratn}[1]  {\operatorname{Bord}^{\mathrm{strat}}_{#1}}
\newcommand{\Borddlong}{\operatorname{Bord}_{3}^{\mathrm{def}}(D_3,D_2,D_1)_{s,t,f}}
\newcommand{\Bordd[1]}{\operatorname{Bord}_{#1}^{\mathrm{def}}(\mathds{D})}
\newcommand{\Strator}{\operatorname{Strat}_{n}^{\mathrm{or}}}
\newcommand{\G}{\mathcal{G}}
\newcommand{\tz}{\mathcal T_\zz}
\newcommand{\dz}{\mathcal D_\zz}
\newcommand{\tzp}{\mathcal T_{\mathcal Z'}}
\newcommand{\Data}{\mathds{D}}
\newcommand{\Obj}{\mathrm{Obj}}
\newcommand{\zz}{\mathcal{Z}}
\newcommand{\zzc}{\mathcal{Z}^{\mathcal C}}
\newcommand{\zrt}{\mathcal{Z}^{\text{RT,}\mathcal C}}
\newcommand{\zrtdef}{\mathcal{Z}_{\mathcal C}}
\newcommand{\Fdp}{\operatorname{\mathcal F}_{\textrm{d}}^{\textrm{p}}}
\newcommand{\zzd}{\mathcal{Z}^{\mathrm{def}}}
\newcommand{\zztriv}{\mathcal{Z}^{\mathrm{triv}}} 
\newcommand{\zzAtriv}{\mathcal{Z}_{\mathrm{triv}}^{\Cat{A}}} 
\newcommand{\euc}{\odot}
\newcommand{\euctwo}{\odot_{\geq 2}}
\newcommand{\ieuc}{\iota^\euc}
\newcommand{\peuc}{\pi^\euc}

\newcommand{\zzTVA}{\mathcal{Z}^{TV}_{\Cat{A}}} 
\newcommand{\zzRTC}{\mathcal{Z}^{RT}_{\Cat{C}}}
\newcommand{\tzztriv}{\mathcal{T}_{{\mathcal{Z}^{triv}}}}

\newcommand{\tzGamma}{\mathcal{T}_{{\mathcal{Z}^{\Gamma}}}}
\newcommand{\unit}{\operatorname{\mathbf{1}}}  
\newcommand{\bigslant}[2]{{\raisebox{.0em}{$#1$}\left/\raisebox{-.2em}{$#2$}\right.}}
\newcommand{\Cubedp}{\operatorname{Cube}_{3}^{\mathrm{def}}(\mathds{D}')}
\newcommand{\Borddp}{\operatorname{Bord}_{3}^{\mathrm{def}}(\mathds{D}')}
\newcommand{\Vect}{\operatorname{vect}}
\newcommand{\Vectk}{\operatorname{vect}}
\newcommand{\eps}{\varepsilon}
\newcommand{\al}{\alpha}
\newcommand{\alb}{\bar{\alpha}}
\newcommand{\T}{\mathcal{T}}
\newcommand{\Ss}{\mathcal{S}}
\newcommand{\X}{\mathcal{X}}
\newcommand{\Y}{\mathcal{Y}}
\newcommand{\sta}{\boxempty}
\newcommand{\fus}{\otimes}
\newcommand{\sd}{^{\star}}
\newcommand{\dagg}{^{\dagger}}
\newcommand{\hash}{^{\#}}
\newcommand{\Set}{\mathrm{Set}}
\newcommand{\Ball}{\mathrm{Ball}}
\newcommand{\Sph}{\mathsf{Sph}}
\newcommand{\fork}{\pitchfork }
\newcommand{\Cat}[1]         {\operatorname{\mathcal{#1}}}
\newcommand{\oprev}[1]         {\operatorname{\mathcal{#1}}^{\mathrm{op,rev}}}
\newcommand{\Catpre}[1]         {\operatorname{\mathcal{#1}^{\mathrm{pre}}}}

\newcommand{\dX}{{}^\dagger\hspace{-1.8pt}X}
\newcommand{\dA}{{}^\dagger\hspace{-1.8pt}A}
\newcommand{\dsX}{{}^\dagger\hspace{-1.8pt}\mathcal{X}}
\newcommand{\deqX}{{}^\star\hspace{-1.8pt}X} 
\newcommand{\dseqX}{{}^\star\hspace{-1.8pt}\mathcal{X}} 
\newcommand{\dY}{{}^\dagger\hspace{-0.3pt}Y}
\newcommand{\dphi}{{}^\dagger\hspace{-0.9pt}\phi}
\newcommand{\dPhi}{{}^\dagger\hspace{-0.9pt}\Phi}
\newcommand{\FEnd}{\mathcal{E}\hspace*{-.7pt}nd}

\newcommand{\SSFR}{$\Delta$-separable symmetric Frobenius algebra}
\newcommand{\SSFRs}{$\Delta$-separable symmetric Frobenius algebras}
\newcommand{\ATAA}{{}_{A}T_{AA}}
\newcommand{\ATAAi}{{}_{A}T_{A_{1}A_{2}}}
\newcommand{\all}{\alpha}
\newcommand{\allb}{{\bar\alpha}}

\newcommand{\opp}             {{\mathrm{op}}} 
\newcommand{\Alg}        {\operatorname{\mathsf{Alg}}}
\newcommand{\Frob}        {\operatorname{\mathsf{Frob}}}
\newcommand{\Lincat}        {\operatorname{\mathsf{Cat}^{ses}}}
\newcommand{\Deftqft}{\operatorname{TQFT}^{\mathrm{def}}}
\newcommand{\KVvect}   {\operatorname{KV-2Vect}}
\newcommand{\CYvect}   {\operatorname{CY-2Vect}}
\newcommand{\evx}[1]   {\operatorname{\mathsf{ev}}_{#1}}
\newcommand{\coevx}[1]   {\operatorname{\mathsf{coev}}_{#1}}
\newcommand{\evp}[1]   {\operatorname{\mathsf{ev}}_{#1}^{\prime}}
\newcommand{\coevp}[1]   {\operatorname{\mathsf{coev}}_{#1}^{\prime}}
\newcommand{\evc}[1]   {\operatorname{\mathsf{c-ev}}_{#1}}
\newcommand{\coevc}[1]   {\operatorname{\mathsf{c-coev}}_{#1}}
\newcommand{\evpc}[1]   {\operatorname{\mathsf{c-ev}}_{#1}^{\prime}}
\newcommand{\coevpc}[1]   {\operatorname{\mathsf{c-coev}}_{#1}^{\prime}}
\def\Fun              {{\mathsf{Fun}}}
\def\Funbilin              {{\mathsf{Fun}^{\mathsf{bilin}}}}
\newcommand{\rrr}[1]{{\color{red}{#1}}}
\newcommand{\rrR}[1]{{\color{red3}{#1}}}
\newcommand{\bbb}[1]{{\color{blue}{#1}}}
\definecolor{DarkViolet} {rgb}{0.580392,0.000000,0.827450}
\newcommand{\vio}[1]{{\color{DarkViolet}{#1}}}
\newcommand{\green}[1]{{\color{green}{#1}}}

\newcommand\nxt{\noindent\raisebox{.08em}{\rule{.44em}{.44em}}\hspace{.4em}}
\newcommand\arxiv[2]      {\href{http://arXiv.org/abs/#1}{#2}}
\newcommand\doi[2]        {\href{http://dx.doi.org/#1}{#2}}
\newcommand\httpurl[2]    {\href{http://#1}{#2}}

\renewcommand{\labelenumi}{(\roman{enumi})}

\allowdisplaybreaks

\deffootnote[1em]{1em}{1em}{\textsuperscript{\thefootnotemark}}

\theoremstyle{definition} 
\newtheorem{definition}{Definition}
\newtheorem{proposition}[definition]{Proposition}
\newtheorem{theorem}[definition]{Theorem}

\newtheorem{lemma}[definition]{Lemma}
\newtheorem{corollary}[definition]{Corollary}
\newtheorem{remark}[definition]{Remark}

\newtheorem{example}[definition]{Example}

\newtheorem{notation}[definition]{Notation}

\numberwithin{equation}{section}
\numberwithin{definition}{section}
\numberwithin{figure}{section}

\newcommand\void[1]{}

\begin{document}

\title{Orbifolds\\ 
	of Reshetikhin--Turaev TQFTs}

\author{%
	\!\!\!\!\!\!\!Nils Carqueville$^*$ \quad
	Ingo Runkel$^\#$ \quad
	Gregor Schaumann$^\vee$%
	\\[0.5cm]
	\hspace{-1.8cm}  \normalsize{\texttt{\href{mailto:nils.carqueville@univie.ac.at}{nils.carqueville@univie.ac.at}}} \\  %
	\hspace{-1.8cm}  \normalsize{\texttt{\href{mailto:ingo.runkel@uni-hamburg.de}{ingo.runkel@uni-hamburg.de}}} \\
	\hspace{-1.8cm}  \normalsize{\texttt{\href{mailto:gregor.schaumann@uni-wuerzburg.de}{gregor.schaumann@uni-wuerzburg.de}}}\\[0.1cm]
	\hspace{-1.2cm} {\normalsize\slshape $^*$Fakult\"at f\"ur Physik, Universit\"at Wien, Austria}\\[-0.1cm]
	\hspace{-1.2cm} {\normalsize\slshape $^\#$Fachbereich Mathematik, Universit\"{a}t Hamburg, Germany}\\[-0.1cm]
	\hspace{-1.2cm} {\normalsize\slshape $^\vee$Institut f\"{u}r Mathematik, Universit\"{a}t W\"{u}rzburg, Germany
	}\\[-0.1cm]
}

\date{}
\maketitle

\begin{abstract}
We construct three classes of generalised orbifolds of Reshetikhin--Turaev theory for a modular tensor category~$\Cat{C}$, using the language of defect TQFT from \cite{CRS1}: 
(i) 
spherical fusion categories give orbifolds for the ``trivial'' defect TQFT associated to $\Vectk$, 
(ii) 
$G$-crossed extensions of~$\Cat{C}$ give group orbifolds for any finite group~$G$, and 
(iii) 
we construct orbifolds from commutative \SSFRs\ in~$\Cat{C}$. 
We also explain how the Turaev--Viro state sum construction fits into our framework by proving that it is isomorphic to the orbifold of case~(i). 
Moreover, we treat the cases (ii) and (iii) in the more general setting of ribbon tensor categories. 
For case (ii) we show how Morita equivalence leads to isomorphic orbifolds, and we discuss Tambara--Yamagami categories as particular examples.

\medskip 

{\footnotesize 2020 Mathematics Subject Classification: 57K16, 18M20, 57R56}
\end{abstract}

\newpage

\tableofcontents

\section{Introduction  and summary}
\label{sec:intro}

For any modular tensor category~$\Cat{C}$, Reshetikhin and Turaev \cite{retu2, tur} constructed a 3-dimensional topological quantum field theory 
	 $\zrt \colon  \Borden{3} \longrightarrow \Vectk$. This construction 
is intimately related 
to 
	 the connection between the representation theory of quantum groups and knot theory \cite{tur}, 	
	 and rational conformal field theory \cite{FRS25}. 
The 
	symmetric monoidal 
functor $\zrt$ 
	 acts on diffeomorphism classes of bordisms 
with embedded ribbons that are labelled with data 
	 from~$\Cat{C}$, hence it assigns topological invariants to ribbon embeddings into 3-manifolds. 
In \cite{CRS2} we extended this by constructing a \textsl{Reshetikhin--Turaev defect TQFT} $\zzc \colon \Borddefen{3}(\D^{\Cat{C}}) \to \Vectk$ that assigns invariants to 
	equivalence classes of 
stratified bordisms whose 3-, 2- and 1-strata are respectively labelled by~$\Cat{C}$, certain Frobenius algebras in~$\Cat{C}$ and their cyclic modules.
	 The original functor~$\zrt$ is isomorphic to a restriction of~$\zzc$, as ribbons can be modelled by a combination of 1- and 2-strata, cf.\ \cite[Rem.\,5.9]{CRS2}. 
	 For an $n$-dimensional defect TQFT~$\zz$ with $n \in \{2,3\}$, the labelled strata (or ``defects'', a term used in physics to refer to regions in spacetime with certain properties that distinguish them from their surroundings) of codimension~$j$ are known to correspond to $j$-cells in the $n$-category associated to~$\zz$ \cite{dkr1107.0495, CMS}; this is also expected for $n\geqslant 4$. 
Defects in Reshetikhin--Turaev theory had previously been studied in \cite{ks1012.0911, fsv1203.4568, CMS}. 

In the present paper we construct orbifolds of Reshetikhin--Turaev TQFTs. 
Inspired by earlier work on rational 
	 conformal field theory 
\cite{ffrs0909.5013}, a 
	 (generalised) orbifold
theory was developed for 2-dimensional TQFTs in \cite{cr1210.6363}, which we then further generalised to arbitrary dimensions in \cite{CRS1}: 
Given an $n$-dimensional defect TQFT~$\zz$ (i.\,e.\ a symmetric monoidal functor on decorated stratified $n$-dimensional bordisms, cf.\ Section~\ref{sec:resh-tura-defect}) 
and an ``orbifold datum''~$\A$ (consisting of special labels for $j$-strata for all $j\in \{0,1,\dots,n\}$, cf.\ Section~\ref{sec:orbif-defect-tqfts}), 
the generalised orbifold construction produces a closed TQFT $\zz_\A \colon \Bord_n \to \Vect$ roughly as follows. 
On any given bordism, $\zz_\A$ acts by choosing a triangulation, decorating its dual stratification with the data~$\A$, evaluating with~$\zz$, and then applying a certain projector. 
	 The defining properties of orbifold data~$\A$ are such that $\zz_\A$ is independent of the choice of triangulation. 
	
In dimension $n=2$, orbifold data turn out to be certain Frobenius algebras in the 2-category associated to~$\zz$, and both 
	 state sum models \cite{dkr1107.0495} and ordinary group orbifolds \cite{cr1210.6363, BCP2} are examples of orbifold TQFTs~$\zz_\A$. 
	Here by ``group orbifolds'' we mean TQFTs $\zz_\A$, where~$\A$ is obtained from an action of a finite symmetry group on~$\zz$. (This is also the origin of our usage of the term ``orbifold'' TQFT: If $\zz^X$ is a TQFT obtained from a sigma model with target manifold~$X$ that comes with a certain action of a finite group~$G$, then there is an orbifold datum~$\A_G$ such that $(\zz^X)_{\A_G} \cong \zz^{X/\!\!/G}$ where $\zz^{X/\!\!/G}$ is a TQFT associated to the orbifold (in the geometric sense) $X/\!\!/G$.) 
There are also interesting 2-dimensional 
	 orbifold TQFTs 
that go beyond these classes of examples, cf.\ \cite{CRCR, nrc1509.08069, OEReck}. 

For general 3-dimensional defect TQFTs we worked out the defining conditions on orbifold data in \cite{CRS1}. 
In the present paper we focus on Reshetikhin--Turaev defect TQFTs~$\zzc$ and reformulate their orbifold conditions internally to the modular tensor 
	 category~$\Cat{C}$.  
This is achieved in Proposition~\ref{prop:internal-orb} which is the key technical result in our paper and is used to prove the two main theorems below.

\smallskip 

Our first main result 
	 (stated 
as Proposition~\ref{prop:orbidata} and Theorem~\ref{thm:ZTVZtrivorb}) concerns orbifolds of the ``trivial'' Reshetikhin--Turaev defect TQFT $\zz^{\textrm{triv}} := \zz^{\Vect}$, i.\,e.\ when the modular tensor category is simply $\Vect$. 
Recall that (as we review in Section~\ref{subsec:TVire}) 
from every spherical fusion category~$\Cat{S}$ one can construct a 3-dimensional state sum TQFT called Turaev--Viro theory 
	 $\zz^{\text{TV},\Cat{S}}$
\cite{TVmodel, bwTV1}. 

\medskip 

\noindent
\textbf{Theorem A. }
For every spherical fusion category~$\Cat{S}$ there is an orbifold datum~$\Aca$ for $\zz^{\textrm{triv}}$ such that 
$
\zz^{\text{triv}}_{\Aca} 
\cong 
\zz^{\text{TV},\Cat{S}}
$. 

\medskip 

	 This result, appearing as Theorem~\ref{thm:ZTVZtrivorb} in the main text, vindicates 
the slogan ``state sum models are orbifolds of the trivial theory'' 
in three dimensions. 
This can in fact be seen as a special case of the slogan ``3-dimensional orbifold data are spherical fusion categories internal to 3-categories with duals'', cf.\ Remark~\ref{rem:SFCinternal}. 

\smallskip 

Our second main result concerns group extensions of tensor categories.%
\footnote{A more geometric approach to group orbifolds of Reshetikhin--Turaev TQFTs and more generally of 3-2-1-extended TQFTs has been given in \cite{SchweigertWoike201802}.}
Recall that an extension of a tensor  category $\Cat{C}$ by a finite group $G$ is a tensor category~$\Cat{B}$ which is graded by~$G$ with neutral component $\Cat{B}_{1}=\Cat{C}$. 
To formulate
	 our result 
we note that the nondegeneracy condition on a modular tensor category~$\Cat{C}$ is not needed to define orbifold data~$\A$ for~$\zzc$, and hence one can speak of orbifold data in arbitrary ribbon categories~$\B$ 
(see Section~\ref{subsec:SODribbon} for details). 
In Section~\ref{sec:Gextensions} we prove 
	 Theorem~\ref{thm:Gcrossed} (see e.\,g.\ \cite{TuraevBook2010} for the notion of ``ribbon crossed $G$-category''), which we paraphrase as follows: 

\medskip 

\noindent
\textbf{Theorem B. }
Let~$\B$ be a ribbon fusion category and let~$G$ be a finite group. 
Every 
	 ribbon crossed $G$-category $\widehat{\B} = \bigoplus_{g\in G} \B_g$, such that the component~$\B_1$ labelled by the unit $1\in G$ satisfies $\B_1 = \B$, gives
rise to an orbifold datum for~$\B$. 

\medskip 

We will be particularly interested in the situation where $\B=\B_1$ is additionally a
full ribbon subcategory of a modular tensor category $\Cat{C}$, in which case an extension~$\widehat{\B}$ provides orbifold data in~$\Cat{C}$.
A special case of this is $\B = \Cat{C}$ and where $\widehat{\B} = \Cat{C}_G^\times$ is a $G$-crossed extension. 
An important source of examples for $G=\Z_2$ are Tambara--Yamagami categories, which are $\Z_2$-extensions of $H$-graded vector spaces for a finite abelian group~$H$. 
This is explained in Section~\ref{sec:Gextensions}, where we also discuss orbifold data for the modular tensor categories associated to the affine Lie algebras $\widehat{\mathfrak{sl}}(2)_k$. 
Moreover, we prove a version of Theorem~B that holds for certain non-semisimple ribbon categories~$\B$, cf.\ Remark~\ref{rem:GextNonSS}. 

Taken together, Theorems~A and~B say that orbifolds unify state sum models and group actions in three dimensions.%
\footnote{The unification of state sum models and group orbifolds in two dimensions is a corollary of \cite{dkr1107.0495, cr1210.6363, BCP2}.}

\smallskip 

The orbifold data in Theorem~B depend on certain choices, which are however all related by Morita equivalences that in turn lead to isomorphic orbifold TQFTs (when~$\B$ is a subcategory of a modular category~$\Cat{C}$), as we explain in Section~\ref{subsec:MoritaOD}.

	\smallskip
	
	As a third source of orbifold data for the Reshetikhin--Turaev defect TQFT~$\zzc$ we identify commutative $\Delta$-separable Frobenius algebras in~$\Cat{C}$ in Section~\ref{sec:commSSFR}. 
	
\medskip

For the whole paper we fix an algebraically closed field~$\Bbbk$ of characteristic zero, and we write the symmetric monoidal category of finite-dimensional $\Bbbk$-vector spaces simply as $\Vectk$.

\vspace{-0.1cm}

\subsubsection*{Acknowledgements} 

We would like to thank 
	Ehud Meir, 
	Daniel Scherl and 
	Michael M\"uger 
for helpful discussions.
The work of N.\,C.~is partially supported by a grant from the Simons Foundation. 
N.\,C.~and G.\,S.~are partially supported by the stand-alone project P\,27513-N27 of the Austrian Science Fund. 
The authors acknowledge support by the Research Training Group 1670 of the German Research Foundation.

\section{TQFTs with defects and orbifolds}
\label{sec:backgground}
 
In this section we briefly review the general notions of 3-dimensional defect TQFTs and their orbifolds from \cite{CMS, CRS1}, and the extension of Reshetikhin--Turaev theory to a defect TQFT with surface defects from \cite{CRS2}. 

\medskip 

We start by recalling three types of tensor categories over~$\Bbbk$ (see e.\,g.~\cite{tur, EGNO-book} for details).  
A \textsl{spherical fusion category}~$\Cat{S}$ is a semisimple $\Bbbk$-linear pivotal monoidal category with finitely many isomorphism classes of simple objects $i\in\Cat{S}$, such that left and right traces coincide and $\End_{\Cat{S}}(\one) = \Bbbk$. 
Pivotality implies that~$\Cat{S}$ has coherently isomorphic left and right duals, and sphericality implies that the associated left and right dimensions are equal. 
The \textsl{global dimension} of~$\Cat{S}$ is the sum $\dim{\Cat{S}} = \sum_{i} \dim(i)^2$ over a choice of representatives~$i$ of the isomorphism classes of simple objects in~$\Cat{S}$. 
Since $\text{char}(\Bbbk) = 0$ by assumption, we have that $\dim{\Cat{S}} \neq 0$ \cite{ENO1}. 
A \textsl{ribbon fusion category} is  a braided spherical fusion category.
A \textsl{modular tensor category} is a ribbon fusion category with nondegenerate braiding.

\subsection{Reshetikhin--Turaev defect TQFTs}
\label{sec:resh-tura-defect}

Let $\Cat{C}$ be a modular tensor category over~$\Bbbk$. 
There is an associated (typically anomalous) 3-dimensional TQFT:
\begin{equation}
  \label{eq:3d-tqft}
  \zrt \colon  \Borden{3} \longrightarrow \Vectk \, ,
\end{equation}
called Reshitikhin-Turaev theory. 
Here $\Borden{3}$ is a certain extension of the symmetric monoidal category $\Bord_{3}$ of 3-dimensional bordisms, which is needed to deal with the anomaly. 
For all details we refer to \cite{retu2,tur}; the constructions in the present paper do not require dealing with the anomaly in an explicit way.

In \cite{CRS2} we constructed surface and line defects for $\zrt$ from \SSFRs{} and their (cyclic) modules.  
We briefly recall the notion of a 3-dimensional defect TQFT from \cite{CMS, CRS1}, and the extension~$\zzc$ of~\eqref{eq:3d-tqft} to a full defect TQFT from \cite{CRS2}.

\bigskip
\noindent
\textbf{\textsl{Conventions. }}
We adopt the conventions from \cite{CRS2}, in particular we read string diagrams for~$\Cat{C}$ from bottom to top. 
For instance the braiding $c_{X,Y} \colon X \otimes Y \to Y \otimes X$ and its inverse are written as 
\be
c_{X,Y}
= 
\begin{tikzpicture}[very thick,scale=0.5,color=blue!50!black, baseline]
\draw[color=blue!50!black] (-1,-1) node[below] (A1) {{\scriptsize$X$}};
\draw[color=blue!50!black] (1,-1) node[below] (A2) {{\scriptsize$Y$}};
\draw[color=blue!50!black] (-1,1) node[above] (B1) {{\scriptsize$Y$}};
\draw[color=blue!50!black] (1,1) node[above] (B2) {{\scriptsize$X$}};
\draw[color=blue!50!black] (A2) -- (B1);
\draw[color=white, line width=4pt] (A1) -- (B2);
\draw[color=blue!50!black] (A1) -- (B2);
\end{tikzpicture} 
\, , \quad
c^{-1}_{X,Y}
= 
\begin{tikzpicture}[very thick,scale=0.5,color=blue!50!black, baseline]
\draw[color=blue!50!black] (-1,-1) node[below] (A1) {{\scriptsize$Y$}};
\draw[color=blue!50!black] (1,-1) node[below] (A2) {{\scriptsize$X$}};
\draw[color=blue!50!black] (-1,1) node[above] (B1) {{\scriptsize$X$}};
\draw[color=blue!50!black] (1,1) node[above] (B2) {{\scriptsize$Y$}};
\draw[color=blue!50!black] (A1) -- (B2);
\draw[color=white, line width=4pt] (A2) -- (B1);
\draw[color=blue!50!black] (A2) -- (B1);
\end{tikzpicture} 
\, , 
\ee
and we denote the twist isomorphism on an object $U \in \Cat{C}$ by~$\theta_U$.

An algebra in $\Cat{C}$ is an object $A\in\Cat{C}$ together with morphisms $\mu\colon A \otimes A \to A$ and $\eta \colon \one \to A$ satisfying associativity and unit conditions. 
If $(A_1,\mu_1,\eta_1)$ and $(A_2,\mu_2,\eta_2)$ are algebras in~$\mathcal C$ then the tensor product $A_1 \otimes A_2$ also carries an algebra structure; our convention is that $A_1\otimes A_2$ has the multiplication
\be
\mu_{A_1 \otimes A_2} = ( \mu_1 \otimes \mu_2 ) \circ ( 1_{A_1} \otimes c_{A_2,A_1} \otimes 1_{A_2}) 
= 
\begin{tikzpicture}[very thick,scale=0.75,color=blue!50!black, baseline]
\draw[color=green!50!black] (-1,-1) node[below] (A1) {{\scriptsize$A_1$}};
\draw[color=green!50!black] (0,-1) node[below] (A2) {{\scriptsize$A_2$}};
\draw[color=green!50!black] (1,-1) node[below] (A1r) {{\scriptsize$A_1$}};
\draw[color=green!50!black] (2,-1) node[below] (A2r) {{\scriptsize$A_2$}};
\draw[color=green!50!black] (0,1) node[above] (A1up) {{\scriptsize$A_1$}};
\draw[color=green!50!black] (1,1) node[above] (A2up) {{\scriptsize$A_2$}};
\fill[color=green!50!black] (0,0) circle (2.5pt) node (mult1) {};
\fill[color=green!50!black] (1,0) circle (2.5pt) node (mult2) {};
\draw[color=green!50!black] (A1) -- (0,0);
\draw[color=green!50!black] (A1r) -- (0,0);
\draw[color=white, line width=4pt] (A2) -- (1,0);
\draw[color=green!50!black] (A2) -- (1,0);
\draw[color=green!50!black] (A2r) -- (1,0);
\fill[color=green!50!black] (1,0) circle (2.5pt) node (mult2) {};
\draw[color=green!50!black] (0,0) -- (A1up);
\draw[color=green!50!black] (1,0) -- (A2up);
\end{tikzpicture} 
\ee
and unit $\eta_{A_1 \otimes A_2} = \eta_1 \otimes \eta_2$.

Let $A,B$ be algebras in~$\mathcal C$, 
let~$M$ be a right $A$-module and~$N$ a right $B$-module. 
From \cite[Expl.\,2.13(ii)]{CRS2} we obtain that $M \otimes N$ is an $(A \otimes B)$-module with component actions 
\be
\begin{tikzpicture}[very thick,scale=0.75,color=blue!50!black, baseline,xscale=-1]
\draw[color=green!50!black] (-1,-1) node[below] (A1) {};
\draw[color=green!50!black] (-0.5,-1) node[below] (A2) {};
\draw (0,-1) -- (0,1); 
\draw[color=green!50!black] (-1,-1) .. controls +(0,0.5) and +(-0.5,-0.5) .. (0,0.5);
\draw (0,-1) node[below] (M) {{\scriptsize$M\otimes N$}};
\draw[color=green!50!black] (-1,-1) node[below] (A1) {{\scriptsize$A$}};
\fill[color=black] (0,0.5) circle (2.9pt) node[right] (meet2) {};
\fill[color=black] (-0.25,0.6) circle (0pt) node (M) {{\scriptsize$1$}};
\end{tikzpicture} 
:= 
\begin{tikzpicture}[very thick,scale=0.75,color=blue!50!black, baseline,xscale=-1]
\draw[color=green!50!black] (-1,-1) node[below] (A1) {};
\draw[color=green!50!black] (-0.5,-1) node[below] (A2) {};
\draw[color=green!50!black] (-1,-1) .. controls +(0,0.5) and +(-0.5,-0.5) .. (0.75,0.5);
\draw[color=white, line width=4pt] (0,-1) -- (0,1);
\draw (0,-1) -- (0,1); 
\draw (0.75,-1) -- (0.75,1); 
\draw (0,-1) node[below] (M) {{\scriptsize$N$}};
\draw (0.75,-1) node[below] (M) {{\scriptsize$M$}};
\draw[color=green!50!black] (-1,-1) node[below] (A1) {{\scriptsize$A$}};
\fill[color=black] (0.75,0.5) circle (2.9pt) node[right] (meet2) {};
\end{tikzpicture} 
\, , \quad
\begin{tikzpicture}[very thick,scale=0.75,color=blue!50!black, baseline,xscale=-1]
\draw[color=green!50!black] (-1,-1) node[below] (A1) {};
\draw[color=green!50!black] (-0.5,-1) node[below] (A2) {};
\draw (0,-1) -- (0,1); 
\draw[color=green!50!black] (-1,-1) .. controls +(0,0.5) and +(-0.5,-0.5) .. (0,0.5);
\draw (0,-1) node[below] (M) {{\scriptsize$M\otimes N$}};
\draw[color=green!50!black] (-1,-1) node[below] (A1) {{\scriptsize$B$}};
\fill[color=black] (0,0.5) circle (2.9pt) node[right] (meet2) {};
\fill[color=black] (-0.25,0.6) circle (0pt) node (M) {{\scriptsize$2$}};
\end{tikzpicture} 
:= 
\label{eq:actions-MN}
\begin{tikzpicture}[very thick,scale=0.75,color=blue!50!black, baseline,xscale=-1]
\draw[color=green!50!black] (-1,-1) node[below] (A1) {};
\draw[color=green!50!black] (-0.5,-1) node[below] (A2) {};
\draw (0,-1) -- (0,1); 
\draw (0.75,-1) -- (0.75,1); 
\draw[color=green!50!black] (-1,-1) .. controls +(0,0.5) and +(-0.5,-0.5) .. (0,0.5);
\draw (0,-1) node[below] (M) {{\scriptsize$N$}};
\draw (0.75,-1) node[below] (M) {{\scriptsize$M$}};
\draw[color=green!50!black] (-1,-1) node[below] (A1) {{\scriptsize$B$}};
\fill[color=black] (0,0.5) circle (2.9pt) node[right] (meet2) {};
\end{tikzpicture} 
\, .
\ee
Analogously, $M  \otimes N$ becomes a $(B \otimes A)$-module with the actions 
\be
\label{eq:actions-M-rev}
\begin{tikzpicture}[very thick,scale=0.75,color=blue!50!black, baseline,xscale=-1]
\draw[color=green!50!black] (-1,-1) node[below] (A1) {};
\draw[color=green!50!black] (-0.5,-1) node[below] (A2) {};
\draw (0,-1) -- (0,1); 
\draw (0.75,-1) -- (0.75,1); 
\draw[color=white, line width=4pt] (-1,-1) .. controls +(0,0.5) and +(-0.5,-0.5) .. (0.75,0.5);
\draw[color=green!50!black] (-1,-1) .. controls +(0,0.5) and +(-0.5,-0.5) .. (0.75,0.5);
\draw (0,-1) node[below] (M) {{\scriptsize$N$}};
\draw (0.75,-1) node[below] (M) {{\scriptsize$M$}};
\draw[color=green!50!black] (-1,-1) node[below] (A1) {{\scriptsize$A$}};
\fill[color=black] (0.75,0.5) circle (2.9pt) node[right] (meet2) {};
\end{tikzpicture} 
\, , \quad
\begin{tikzpicture}[very thick,scale=0.75,color=blue!50!black, baseline,xscale=-1]
\draw[color=green!50!black] (-1,-1) node[below] (A1) {};
\draw[color=green!50!black] (-0.5,-1) node[below] (A2) {};
\draw (0,-1) -- (0,1); 
\draw (0.75,-1) -- (0.75,1); 
\draw[color=green!50!black] (-1,-1) .. controls +(0,0.5) and +(-0.5,-0.5) .. (0,0.5);
\draw (0,-1) node[below] (M) {{\scriptsize$N$}};
\draw (0.75,-1) node[below] (M) {{\scriptsize$M$}};
\draw[color=green!50!black] (-1,-1) node[below] (A1) {{\scriptsize$B$}};
\fill[color=black] (0,0.5) circle (2.9pt) node[right] (meet2) {};
\end{tikzpicture} 
\, .
\ee

\bigskip 
\noindent
\textbf{\textsl{3-dimensional defect TQFT. }}
We recall from \cite{CMS, CRS1} that a  3-dimensional defect TQFT is a symmetric monoidal functor
\be
\label{eq:def-TQFT}
\zz \colon \Borddefn{3}(\D) \longrightarrow \Vectk \, ,
\ee
where the source category consists of stratified and decorated bordisms with orientations. 
For details we refer to \cite{CRS1}, but the main ingredients are as follows: 
A bordism $N\colon \Sigma \rightarrow \Sigma'$ between to stratified surfaces $\Sigma,\Sigma'$ has 3-, 2- and 1-strata in the interior, while on the boundary also 0-strata are allowed. 
The possible decorations for the strata are specified  by a set of 3-dimensional \textsl{defect data} $\D$  which is 
	a 
tuple 
\be
	\D = \big(D_3,D_2,D_1;s,t,j\big) \, .
\ee
Here $D_i$, $i\in \{ 1,2,3 \}$, are sets whose elements label the $i$-dimensional strata of bordisms; the case $i=0$ can naturally be added by a universal construction, see Remark~\ref{rem:completions}.
The \textsl{source, target} and \textsl{junction maps} $s,t \colon D_2 \to D_3$ and $j \colon D_1 \to \text{(cyclic lists of elements of $D_2$)}$
specify the adjacency conditions for the decorated strata. This is best described in an example:
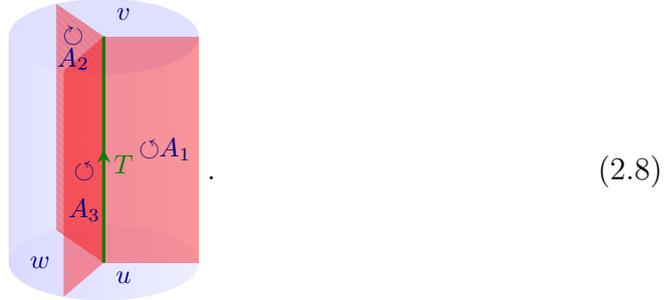
\be\label{eq:starlikecylinder}
\tikzzbox{\begin{tikzpicture}[very thick,scale=1.0,color=blue!50!black, baseline=-1.9cm]
	\fill [blue!15,
	opacity=0.5, 
	left color=blue!15, 
	right color=white] 
	(-1.25,0) -- (-1.25,-3) arc (180:360:1.25 and 0.5) -- (1.25,0) arc (0:180:1.25 and -0.5);
	\fill [blue!35,opacity=0.1] (-1.25,-3) arc (180:360:1.25 and 0.5) -- (1.25,-3) arc (0:180:1.25 and 0.5);
	\fill [blue!25,opacity=0.5] (-1.25,0) arc (180:360:1.25 and 0.5) -- (1.25,0) arc (0:180:1.25 and 0.5);
	%
	\fill [red,opacity=0.4] (0,0) -- (0,-3) -- (1.25,-3) -- (1.25,0);
	\fill [pattern=north west lines, opacity=0.3] (0,0) -- ($(0,0)+(120:1.25 and 0.5)$) -- ($(0,-3)+(120:1.25 and 0.5)$) -- (0,-3) -- (0,0);
	\fill [red,opacity=0.4] (0,0) -- ($(0,0)+(120:1.25 and 0.5)$) -- ($(0,-3)+(120:1.25 and 0.5)$) -- (0,-3) -- (0,0);
	\fill [red,opacity=0.4] (0,0) -- ($(0,0)+(245:1.25 and 0.5)$) -- ($(0,-3)+(245:1.25 and 0.5)$) -- (0,-3) -- (0,0);
	\draw[
	color=green!50!black, 
	>=stealth,
	decoration={markings, mark=at position 0.5 with {\arrow{>}},
	}, postaction={decorate}
	] 
	(0,-3) -- (0,0);
	\draw[line width=1] (0.6, -1.5) node[line width=0pt] (beta) {{\footnotesize  $\circlearrowleft$}};
	\draw[line width=1]  ($(0,-1.8)+(245:0.6 and 0)$)  node[line width=0pt] (beta) {{\footnotesize  $\circlearrowleft$}};
	\draw[line width=1]  ($(0,0)+(120:0.8 and 0)$)  node[line width=0pt] (beta) {{\footnotesize  $\circlearrowright$}};
	\draw[line width=1] (0.94, -1.5) node[line width=0pt] (beta) {{\footnotesize  $A_1$}};
	\draw[line width=1]  ($(0,-2.3)+(245:0.6 and 0)$)  node[line width=0pt] (beta) {{\footnotesize  $A_3$}};
	\draw[line width=1]  ($(0,-0.3)+(120:0.8 and 0)$)  node[line width=0pt] (beta) {{\footnotesize  $A_2$}};
	\draw[line width=1]  ($(0.4,-3.2)+(100:0.8 and 0)$)  node[line width=0pt] (beta) {{\footnotesize  $u$}};
	\draw[line width=1]  ($(0.4,+0.3)+(100:0.8 and 0)$)  node[line width=0pt] (beta) {{\footnotesize  $v$}};
	\draw[line width=1]  ($(-0.7,-3.0)+(100:0.8 and 0)$)  node[line width=0pt] (beta) {{\footnotesize  $w$}};
	\draw[line width=1, color=green!50!black]  ($(0.4,-1.7)+(100:0.8 and 0)$)  node[line width=0pt] (beta) {{\footnotesize  $T$}};
	\end{tikzpicture}}
\, .
\ee
Here, $u,v,w \in D_{3}$ decorate 3-strata, $A_1,A_2,A_3 \in D_{2}$ decorate oriented 2-strata such that for example $s(A_1)=u$ and $t(A_1)=v$.  
Drawing a 2-stratum with a stripy pattern indicates that its orientation is opposite to that of the paper plane. 
To take also orientation reversal into account we extend the source and target maps 
to maps $s,t \colon D_2 \times \{ \pm \} \to D_3$ and similarly for the junction map~$j$,
see \cite{CRS2} for the full definition and more details.
Finally $T \in D_{1}$ labels the 1-stratum, and the junction map 
applied to~$T$ is the cyclic set of the decorations of incident 2-strata, $j(T)=( (A_1,+),\,(A_2,+),\,(A_3,-))/\!\sim$. 

A set of 3-dimensional defect data $\D$ yields the category  $ \Borddefn{3}(\D)$ of decorated 3-dimensional bordisms: The  objects are stratified decorated surfaces, where each $i$-stratum, $i \in \{0,1,2\}$, is decorated by an element from $D_{i+1}$ such that applying the maps $s,t$ or~$j$ to the label of a given 1- or 0-stratum, respectively, gives the decorations for the incident 2- and 1-strata. 
A morphism $N \colon \Sigma \to \Sigma'$ between objects $\Sigma,\Sigma'$ is a compact  stratified $3$-manifold $N$, 
with a decoration that is compatible with $s,t,j$ 
and an isomorphism $\Sigma^{\mathrm{op}} \sqcup \Sigma' \to \partial N$ of  decorated stratified 2-manifolds. 
Here, $\Sigma^{\mathrm{op}}$ is~$\Sigma$ with reversed orientation for all strata (but with the same decorations).
The bordisms  are considered up to isomorphism of stratified decorated manifolds relative to the boundary. 

\medskip 

\begin{remark}
\label{rem:completions}
There are two completion procedures 
for a defect TQFT $\zz \colon \Borddefn{3}(\D) \to \Vectk$ that will be important for us. 
First, one can also allow point defects in the interior of a bordism. 
The 
	maximal set of 
possible decorations~$D_{0}$ for such 0-strata 
	 turns out to be comprised of 
the elements in the vector space that~$\zz$ assigns to a small sphere~$S$ around the given defect point, subject to an invariance condition (that will however be irrelevant for the present paper), see \cite{CMS} and \cite[Sect.\,2.4]{CRS1}. 
The resulting defect TQFT is called \textsl{$D_{0}$-complete}.

Second, one can allow for certain \textsl{point insertions} on strata (called ``Euler defects'' in \cite{CRS1}). 
Point insertions are constructed from 
	 elements $\psi \in D_0$ 
that live on $i$-strata~$N_i$ for $i\in \{2,3\}$
	(which means that there are no 1-strata adjacent to the 0-stratum labelled~$\psi$) 
and which are invertible with respect to a natural multiplication on the associated vector spaces $\zz(S)$. 
Evaluating~$\zz$ on a bordism with point insertions is by definition given by inserting $\psi^{\chi_{\mathrm{sym}}(N_{i})}$, where $\chi_{\mathrm{sym}}(N_{i})$ is the ``symmetric'' Euler characteristic $2 \chi(M_{j})- \chi(\partial M_{j})$, with~$\chi$ the usual Euler characteristic, see \cite[Sect.\,2.5]{CRS1}.
\end{remark}

\smallskip
\noindent
\textbf{\textsl{Reshetikhin--Turaev defect TQFT. }}
In \cite{CRS2} we constructed a defect extension~$\zzc$ of the Reshetikhin--Turaev TQFT~$\zrt$ for every modular tensor category $\Cat{C}$. 
 The associated defect data $\DC \equiv (D_1^{\Cat C},D_2^{\Cat C},D_3^{\Cat C},s,t,j)$ are as follows. 
We have $D_3^{\Cat C} := \{ \Cat C \}$, meaning that all 3-strata are labelled by~$\Cat{C}$, and the label set for surface defects is 
\be
D_2^{\Cat C} := \big\{ \Delta\text{-separable symmetric Frobenius algebras in } \Cat C \big\} \, .
\ee
We recall that a \textsl{$\Delta$-separable symmetric Frobenius algebra~$A$} in $\Cat{C}$ is a tuple $(A,\mu,\eta,\Delta,\eps)$ consisting of an associative unital algebra $(A,\mu,\eta)$ and a coassociative counital coalgebra $(A,\Delta,\eps)$ such that 
\begin{align}
& 
\begin{tikzpicture}[very thick,scale=0.4,color=green!50!black, baseline=0cm]
\draw[-dot-] (0,0) .. controls +(0,-1) and +(0,-1) .. (1,0);
\draw[-dot-] (0,0) .. controls +(0,1) and +(0,1) .. (1,0);
\draw (0.5,-0.8) -- (0.5,-1.5); 
\draw (0.5,0.8) -- (0.5,1.5); 
\end{tikzpicture}
\, = \, 
\begin{tikzpicture}[very thick,scale=0.4,color=green!50!black, baseline=0cm]
\draw (0.5,-1.5) -- (0.5,1.5); 
\end{tikzpicture}
\, , \qquad\qquad 
\begin{tikzpicture}[very thick,scale=0.4,color=green!50!black, baseline=0cm]
\draw[-dot-] (0,0) .. controls +(0,1) and +(0,1) .. (-1,0);
\draw[directedgreen, color=green!50!black] (1,0) .. controls +(0,-1) and +(0,-1) .. (0,0);
\draw (-1,0) -- (-1,-1.5); 
\draw (1,0) -- (1,1.5); 
\draw (-0.5,1.2) node[Odot] (end) {}; 
\draw (-0.5,0.8) -- (end); 
\end{tikzpicture}
= 
\begin{tikzpicture}[very thick,scale=0.4,color=green!50!black, baseline=0cm]
\draw[redirectedgreen, color=green!50!black] (0,0) .. controls +(0,-1) and +(0,-1) .. (-1,0);
\draw[-dot-] (1,0) .. controls +(0,1) and +(0,1) .. (0,0);
\draw (-1,0) -- (-1,1.5); 
\draw (1,0) -- (1,-1.5); 
\draw (0.5,1.2) node[Odot] (end) {}; 
\draw (0.5,0.8) -- (end); 
\end{tikzpicture}
\, , \qquad\qquad 
\tikzzbox{\begin{tikzpicture}[very thick,scale=0.4,color=green!50!black, baseline=0cm]
	\draw[-dot-] (0,0) .. controls +(0,-1) and +(0,-1) .. (-1,0);
	\draw[-dot-] (1,0) .. controls +(0,1) and +(0,1) .. (0,0);
	\draw (-1,0) -- (-1,1.5); 
	\draw (1,0) -- (1,-1.5); 
	\draw (0.5,0.8) -- (0.5,1.5); 
	\draw (-0.5,-0.8) -- (-0.5,-1.5); 
	\end{tikzpicture}}
=
\tikzzbox{\begin{tikzpicture}[very thick,scale=0.4,color=green!50!black, baseline=0cm]
	\draw[-dot-] (0,1.5) .. controls +(0,-1) and +(0,-1) .. (1,1.5);
	\draw[-dot-] (0,-1.5) .. controls +(0,1) and +(0,1) .. (1,-1.5);
	\draw (0.5,-0.8) -- (0.5,0.8); 
	\end{tikzpicture}}
=
\tikzzbox{\begin{tikzpicture}[very thick,scale=0.4,color=green!50!black, baseline=0cm]
	\draw[-dot-] (0,0) .. controls +(0,1) and +(0,1) .. (-1,0);
	\draw[-dot-] (1,0) .. controls +(0,-1) and +(0,-1) .. (0,0);
	\draw (-1,0) -- (-1,-1.5); 
	\draw (1,0) -- (1,1.5); 
	\draw (0.5,-0.8) -- (0.5,-1.5); 
	\draw (-0.5,0.8) -- (-0.5,1.5); 
	\end{tikzpicture}}
\, .
\\
&
\hspace{-0.8cm}
\text{(``$\Delta$-separable'')}
\hspace{0.8cm}
\text{(``symmetric'')} 
\hspace{2.1cm}
\text{(``Frobenius'')} 
\nonumber
\end{align}

As decorations for the line defects we take 
\be
D_1^{\Cat C} := \bigsqcup_{n \in \mathbb{Z}_{\geqslant 0}} L_n \ ,
\ee 
where $L_0 = \big\{ X \in \Cat{C} \,\big|\, \theta_X = \id_X \big\}$,  
and, for $n>0$,
\begin{align*}
	L_n =\, &\big\{ 
	\big((A_1,\eps_1), (A_2,\eps_2), \dots, (A_n,\eps_n), M\big) \;\big|\; 
					A_i \in D_2^{\Cat C} , \; 
					\eps_i \in \{ \pm \} , \; 
	 \\ & \qquad
					M \text{ is a cyclic multi-module for }
					\big((A_1,\eps_1), (A_2,\eps_2), \dots, (A_n,\eps_n)\big)
					\big\} \, . 
\end{align*}
A \textsl{multi-module} over $((A_1,\eps_1), \dots, (A_n,\eps_n))$ is an $(A_{1}^{\eps_1} \otimes \dots \otimes A_{n}^{\eps_n})$-module~$M$, where $A_i^+ = A_i$ and $A_i^-$ denotes the opposite algebra $A_i^{\mathrm{op}}$.
A multi-module is cyclic if it is equivariant with respect to cyclic permutations which leave the list $((A_1,\eps_1), \dots, (A_n,\eps_n))$ invariant, 
see \cite[Def.\,5.1]{CRS2} for the precise definition. 
The multi-modules that we consider in the present paper all have only trivial cyclic symmetry, so they are all automatically equivariant and there exists only one equivariant structure. 
Hence we will have no need to pay attention to this equivariance.

We furthermore have
$s(A,\pm) \stackrel{\text{def}}{=} \Cat C \stackrel{\text{def}}{=} t(A,\pm)$ 
for all $A \in D_2^{\Cat C}$, 
and $j(M)\stackrel{\text{def}}{=}  \Cat C$ for $M \in L_{0}$, while  
\be
j \big( ((A_1,\eps_1), \dots, (A_n,\eps_n), M) \big) 
\stackrel{\text{def}}{=} 
\big((A_1,\eps_1), \dots, (A_n,\eps_n)\big)/\!\sim
\ee
for $M \in L_{n}$ with $n>0$, where as before $(\cdots)/\sim$ denotes cyclic sets. 

It is shown in \cite[Thm.\,5.8\,\&\,Rem.\,5.9]{CRS2}, that the TQFT $\zrt$ is naturally extended to a 3-dimensional defect TQFT 
\begin{equation}
\label{eq:RT-defTQFT}
\zzc \colon \Borddefen{3}(\D^{\Cat{C}}) \longrightarrow \Vectk 
\end{equation}
that we call \textsl{Reshetikhin--Turaev defect TQFT}. 
The definition of the functor $\zzc$ is roughly as follows. 
For a closed 3-bordism~$N$ pick an oriented triangulation of each 2-stratum relative to its boundary. 
The Poincar\'e dual of the triangulation gives a ribbon graph in~$N$ that is decorated by the data of the corresponding Frobenius algebra. 
By definition, evaluating~$\zzc$ on~$N$ is evaluating~$\zrt$ on the bordism which is~$N$ augmented by the ribbon graphs; this is independent of the choice of triangulation by the properties of \SSFRs. 
On objects and general bordisms our functor~$\zzc$ is defined in terms of a standard limit construction which is detailed in \cite[Sect.\,5]{CRS2}.

\subsection{Orbifolds of defect TQFTs}
\label{sec:orbif-defect-tqfts}

As recalled in the introduction, there is a general notion of orbifolds of $n$-dimensional TQFTs for any $n\geqslant 1$. 
Already for $n=3$, this produces a large list of axioms, and for practical purposes we define ``special'' 3-dimensional orbifold data to reduce the number of axioms, as recalled next.  

Fix a 3-dimensional defect TQFT $\zz\colon  \Bordd[3] \to \Vectk$. 
A \textsl{special orbifold datum~$\orb$ for~$\zz$} is a list of elements $\orb_j \in D_j$ for $j \in \{ 1,2,3 \}$ as well as $\orb_0^+,\orb_0^- \in D_0$ together with point insertions~$\psi$ and~$\phi$ for $\A_2$-labelled 2-strata and $\A_3$-labelled 3-strata, respectively (the ``Euler defects'' as recalled above), subject to the constraints below. 
In anticipation of our application to Reshetikhin--Turaev theory, we will use the notation
\be\label{eq:orb-data-notation}
\orb_3 = \Cat{C}
\,,\quad
\orb_2 = A
\,,\quad
\orb_1 = T
\,,\quad
\orb_0^+ = \alpha
\,,\quad
\orb_0^- = \bar\alpha \, , 
\ee
where ``$A$'' is for ``algebra'' and ``$T$'' is for ``tensor product''. 
The labels for 0-strata are are elements in the vector space that~$\zz$ assigns to a sphere around a 0-stratum which is dual to a 3-simplex (recall the $D_0$-completion mentioned in Remark~\ref{rem:completions}), 
\vspace{-0.8cm}
\be
\alpha \in 
\zz \Bigg( \!\!\!\!\!\!\!\!\!
\tikzzbox{
}
\Bigg)
\ee 
\vspace{-1cm}

\noindent
where all 2-, 1- and 0-strata of the defect spheres are labelled by~$\Cat{C}$, $A$ and~$T$, respectively, while the point insertions are 
\be 
\phi \in \zz \Big( 
\tikzzbox{%
}
\Big)^{\!\times} 
,
\ee
and the neighbourhoods of defects labelled by the data~\eqref{eq:orb-data-notation} 
look as follows:
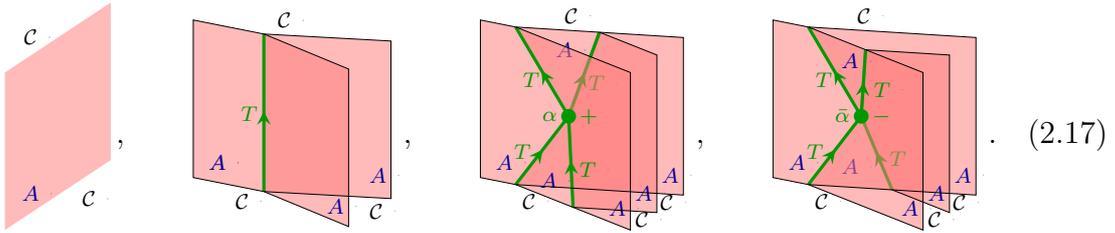
\be
\label{eq:ATalphaalphabarpic}
\!\!\!\!\!\!
\tikzzbox{%
}
\, . 
\ee
These data are subject to the axioms (which have to hold after applying the functor~$\zz$ to invisible 3-balls surrounding the diagrams)
\begin{align}
&
\label{eq:347} 
\tikzzbox{%
}
\, .  
\end{align}
In~\eqref{eq:351} in the first picture, all 2-strata have the same orientation as the paper plane, while in the second and third picture the rear, respectively front,  hemispheres have opposite orientation.
	We note that in the published version of \cite{CRS1} the identities corresponding to~\eqref{eq:351} incorrectly feature insertions of~$\phi$ and not the correct~$\phi^2$. 

We remark that any 3-dimensional defect TQFT~$\zz$ naturally gives rise to a Gray category with duals~$\tz$ as shown in \cite{CMS}; in \cite{CRS1} the definition of orbifold data is generalised to a notion internal to any such 3-category. 

\medskip 

Given an orbifold datum~$\A$ for a defect TQFT $\zz\colon \Bordd[3] \to \Vectk$, the associated \textsl{$\A$-orbifold theory} is a closed TQFT 
\be 
\zz_\A \colon \Bord_3 \lra \Vectk 
\ee
constructed in \cite[Sect.\,3.2]{CRS1}.
	On an object~$\Sigma$, $\zz_\A$ is evaluated by considering the cylinder bordism $\Sigma \times [0,1]$ and proceeding roughly as follows:
	For every triangulation~$\tau$ with total order of~$\Sigma$, denote the Poincar\'{e} dual stratification by $\Sigma^{\tau}$. By decorating with the orbifold datum $\A$ we obtain
	an  object $\Sigma^{\tau, \orb} \in \Bordd[3]$. 
	For two such triangulations $\tau,\tau'$ of~$\Sigma$, the cylinder $C_\Sigma = \Sigma \times [0,1]$, regarded  as a bordism $\Sigma \to \Sigma$ in $\Bord_3$, has an oriented triangulation~$t$ extending the triangulations $\tau$ and $\tau'$ on the ingoing and outgoing boundaries, respectively.  
	Again decorating the Poincar\'{e} dual $C_\Sigma^t$ with the orbifold datum~$\A$ we obtain a morphism $C_\Sigma^{t,\orb}\colon \Sigma^{\tau, \orb} \lra \Sigma^{\tau', \orb}$. 
	By triangulation independence we get a projective system
	\begin{equation}
	\label{eq:proj-system-orb-object}
	\zz( C_\Sigma^{t,\orb})\colon \zz(\Sigma^{\tau, \orb}) \lra \zz(\Sigma^{\tau', \orb})
	\end{equation}
	 whose limit is by definition $\zz_\A(\Sigma)$.
	 
	On a bordism~$N\colon \Sigma_{1} \rightarrow \Sigma_{2}$, the functor $\zz_\A$ is evaluated by 
	(i) choosing oriented triangulations $\tau_{1}, \tau_{2}$ of $\Sigma_{1}, \Sigma_{2}$ and extending them to an oriented triangulation of~$N$,   
	(ii) decorating the Poincar\'{e} dual stratification with the data~$\A$ to obtain a morphism $N^{t,\A}$ in $\Bordd[3]$, 
	(iii) evaluating~$\zz$ on $N^{t,\A}$, to obtain a morphism of projective systems
	\begin{equation}
	\label{eq:proj-system-orb-bord}
	\zz( N^{t,\A})\colon \zz(\Sigma_{1}^{\tau_{1}, \A}) \lra \zz(\Sigma_{2}^{\tau_{2}, \A}) \, ,
	\end{equation}
	and (iv) taking the limit to make the construction independent of choices of triangulations. 
	Note that by the construction in \cite[Sect.\,2.5]{CRS1}, for a bordism~$N$ with triangulation~$t$, on each 2- and 3-stratum adjacent to the boundary of $M^{t,\A}$, there is one inserted point defect $\psi$ and $\phi$, respectively, while 2- and 3-strata in the interior have $\psi^2$- and $\phi^2$-insertions. 
 
For
	more 
details we refer to \cite{CRS1}, but we note that in the case of a closed 3-manifold $N\colon \emptyset \to \emptyset$ we have $\zz_\A(N) = \zz(N^{t,\A})$ for all triangulations: 
in this case step (iv) above is unnecessary since by the defining property of the orbifold datum~$\A$ the value of~$\zz$ on $N^{t,\A}$ is invariant under change of triangulation.

\section{Orbifold data for Reshetikhin--Turaev theory}
\label{sec:orbif-data-resh}

We now specialise to the case of the defect TQFT~$\zzc$ from Section \ref{sec:resh-tura-defect}. 
Since the defect data for~$\zzc$ are described internal to the given modular tensor category~$\Cat{C}$, it is desirable to describe also orbifold data and their constraints internal to~$\Cat{C}$.
This internal formulation can be stated in any (not necessarily semisimple) ribbon category (Definition~\ref{def:sodribbon}). 
We will show that commutative \SSFRs\ provide examples of orbifold data, and we describe relations (such as Morita equivalence) 
between orbifold data that lead to isomorphic orbifold TQFTs.

\subsection{State spaces for spheres}

We will need to express the state spaces assigned to spheres with two and four 0-strata, respectively, and a network of 1-strata, in terms of Hom spaces in the modular tensor category $\Cat{C}$. 
For two 0-strata, the following result was proven in \cite[Lem.\,5.10]{CRS2}.

\begin{lemma}
  \label{lemma:def-points}
Let $M,N$ be two cyclic multi-modules over a list $(A_{1}, \ldots, A_{n})$ of $\Delta$-separable symmetric Frobenius algebras $A_{i}$. 
The vector space $\zzc(S_{M,N})$ associated to the 2-sphere $S_{M,N}$ with $M,N$ on its South and North pole, respectively, is given by the space of maps of multi-modules, 
\begin{equation}
  \label{eq:tft-multi}
  \zzc(S_{M,N})=\Hom_{A_{1}, \ldots, A_{n}}(M,N) \, . 
\end{equation}
\end{lemma}

More complicated state spaces are obtained by invoking the tensor product over algebras in~$\Cat{C}$:
Recall that for an algebra $A \in \Cat{C}$, a right $A$-module $(M_{A}, \rho_{M})$ and a left $A$-module $({}_{A}N, \rho_{N})$, the tensor product over~$A$, denoted by $M \otimes_{A} N$, is the coequaliser of 
\be
\begin{tikzpicture}[
baseline=(current bounding box.base), 
>=stealth,
descr/.style={fill=white,inner sep=2.5pt}, 
normal line/.style={->}
] 
\matrix (m) [matrix of math nodes, row sep=3em, column sep=2.0em, text height=1.5ex, text depth=0.25ex] {%
	M \otimes A \otimes N & & M \otimes N \, .  \\ };
\path[font=\scriptsize] (m-1-1) edge[->, transform canvas={yshift=0.5ex}] node[auto] {$ \id_{M} \otimes \rho_{N} $} (m-1-3)
(m-1-1) edge[->, transform canvas={yshift=-0.5ex}] node[auto, swap] {$ \rho_{M} \otimes \id_{N} $} (m-1-3); 
\end{tikzpicture}
\ee
For a \SSFR~$A$, we can compute the tensor product  $M \otimes_{A} N$ as the image of the projector
\begin{equation}
\label{eq:proj-ssfr}
p_{M,N} 
:= 
(\rho_{M} \otimes \rho_{N}) \circ \big(\id_{M} \otimes (\Delta \circ \eta) \otimes \id_{N}\big)
=
\begin{tikzpicture}[very thick,scale=0.75,color=blue!50!black, baseline=0cm]
\draw (0,-1.5) -- (0,1.5); 
\draw (2,-1.5) -- (2,1.5); 
%
\draw[-dot-, color=green!50!black] (0.5,0) .. controls +(0,-1) and +(0,-1) .. (1.5,0);
\draw[color=green!50!black] (0.5,0) to[out=90, in=-45] (0,1);
\draw[color=green!50!black] (1.5,0) to[out=90, in=-135] (2,1); 
\draw[color=green!50!black] (1,-1.25) node[Odot] (end) {}; 
\draw[color=green!50!black] (1,-1.2) -- (1,-0.7);
%
\fill (0,1) circle (2.9pt) node[left] (meet) {};
\fill (2,1) circle (2.9pt) node[left] (meet) {};
%
\fill (-0.25,-1.5) circle (0pt) node (M) {{\scriptsize$M$}};
\fill (2.2,-1.5) circle (0pt) node (M) {{\scriptsize$N$}};
\fill[color=green!50!black] (0.7,0.2) circle (0pt) node (M) {{\scriptsize$A$}};
\end{tikzpicture} 
. 
\end{equation}
It follows that for \SSFRs\ $A,B$ and modules $M_{A}, {}_{A}N, M'_{B}, {}_{B}N'$ we have 
\begin{align}
& \Hom_{\Cat{C}}(M \otimes_{A}N, M' \otimes_{B} N')
\nonumber 
\\ 
& \qquad = \big\{f \in \Hom_{\Cat{C}}(M \otimes N, M'\otimes N') \,\big|\, 
  f \circ p_{M,N}=f=p_{M',N'} \circ f \big\} \,.
\end{align}
The proof of the next lemma is analogous to the proof of Lemma~\ref{lemma:def-points} in \cite[Lem.\,5.10]{CRS2}. 
It basically amounts to the fact that in the definition of the defect TQFT~$\zzc$ sketched in Section~\ref{sec:resh-tura-defect}, the dual of a triangulation of an $A$-labelled 2-stratum produces projectors to tensor products over~$A$.

\begin{lemma}
\label{lemma:hor-comp}
Let $A_{1}, \ldots, A_{6}$ be \SSFRs\ in~$\Cat{C}$, 
and let 
	$_{A_{4}}K_{A_{2} \otimes A_{5}}$, 
	${}_{A_{1}}L_{A_{4} \otimes A_{6}}$, 
	${}_{A_{3}}M_{A_{5} \otimes A_{6}}$ 
	and 
	$_{A_{1}}N_{A_{2} \otimes A_{3}}$ 
be modules. 
The vector space $\zzc(\Sigma )$ associated to the defect 2-sphere 
\vspace{-1.4cm}
\begin{equation}
\Sigma \;= \hspace{-0.7cm}
\tikzzbox{\begin{tikzpicture}[very thick,scale=2.1,color=green!60!black=-0.1cm, >=stealth, baseline=0]
	\fill[ball color=white!95!blue] (0,0) circle (0.95 cm);
	\coordinate (v1) at (-0.4,-0.6);
	\coordinate (w1) at (-0.45,-0.6);
	\coordinate (v2) at (0.4,-0.6);
	\coordinate (w2) at (0.45,-0.6);
	\coordinate (v3) at (0.4,0.6);
	\coordinate (w3) at (0.32,0.65);
	\coordinate (v4) at (-0.4,0.6);
	\coordinate (w4) at (-0.5,0.65);
	\coordinate (a1) at (-0.1,-0.7);
	\coordinate (a2) at (-0.52,-0.0);
	\coordinate (a3) at (-0.1,-0.0);
	\coordinate (a4) at (0.3,-0.0);
	\coordinate (a5) at (0.0,0.62);
	\coordinate (a6) at (-0.8,0.25);
	\draw[color=red!80!black, very thick, rounded corners=0.5mm, postaction={decorate}, decoration={markings,mark=at position .5 with {\arrow[draw=red!80!black]{<}}}] 
	(v2) .. controls +(0,-0.25) and +(0,-0.25) .. (v1);
	\draw[color=red!80!black, very thick, rounded corners=0.5mm, postaction={decorate}, decoration={markings,mark=at position .62 with {\arrow[draw=red!80!black]{<}}}] 
	(v4) .. controls +(0,0.15) and +(0,0.15) .. (v3);
	\draw[color=red!80!black, very thick, rounded corners=0.5mm, postaction={decorate}, decoration={markings,mark=at position .5 with {\arrow[draw=red!80!black]{<}}}] 
	(v4) .. controls +(0.25,-0.1) and +(-0.05,0.5) .. (v2);
	\draw[color=red!80!black, very thick, rounded corners=0.5mm, postaction={decorate}, decoration={markings,mark=at position .58 with {\arrow[draw=red!80!black]{<}}}] 
	(v3) .. controls +(-0.45,1.0) and +(-1.0,0.6) .. (v1);
	\draw[color=red!80!black, very thick, rounded corners=0.5mm, postaction={decorate}, decoration={markings,mark=at position .5 with {\arrow[draw=red!80!black]{<}}}] 
	(v1) .. controls +(-0.15,0.5) and +(-0.15,-0.5) .. (v4);
	\draw[color=red!80!black, very thick, rounded corners=0.5mm, postaction={decorate}, decoration={markings,mark=at position .5 with {\arrow[draw=red!80!black]{<}}}] 
	(v3) .. controls +(0.25,-0.5) and +(0.25,0.5) .. (v2);
	\fill (v1) circle (1.6pt) node[black, opacity=0.6, right, font=\scriptsize] { $N$ };
	\fill (w1) circle (0pt) node[black, opacity=0.6, below, font=\scriptsize] { $-$ };
	\fill (v2) circle (1.6pt) node[black, opacity=0.6, right, font=\scriptsize] { $M$ };
	\fill (w2) circle (0pt) node[black, opacity=0.6, below, font=\scriptsize] { $-$ };
	\fill (v3) circle (1.6pt) node[black, opacity=0.6, right, font=\scriptsize] { $K$ };
	\fill (w3) circle (0pt) node[black, opacity=0.6, below, font=\scriptsize] { $+$ };
	\fill (v4) circle (1.6pt) node[black, opacity=0.6, right, font=\scriptsize] { $L$ };
	\fill (w4) circle (0pt) node[black, opacity=0.6, below, font=\scriptsize] { $+$ };
	\fill (a1) circle (0pt) node[black, opacity=0.6, right, font=\scriptsize] { $A_{3}$ };
	\fill (a2) circle (0pt) node[black, opacity=0.6, right, font=\scriptsize] { $A_{1}$ };
	\fill (a3) circle (0pt) node[black, opacity=0.6, right, font=\scriptsize] { $A_{6}$ };
	\fill (a4) circle (0pt) node[black, opacity=0.6, right, font=\scriptsize] { $A_{5}$ };
	\fill (a5) circle (0pt) node[black, opacity=0.6, right, font=\scriptsize] { $A_{4}$ };
	\fill (a6) circle (0pt) node[black, opacity=0.6, right, font=\scriptsize] { $A_{2}$ };
	\end{tikzpicture}}
\end{equation}
is isomorphic to 
\begin{equation}
\label{eq:hor-comp-S}
\Hom_{A_{1},A_{2}\otimes A_{5} \otimes A_{6}} \big({}_{A_{1}}N_{A_{2} \otimes A_{3}} \otimes_{A_{3}} {}_{A_{3}}M_{A_{5}\otimes A_{6}},
{}_{A_{1}}L_{A_{4} \otimes A_{6}} \otimes_{A_{4}} {}_{A_{4}}K_{A_{2} \otimes A_{5}}\big) \, .
\end{equation}
\end{lemma}

\subsection{Special orbifold data internal to ribbon fusion categories}
\label{subsec:SODribbon}

In this section we translate the orbifold data from Section~\ref{sec:orbif-defect-tqfts} to data and axioms internal to a given modular tensor category~$\Cat{C}$. 
In fact we will find that this notion does not require the nondegeneracy  
of~$\Cat{C}$ and thus makes sense in arbitrary ribbon fusion categories (Definition~\ref{def:sodribbon}), which will be relevant in our applications to $G$-extensions in Section~\ref{sec:Gextensions}. 

\medskip

Recall the notation introduced in \eqref{eq:orb-data-notation}. 
We will now describe the data $\Cat{C}, A, T, \alpha, \bar\alpha, \psi, \phi$, and the conditions these need to satisfy, for special orbifold data in Reshetikhin--Turaev TQFTs.
We start with the first three elements $\Cat{C}, A, T$. 
According to the definition of~$\DC$ as recalled in Section~\ref{sec:resh-tura-defect}, 
we have 

\begin{enumerate}
\item $\Cat{C}$ is a modular tensor category (from which the TQFT~$\zrt$ is constructed),
\item $A$ is a \SSFR\ in $\Cat{C}$, 
\item $T=\ATAA$ is an $(A, A\otimes A)$-bimodule. 
\end{enumerate}

To keep track of the various $A$-actions, we sometimes denote the bimodule~$T$ as $\ATAAi$; the corresponding 3-dimensional picture then is 
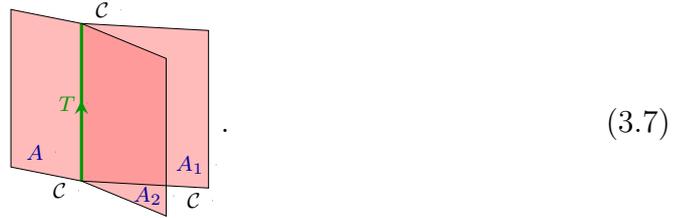
\begin{equation}
\label{eq:T-expl}
\tikzzbox{\begin{tikzpicture}[thick,scale=2.321,color=blue!50!black, baseline=0.0cm, >=stealth, 
	style={x={(-0.6cm,-0.4cm)},y={(1cm,-0.2cm)},z={(0cm,0.9cm)}}]
	\pgfmathsetmacro{\yy}{0.2}
	\coordinate (T) at (0.5, 0.4, 0);
	\coordinate (L) at (0.5, 0, 0);
	\coordinate (R1) at (0.3, 1, 0);
	\coordinate (R2) at (0.7, 1, 0);
	\coordinate (1T) at (0.5, 0.4, 1);
	\coordinate (1L) at (0.5, 0, 1);
	\coordinate (1R1) at (0.3, 1, );
	\coordinate (1R2) at (0.7, 1, );
	%
	\coordinate (p3) at (0.1, 0.1, 0.5);
	\coordinate (p2) at (0.5, 0.95, 0.5);
	\coordinate (p1) at (0.9, 0.1, 0.5);
	%
	\fill [red!50,opacity=0.545] (L) -- (T) -- (1T) -- (1L);
	\fill [red!50,opacity=0.545] (R1) -- (T) -- (1T) -- (1R1);
	\fill [red!50,opacity=0.545] (R2) -- (T) -- (1T) -- (1R2);
	\fill[color=blue!60!black] (0.5,0.25,0.15) circle (0pt) node[left] (0up) { {\scriptsize$ A$} };
	\fill[color=blue!60!black] (0.15,0.95,0.07) circle (0pt) node[left] (0up) { {\scriptsize$ A_1$} };
	\fill[color=blue!60!black] (0.55,0.95,0.05) circle (0pt) node[left] (0up) { {\scriptsize$ A_2$} };
	%
	\draw[string=green!60!black, very thick] (T) -- (1T);
	\fill[color=green!60!black] (0.5,0.43,0.5) circle (0pt) node[left] (0up) { {\scriptsize$T$} };
	%
	%
	\fill[color=black] (0.5,1.14,0.04) circle (0pt) node[left] (0up) { {\scriptsize$\mathcal C$} };
	\fill[color=black] (0.7,0.5,0.05) circle (0pt) node[left] (0up) { {\scriptsize$\mathcal C$} };
	\fill[color=black] (0.3,0.5,1.02) circle (0pt) node[left] (0up) { {\scriptsize$\mathcal C$} };
	%
	\draw [black,opacity=1, very thin] (1T) -- (1L) -- (L) -- (T);
	\draw [black,opacity=1, very thin] (1T) -- (1R1) -- (R1) -- (T);
	\draw [black,opacity=1, very thin] (1T) -- (1R2) -- (R2) -- (T);
	\end{tikzpicture}}
\, .
\end{equation}
Consistently with the 3-dimensional picture, the right $(A \otimes A)$-action 
is equivalently described by two right $A$-actions on~$T$, denoted with the corresponding number on the $A$-strings. 
These $A$-actions commute in the following sense: 
\be
\label{eq:A1A2comp}
\begin{tikzpicture}[very thick,scale=0.75,color=blue!50!black, baseline,xscale=-1]
\draw (0,-1) node[below] (X) {{\scriptsize$T$}};
\draw[color=green!50!black] (-0.5,-1) node[below] (X) {{\scriptsize$A_{2}$}};
\draw[color=green!50!black] (-1,-1) node[below] (X) {{\scriptsize$A_{1}$}};
\draw (0,1) node[right] (Xu) {};
\draw (0,-1) -- (0,1); 
\draw[color=green!50!black] (-0.5,-1) .. controls +(0,0.25) and +(-0.25,-0.25) .. (0,-0.25);
\draw[color=green!50!black] (-1,-1) .. controls +(0,0.5) and +(-0.5,-0.5) .. (0,0.6);
\fill[color=blue!50!black] (0,-0.25) circle (2.9pt) node (meet) {};
\fill[color=blue!50!black] (0,0.6) circle (2.9pt) node (meet2) {};
\fill[color=black] (-0.25,-0.25) circle (0pt) node (M) {{\scriptsize$2$}};
\fill[color=black] (-0.25,0.6)  circle (0pt) node (M) {{\scriptsize$1$}};
\end{tikzpicture} 
=
\begin{tikzpicture}[very thick,scale=0.75,color=blue!50!black, baseline,xscale=-1]
\draw (0,-1) node[below] (X) {{\scriptsize$T$}};
\draw[color=green!50!black] (-0.5,-1) node[below] (A1) {{\scriptsize$A_2$}};
\draw[color=green!50!black] (-1,-1) node[below] (A2) {{\scriptsize$A_1$}};
\draw (0,1) node[right] (Xu) {};
\draw[color=green!50!black] (A2) .. controls +(0,0.5) and +(-0.25,-0.25) .. (0,-0.25);
\draw[color=white, line width=4pt] (A1) .. controls +(0,0.5) and +(-0.5,-0.5) .. (0,0.6);
\draw[color=green!50!black] (A1) .. controls +(0,0.5) and +(-0.5,-0.5) .. (0,0.6);
\draw (0,-1) -- (0,1); 
\fill[color=blue!50!black] (0,-0.25) circle (2.9pt) node (meet)  {};
\fill[color=blue!50!black] (0,0.6) circle (2.9pt) node (meet2) {};
\fill[color=black] (-0.25,-0.25) circle (0pt) node (M) {{\scriptsize$1$}};
\fill[color=black] (-0.25,0.6)  circle (0pt) node (M) {{\scriptsize$2$}};
\end{tikzpicture} 
,
\ee
see \cite[Lemma\,2.1]{CRS2}; of course both actions commute with the left $A$-action. 

Next we turn to the data~$\al$ and~$\alb$. 
They correspond to certain maps of tensor products over~$A$ of multi-modules, and to keep track of the actions we enumerate the $A$-actions. 
From Lemma~\ref{lemma:hor-comp} we obtain that 

\begin{enumerate}
\setcounter{enumi}{3}           
\item 
$\all$ is a map of multi-modules 
\begin{equation}
\label{eq:all-explic}
\all\colon {}_{A_{1}}\big({}_{A_{1}}T_{A_{2}A_{3}} \otimes  {}_{A_{3}}T_{A_{5}A_{6}} \big)_{A_{2}A_{5}A_{6}} 
\longrightarrow 
{}_{A_{1}}\big({}_{A_{1}}T_{A_{4}A_{6}} \otimes  {}_{A_{4}}T_{A_{2}A_{5}} \big)_{A_{2}A_{5}A_{6}} 
\end{equation}
such that $p_{1} \circ \all=\all=\all \circ p_{2}$, where $p_{1}$ is the projector with respect to the action of $A_{4}$ in the second term, see Equation \eqref{eq:proj-ssfr}, while $p_{2}$ is the projector with respect to the action of $A_{3}$ in the first term.
\item 
$\allb$ is a map of multi-modules 
\begin{equation}
\label{eq:allb-explic}
\allb \colon
{}_{A_{1}}\big({}_{A_{1}}T_{A_{4}A_{6}} \otimes  {}_{A_{4}}T_{A_{2}A_{5}} \big)_{A_{2}A_{5}A_{6}} 
\longrightarrow 
{}_{A_{1}}\big({}_{A_{1}}T_{A_{2}A_{3}} \otimes  {}_{A_{3}}T_{A_{5}A_{6}} \big)_{A_{2}A_{5}A_{6}} 
\end{equation}
such that $p_{2} \circ \allb=\allb=\allb \circ p_{1}$.
\end{enumerate}
Here we used the conventions as in
\eqref{eq:actions-MN}--\eqref{eq:actions-M-rev}
for the actions on the three-fold tensor product $A_2 \otimes A_5 \otimes A_6$.
The conditions for $\all$ and $\allb$ are more accessible when expressed graphically. In the pictures it is unambiguous to work only with labels $1,2$ for the actions of  
$\ATAAi$. The condition on $\all \colon T \otimes T \rightarrow T \otimes T$ to be a map of multi-modules reads 
\be
\label{eq:alphaactioncomp}
\begin{tikzpicture}[very thick,scale=0.75,color=blue!50!black, baseline=0cm]
\draw (-0.25,0) -- (-0.25,2); 
\draw (0.25,0) -- (0.25,2); 
\draw (-0.25,-2) -- (-0.25,0); 
\draw (0.25,-2) -- (0.25,0); 
%
%
\fill[color=white] (0,0) node[inner sep=4pt,draw, rounded corners=1pt, fill, color=white] (R2) {{\scriptsize$\;\alpha\;$}};
\draw[line width=1pt, color=black] (0,0) node[inner sep=4pt, draw, semithick, rounded corners=1pt] (R) {{\scriptsize$\;\alpha\;$}};
\draw[color=green!50!black] (-0.75,-2) to[out=90, in=220] (-0.25,-0.75);
\fill (-0.25,-0.75) circle (2.9pt) node[left] (meet) {};
\draw[color=green!50!black] (1.5,-2) to[out=90, in=0] (0.25,-0.75);
\draw[color=green!50!black] (1.1,-2) to[out=90, in=0] (0.25,-1.25);
\draw[color=green!50!black] (0.7,-2) to[out=90, in=-40] (-0.25,-1.25);
\draw[color=white, line width=4pt] (0.25,-2) -- (0.25,-0.5); 
\draw (0.25,-2) -- (0.25,-0.5);
\fill (-0.25,-0.75) circle (2.9pt) node[left] (meet) {};
\fill (0.25,-0.75) circle (2.9pt) node[left] (meet) {};
\fill (0.25,-1.25) circle (2.9pt) node[left] (meet) {};
\fill (-0.25,-1.25) circle (2.9pt) node[left] (meet) {};
%
\fill[color=black] (0.45,-0.55) circle (0pt) node (M) {{\scriptsize$2$}};
\fill[color=black] (0.45,-1.05) circle (0pt) node (M) {{\scriptsize$1$}};
\fill[color=black] (-0.05,-1.05) circle (0pt) node (M) {{\scriptsize$1$}};
\end{tikzpicture} 
= 
\begin{tikzpicture}[very thick,scale=0.75,color=blue!50!black, baseline=0cm]
\draw (-0.25,0) -- (-0.25,2); 
\draw (0.25,0) -- (0.25,2); 
\draw (-0.25,-2) -- (-0.25,0); 
\draw (0.25,-2) -- (0.25,0); 
%
%
\fill[color=white] (0,0) node[inner sep=4pt,draw, rounded corners=1pt, fill, color=white] (R2) {{\scriptsize$\;\alpha\;$}};
\draw[line width=1pt, color=black] (0,0) node[inner sep=4pt, draw, semithick, rounded corners=1pt] (R) {{\scriptsize$\;\alpha\;$}};
\draw[color=green!50!black] (-1,-2) to[out=90, in=230] (-0.25,1.25);
\draw[color=white, line width=4pt] (2,-2) to[out=90, in=0] (-0.25,1.75);
\draw[color=green!50!black] (2,-2) to[out=90, in=0] (-0.25,1.75);
\draw[color=green!50!black] (1.5,-2) to[out=90, in=-30] (0.25,1.25);
\draw[color=green!50!black] (1,-2) to[out=90, in=-40] (0.25,0.75);
\fill (-0.25,1.25) circle (2.9pt) node[left] (meet) {};
\fill (-0.25,1.75) circle (2.9pt) node[left] (meet) {};
\fill (0.25,1.25) circle (2.9pt) node[left] (meet) {};
\fill (0.25,0.75) circle (2.9pt) node[left] (meet) {};
%
\fill[color=black] (0.45,1.4) circle (0pt) node (M) {{\scriptsize$2$}};
\fill[color=black] (0.45,0.9) circle (0pt) node (M) {{\scriptsize$1$}};
\fill[color=black] (-0.05,1.95) circle (0pt) node (M) {{\scriptsize$2$}};
\end{tikzpicture} 
\ee 
while the conditions involving the projectors is
\be 
\label{eq:alphamodulemap}
\begin{tikzpicture}[very thick,scale=0.75,color=blue!50!black, baseline=0cm]
\draw (0,-3) -- (0,3); 
\draw (2,-3) -- (2,3); 
%
\draw[-dot-, color=green!50!black] (0.5,0) .. controls +(0,-1) and +(0,-1) .. (1.5,0);
\draw[color=green!50!black] (0.5,0) to[out=90, in=-45] (0,1);
\draw[color=green!50!black] (1.5,0) to[out=90, in=-135] (2,1); 
\draw[color=green!50!black] (1,-1.25) node[Odot] (end) {}; 
\draw[color=green!50!black] (1,-1.2) -- (1,-0.7);
%
\fill (0,1) circle (2.9pt) node[left] (meet) {};
\fill (2,1) circle (2.9pt) node[left] (meet) {};
%
\fill[color=black] (0.2,1.1) circle (0pt) node (meet) {{\tiny$1$}};
%
%
\fill[color=white] (1,-2) node[inner sep=6pt,draw, rounded corners=1pt, fill, color=white] (R2) {{\scriptsize$\hspace{0.6cm}\alpha\hspace{0.6cm}$}};
\draw[line width=1pt, color=black] (1,-2) node[inner sep=6pt, draw, semithick, rounded corners=1pt] (R) {{\scriptsize$\hspace{0.6cm}\alpha\hspace{0.6cm}$}};
\end{tikzpicture} 
\; = \;
\begin{tikzpicture}[very thick,scale=0.75,color=blue!50!black, baseline=0cm]
\draw (0,-3) -- (0,3); 
\draw (2,-3) -- (2,3); 
%
%
\fill[color=white] (1,0) node[inner sep=6pt,draw, rounded corners=1pt, fill, color=white] (R2) {{\scriptsize$\hspace{0.6cm}\alpha\hspace{0.6cm}$}};
\draw[line width=1pt, color=black] (1,0) node[inner sep=6pt, draw, semithick, rounded corners=1pt] (R) {{\scriptsize$\hspace{0.6cm}\alpha\hspace{0.6cm}$}};
\end{tikzpicture} 
\; = \;
\begin{tikzpicture}[very thick,scale=0.75,color=blue!50!black, baseline=0cm]
\draw (0,-3) -- (0,3); 
\draw (2,-3) -- (2,3); 
%
\draw[-dot-, color=green!50!black] (0.5,0) .. controls +(0,-1) and +(0,-1) .. (1.5,0);
\draw[color=green!50!black] (0.5,0) to[out=90, in=-45] (0,1);
\draw[color=green!50!black] (1.5,0) to[out=90, in=-135] (2,1); 
\draw[color=green!50!black] (1,-1.25) node[Odot] (end) {}; 
\draw[color=green!50!black] (1,-1.2) -- (1,-0.7);
%
\fill (0,1) circle (2.9pt) node[left] (meet) {};
\fill (2,1) circle (2.9pt) node[left] (meet) {};
%
\fill[color=black] (0.2,1.1) circle (0pt) node (meet) {{\tiny$2$}};
%
%
\fill[color=white] (1,2) node[inner sep=6pt,draw, rounded corners=1pt, fill, color=white] (R2) {{\scriptsize$\hspace{0.6cm}\alpha\hspace{0.6cm}$}};
\draw[line width=1pt, color=black] (1,2) node[inner sep=6pt, draw, semithick, rounded corners=1pt] (R) {{\scriptsize$\hspace{0.6cm}\alpha\hspace{0.6cm}$}};
\end{tikzpicture} 
\, , 
\ee 
and analogously for~$\bar\alpha$. 

From Lemma~\ref{lemma:hor-comp} we furthermore obtain: 
\begin{enumerate}
\setcounter{enumi}{5}
\item 
The point insertion $\psi$ is an invertible morphism $\psi \in \End_{AA}(A)$. 
\item 
The point insertion $\phi$ is an invertible morphism $\phi \in
\End_{\Cat{C}}(\one)$. 
\end{enumerate}

To express the axioms for the orbifold datum  $\A \equiv (\Cat{C},A,T,\al,\alb, \psi, \phi)$ internal to $\Cat{C}$, it is convenient to  consider  for an $A$-module ${}_{A}M$ the map 
\begin{equation}
  \label{eq:induction}
 \Hom_{A,A}({}_{A}A_{A},{}_{A}A_{A}) \longrightarrow \Hom_{A}({}_{A}M,{}_{A}M) \, , 
\end{equation}
which sends~$\psi$ to $\psi_0 := \rho_M \circ (\id_M \otimes (\psi \circ \eta_A) )$, or graphically 
\be 
 
\, . 
\ee 
When $M_{A_{1} A_{2}}$ is a module over $A \otimes A$, we denote the images of~$\psi$ under the map analogous to \eqref{eq:induction} with respect to the $A_{i}$-action by $\psi_{i}$ for $i \in \{1,2\}$, 
\be 
%
 
\, . 
\ee 

The axioms for the data (i)--(vii) can now be formulated as follows: 

\begin{proposition}
\label{prop:internal-orb}
A special 3-dimensional orbifold datum $\A \equiv (\Cat{C},A,T,\al,\alb, \psi, \phi) $ 
for~$\zzc$ consists of the data (i)--(vii) subject to the following conditions, expressed in terms of string diagrams in~$\Cat{C}$: 
\begin{align}
& 
\label{eq:347n}
%
 
\, . 
\end{align}
\end{proposition}

\begin{proof}  
As recalled in Section~\ref{sec:resh-tura-defect}, to evaluate the defect TQFT $\zzc$ we need to pick triangulations for all 2-strata and use the data of the Frobenius algebra~$A$ to label the dual graphs. 
To verify the axioms, note that all 2-strata are discs, thus it is enough to consider one attached $A$-line per 2-stratum neighbouring a given 1-stratum. 
This translates the conditions \eqref{eq:347}--\eqref{eq:351} into those of \eqref{eq:347n}--\eqref{eq:351n}.
\end{proof}

The data and conditions on a special orbifold datum can be formulated for general ribbon categories~$\Cat{B}$, without assumptions such as $\Bbbk$-linearity or semisimplicity. 
This is useful, as such orbifold data can then be placed in a modular tensor category via a ribbon functor. 
Let us describe this in more detail.

\begin{definition}
\label{def:sodribbon}
Let~$\Cat{B}$ be a ribbon category. 
A \textsl{special orbifold datum in $\Cat{B}$} is a tuple $(A,T,\alpha,\bar\alpha,\psi,\phi)$ as in (ii)--(vii) above (with~$\Cat{C}$ replaced by~$\B$), subject to the conditions \eqref{eq:347n}--\eqref{eq:351n}.
\end{definition}

We can use ribbon functors to transport special orbifold data. 
The following result is immediate and we omit its proof.

\begin{proposition}
\label{proposition:push-forward}
Let $F\colon \Cat{B} \rightarrow \Cat{B}'$ be a ribbon functor and $\mathcal A=(A,T,\all,\allb,\psi,\phi)$ a special orbifold datum for $\Cat{B}$. 
Then $F(\mathcal A) :=(F(A),F(T),F(\all),F(\allb),F(\psi),\phi)$ is a special orbifold datum for $\Cat{B}'$ (where we suppress the coherence isomorphism of~$F$ in the notation).
\end{proposition}

\subsection{Morita equivalent orbifold data}
\label{subsec:MoritaOD}

In this section we investigate the interplay between Morita equivalence and special orbifold data in ribbon fusion categories. 
We explain how such orbifold data can be ``transported along Morita equivalences'', and we prove that the corresponding orbifold TQFTs are isomorphic. 

\medskip

Recall that two algebras $A,B$ in a pivotal tensor category are \textsl{Morita equivalent} if there exists an $A$-$B$-bimodule~$X$ together with bimodule isomorphisms 
\be
\label{eq:Moritaisos}
X^* \otimes_A X \cong B
\, , \quad 
X \otimes _B X^* \cong A \, . 
\ee 
By a \textsl{Morita equivalence} between algebras $A,B$ we mean a choice of such a bimodule~$X$ (called \textsl{Morita module}) and isomorphisms as in~\eqref{eq:Moritaisos}. 

\begin{notation}
	Let~$A$ be a $\Delta$-separable symmetric Frobenius algebra, $M$ a left $A$-module and~$N$ a right $A$-module. 
	We sometimes denote the projector~$p_{M,N}$ of~\eqref{eq:proj-ssfr} string-diagrammatically by colouring the region between the $M$- and $N$-lines: 
	\be 
	\label{eq:pMNgreen}
	p_{M,N} 
	= 
	\begin{tikzpicture}[very thick,scale=0.75,color=blue!50!black, baseline=0cm]
	\draw (0,-1.5) -- (0,1.5); 
	\draw (2,-1.5) -- (2,1.5); 
	%
	\draw[-dot-, color=green!50!black] (0.5,0) .. controls +(0,-1) and +(0,-1) .. (1.5,0);
	\draw[color=green!50!black] (0.5,0) to[out=90, in=-45] (0,1);
	\draw[color=green!50!black] (1.5,0) to[out=90, in=-135] (2,1); 
	\draw[color=green!50!black] (1,-1.25) node[Odot] (end) {}; 
	\draw[color=green!50!black] (1,-1.2) -- (1,-0.7);
	%
	\fill (0,1) circle (2.9pt) node[left] (meet) {};
	\fill (2,1) circle (2.9pt) node[left] (meet) {};
	%
	\fill (-0.25,-1.5) circle (0pt) node (M) {{\scriptsize$M$}};
	\fill (2.2,-1.5) circle (0pt) node (M) {{\scriptsize$N$}};
	\fill[color=green!50!black] (0.7,0.2) circle (0pt) node (M) {{\scriptsize$A$}};
	\end{tikzpicture} 
	\equiv 
	\begin{tikzpicture}[very thick,scale=0.75,color=blue!50!black, baseline=0cm]
	\fill [green!50!black,opacity=0.1] (0,-1.5) -- (0,1.5) -- (2,1.5) -- (2,-1.5);
	\draw (0,-1.5) -- (0,1.5); 
	\draw (2,-1.5) -- (2,1.5); 
	%
	\fill (-0.25,-1.5) circle (0pt) node (M) {{\scriptsize$M$}};
	\fill (2.2,-1.5) circle (0pt) node (M) {{\scriptsize$N$}};
	\end{tikzpicture} 
	\, . 
	\ee
	Moreover, we sometimes identify the right-hand side of~\eqref{eq:pMNgreen} with $\id_{M\otimes_A N}$ or with $M\otimes_A N$ itself. 
	For example, we employ this convention in~\eqref{eq:TX} below. 
	We stress that coloured regions always represent projectors of relative tensor products over Frobenius algebras; hence in the example of $ X^* \otimes_A T \otimes_{AA} (X\otimes X)$ in~\eqref{eq:TX}, the rightmost coloured region represents a projector corresponding to the $A$-action on~$T$ and the right $X$-factor (and not between the two modules~$X$ on the right, as their product is not relative over~$A$). 
\end{notation}

\begin{definition}
	\label{def:Moritatransport}
	Let $\A = (A,T,\alpha,\bar\alpha,\psi,\phi)$ be a special orbifold datum in a ribbon fusion category~$\B$, and let $B\in\B$ be an algebra that is Morita equivalent to~$A$ with Morita module~$X$. 
	The \textsl{Morita transport of~$\A$ along~$X$} is the tuple 
	\be 
	X(\A) := \big(B,T^X,\alpha^X,\bar\alpha^X,\psi^X,\phi\big) 
	\ee 
	where
	\begin{align}	
	T^X & := X^* \otimes_A T \otimes_{AA} (X\otimes X) 
	\equiv 
	\begin{tikzpicture}[very thick,scale=0.75,color=blue!50!black, baseline=0cm]
	\fill [green!50!black,opacity=0.1] (-0.75,1) -- (-0.75,-1) -- (1.5,-1) -- (1.5,1);
	\draw[decoration={markings, mark=at position 0.5 with {\arrow{>}},}, postaction={decorate}] (0,-1) -- (0,1);
	\draw[color=orange!80!black, decoration={markings, mark=at position 0.5 with {\arrow{>}},}, postaction={decorate}] (0.75,-1) -- (0.75,1);
	\draw[color=orange!80!black, decoration={markings, mark=at position 0.5 with {\arrow{>}},}, postaction={decorate}] (1.5,-1) -- (1.5,1);
	\draw[color=orange!80!black, decoration={markings, mark=at position 0.6 with {\arrow{>}},}, postaction={decorate}] (-0.75,1) -- (-0.75,-1);
	\end{tikzpicture} 
	\, , \label{eq:TX}
	\\ 
	\label{eq:alphaX}
	\alpha^X & := 
	\begin{tikzpicture}[very thick,scale=0.75,color=blue!50!black, baseline=0cm]
	\draw[color=orange!80!black, decoration={markings, mark=at position 0.3 with {\arrow{>}},}, postaction={decorate}] (-1,-1.5) -- (2,1.5); 
	\draw[color=orange!80!black, decoration={markings, mark=at position 0.5 with {\arrow{>}},}, postaction={decorate}] (-0.5,-1.5) .. controls +(0,0.75) and +(0,0.75) .. (1,-1.5);
	\draw[color=orange!80!black, decoration={markings, mark=at position 0.5 with {\arrow{<}},}, postaction={decorate}] (-1,1.5) .. controls +(0,-1) and +(0,-1) .. (1,1.5);
	\draw[color=orange!80!black, decoration={markings, mark=at position 0.5 with {\arrow{>}},}, postaction={decorate}] (2,-1.5) -- (2.5,1.5); 
	\draw[color=orange!80!black, decoration={markings, mark=at position 0.5 with {\arrow{>}},}, postaction={decorate}] (-2,1.5) -- (-2,-1.5); 
	\draw (0,0) -- (-1.5,1.5); 
	\draw (0,0) -- (1.5,1.5); 
	\draw (0,0) -- (-1.5,-1.5); 
	\draw (0,0) -- (1.5,-1.5); 
	\draw[color=white, line width=4pt] (0,0) -- (1.5,-1.5);  
	\draw (0,0) -- (1.5,-1.5); 
	\draw[color=white, line width=4pt] (2.5,-1.5) -- (-0.5,1.5); 
	\draw[color=orange!80!black, decoration={markings, mark=at position 0.5 with {\arrow{>}},}, postaction={decorate}] (2.5,-1.5) -- (-0.5,1.5); 
	\fill [green!50!black,opacity=0.1] (-2,1.5) -- (-2,-1.5) -- (2,-1.5) -- (2.5,1.5);
	\fill [green!50!black,opacity=0.1] (2.0,-1.5) -- (2.5,-1.5) -- (2.06,-1.1);
	%
	\fill[color=white] (0,0) node[inner sep=3pt,draw, rounded corners=1pt, fill, color=white] (R2) {{\scriptsize$\alpha$}};
	\draw[line width=1pt, color=black] (0,0) node[inner sep=3pt, draw, semithick, rounded corners=1pt] (R) {{\scriptsize$\alpha$}};
	\end{tikzpicture} 
	\, , \quad 
	\bar\alpha^X := 
	\begin{tikzpicture}[very thick,scale=0.75,color=blue!50!black, baseline=0cm]
	\draw[color=orange!80!black, decoration={markings, mark=at position 0.5 with {\arrow{>}},}, postaction={decorate}] (-1,-1.5) .. controls +(0,1) and +(0,1) .. (1,-1.5);
	\draw[color=orange!80!black, decoration={markings, mark=at position 0.5 with {\arrow{<}},}, postaction={decorate}] (-0.5,1.5) .. controls +(0,-0.75) and +(0,-0.75) .. (1,1.5);
	\draw[color=orange!80!black, decoration={markings, mark=at position 0.3 with {\arrow{>}},}, postaction={decorate}] (2,-1.5) -- (-1,1.5); 
	\draw[color=orange!80!black, decoration={markings, mark=at position 0.5 with {\arrow{>}},}, postaction={decorate}] (-2,1.5) -- (-2,-1.5); 
	\draw[color=orange!80!black, decoration={markings, mark=at position 0.5 with {\arrow{>}},}, postaction={decorate}] (2.5,-1.5) -- (2,1.5); 
	\draw (0,0) -- (-1.5,1.5); 
	\draw (0,0) -- (-1.5,-1.5); 
	\draw (0,0) -- (1.5,-1.5); 
	\draw[color=white, line width=4pt] (0,0) -- (1.5,1.5); 
	\draw (0,0) -- (1.5,1.5); 
	\draw[color=white, line width=4pt] (-0.5,-1.5) -- (2.5,1.5); 
	\draw[color=orange!80!black, decoration={markings, mark=at position 0.7 with {\arrow{>}},}, postaction={decorate}] (-0.5,-1.5) -- (2.5,1.5); 
	\fill [green!50!black,opacity=0.1] (-2,1.5) -- (-2,-1.5) -- (2.5,-1.5) -- (2,1.5);
	\fill [green!50!black,opacity=0.1] (2.0,1.5) -- (2.5,1.5) -- (2.06,1.1);
	%
	\fill[color=white] (0,0) node[inner sep=3pt,draw, rounded corners=1pt, fill, color=white] (R2) {{\scriptsize$\bar\alpha$}};
	\draw[line width=1pt, color=black] (0,0) node[inner sep=3pt, draw, semithick, rounded corners=1pt] (R) {{\scriptsize$\bar\alpha$}};
	\end{tikzpicture} 
	\end{align}
	and $\psi^X \in \End_{BB}(B)$ is a choice of square root of 
	\be 
	\label{eq:psiXsquared}
	\big( \psi^X \big)^2 := 
	\begin{tikzpicture}[very thick,scale=0.75,color=orange!80!black, baseline]
	%
	\draw[decoration={markings, mark=at position 0.5 with {\arrow{>}},}, postaction={decorate}] 
	([shift=(0:1)]1,0) arc (0:180:1);
	\draw[decoration={markings, mark=at position 0.5 with {\arrow{<}},}, postaction={decorate}] 
	([shift=(0:1)]1,0) arc (0:-180:1);
	%
	\draw[color=green!50!black] (2.5,0) to[out=90, in=-45] ([shift=(40:1)]1,0);
	\draw[-dot-, color=green!50!black] (2.5,0) .. controls +(0,-0.5) and +(0,-0.5) .. (3,0);
	\draw[color=green!50!black] (2.75,-0.3) -- (2.75,-1.5);
	\draw[color=green!50!black] (3,0) -- (3,1.5);
	%
	\fill[color=black] ([shift=(0:1)]1,0) circle (2.9pt) node[left] (meet) {{\scriptsize$\!\psi_0^{2}$}};
	\fill ([shift=(-40:1)]1,0) circle (0pt) node[right] (meet) {{\scriptsize$\!X$}};
	\fill ([shift=(40:1)]1,0) circle (2.9pt) node[right] (meet) {};
	%
	\fill[color=green!50!black] (2.95,-1.5) circle (0pt) node (M) {{\scriptsize$B$}};
	\fill[color=green!50!black] (3.2,1.5) circle (0pt) node (M) {{\scriptsize$B$}};
	\end{tikzpicture} 
	\, . 
	\ee 
\end{definition}

For algebras $A,B$ as in Definition~\ref{def:Moritatransport} we can choose a decomposition into \textsl{simple} $\Delta$-separable symmetric Frobenius algebras~$A_i$ and~$B_i$, respectively, 
\be 
\label{eq:ABdecomp}
A \cong \bigoplus_i A_i 
\, , \quad 
B \cong \bigoplus_i B_i \, , 
\ee 
such that~$A_i$ and~$B_i$ are Morita equivalent with bimodule isomorphisms  
\be 
\label{eq:Moritacomponents}
X_i^* \otimes_{A_i} X_i \cong B_i
\, , \quad 
X_i \otimes _{B_i} X_i^* \cong A_i 
\, , \quad 
X \cong \bigoplus_i X_i \, . 
\ee 

\begin{proposition}
\label{prop:Moritatransport}
Let $\B, \A, B, X$ be as in Definition~\ref{def:Moritatransport} such that $\dim(A_i) \neq 0 \neq \dim(B_i)$ for all simple algebras in~\eqref{eq:ABdecomp}. 
Then $X(\A)$ is a special orbifold datum in~$\B$. 
\end{proposition}

\begin{proof}
We have a decomposition $T \cong \bigoplus_{i,j,k} {}_kT_{i,j}$ of~$T$ into $A_k$-$(A_i\otimes A_j)$-bimodules, and up to the isomorphisms~\eqref{eq:ABdecomp} the maps $\psi, \psi^X$ are diagonal matrices with entries 
\be 
\psi_{A_i} \in \End_{A_i A_i}(A_i) \cong \Bbbk 
\, , \quad 
\psi_{B_i} \equiv (\psi^X)_{B_i}
\in \End_{B_i B_i}(B_i) \cong \Bbbk \, , 
\ee 
respectively. 
	
To prove that $X(\A)$ is a special orbifold datum we have to verify that the conditions \eqref{eq:347n}--\eqref{eq:351n} are satisfied. 
Using the above decompositions and identities of the form 
	\be 
	\label{eq:moddims}
	\begin{tikzpicture}[very thick,scale=0.75,color=blue!50!black, baseline]
	%
	\draw[decoration={markings, mark=at position 0.5 with {\arrow{>}},}, postaction={decorate}] 
	([shift=(0:1)]1,0) arc (0:180:1);
	\draw[decoration={markings, mark=at position 0.5 with {\arrow{<}},}, postaction={decorate}] 
	([shift=(0:1)]1,0) arc (0:-180:1);
	%
	\draw[color=green!50!black] (2.5,0) to[out=90, in=-45] ([shift=(40:1)]1,0);
	\draw[-dot-, color=green!50!black] (2.5,0) .. controls +(0,-0.5) and +(0,-0.5) .. (3,0);
	\draw[color=green!50!black] (2.75,-0.3) -- (2.75,-1.5);
	\draw[color=green!50!black] (3,0) -- (3,1.5);
	%
	\fill ([shift=(-40:1)]1,0) circle (0pt) node[right] (meet) {{\scriptsize$\!M$}};
	\fill ([shift=(40:1)]1,0) circle (2.9pt) node[right] (meet) {};
	%
	\fill[color=green!50!black] (3,-1.5) circle (0pt) node (M) {{\scriptsize$A_i$}};
	\fill[color=green!50!black] (3.25,1.5) circle (0pt) node (M) {{\scriptsize$A_i$}};
	\end{tikzpicture} 
	= \;
	\frac{\dim(M)}{\dim(A_i)} \cdot \id_{A_i}
	\ee 
	for simple~$A_i$ and an $A_i$-module~$M$, these checks mostly become tedious exercises in string diagram manipulations. 
	We provide the details for the first condition in~\eqref{eq:348n} and for~\eqref{eq:351n}, the remaining conditions are treated analogously.
	
We start by verifying condition~\eqref{eq:351n}. 
	Abbreviating $d_U := \dim(U)$ for all $U\in\B$ we first note that~\eqref{eq:moddims} in particular implies 
	\be 
	\label{eq:psiratio}
	\frac{\psi_{B_i}^2}{\psi_{A_i}^2} 
	= \frac{d_{X_i}}{d_{B_i}}
	= \frac{d_{A_i}}{d_{X_i}} \, . 
	\ee 
	Note that $d_{A_i} \neq 0 \neq d_{B_i}$ together with~\eqref{eq:Moritacomponents} and~\eqref{eq:moddims} implies that $d_{X_i} \neq 0$ for all~$i$. 
	Thus $\psi_{A_i} \neq 0 \neq \psi_{B_i}$ for all~$i$. 
	Moreover, 
	\begin{align}
	\sum_{i,j,k} \psi_{A_i}^2 \psi_{A_j}^2 \frac{d_{{}_kT_{i,j}}}{d_{A_k}} \cdot \varepsilon_{A_k} 
	& \stackrel{\text{\eqref{eq:moddims}}}{=} 
	\sum_{i,j,k} \psi_{A_i}^2 \psi_{A_j}^2 
	\begin{tikzpicture}[very thick,scale=0.75,color=blue!50!black, baseline]
	%
	\draw[decoration={markings, mark=at position 0.5 with {\arrow{<}},}, postaction={decorate}] 
	([shift=(0:1)]1,0) arc (0:180:1);
	\draw[decoration={markings, mark=at position 0.5 with {\arrow{>}},}, postaction={decorate}] 
	([shift=(0:1)]1,0) arc (0:-180:1);
	%
	\draw[color=green!50!black] (-1,-1.5) to[out=90, in=225] ([shift=(140:1)]1,0);
	%
	\fill ([shift=(140:1)]1,0) circle (2.9pt) node[right] (meet) {};
	%
	\fill[color=green!50!black] (-0.7,-1.5) circle (0pt) node (M) {{\scriptsize$A_k$}};
	\fill (1.75,1.1) circle (0pt) node (M) {{\scriptsize${}_kT_{i,j}$}};
	\end{tikzpicture} 
	= 
	\begin{tikzpicture}[very thick,scale=0.75,color=blue!50!black, baseline]
	%
	\draw[decoration={markings, mark=at position 0.5 with {\arrow{<}},}, postaction={decorate}] 
	([shift=(0:1)]1,0) arc (0:180:1);
	\draw[decoration={markings, mark=at position 0.5 with {\arrow{>}},}, postaction={decorate}] 
	([shift=(0:1)]1,0) arc (0:-180:1);
	%
	\draw[color=green!50!black] (-1,-1.5) to[out=90, in=225] ([shift=(140:1)]1,0);
	%
	\fill[color=black] ([shift=(-180:1)]1,0) circle (2.9pt) node[right] (meet) {{\scriptsize$\!\psi_1^{2}$}};
	\fill[color=black] ([shift=(-140:1)]1,0) circle (2.9pt) node[right] (meet) {{\scriptsize$\!\psi_2^{2}$}};
	\fill ([shift=(140:1)]1,0) circle (2.9pt) node[right] (meet) {};
	%
	\fill[color=green!50!black] (-0.8,-1.5) circle (0pt) node (M) {{\scriptsize$A$}};
	\fill (1.75,1.1) circle (0pt) node (M) {{\scriptsize$T$}};
	\end{tikzpicture} 
	\nonumber
	\\ 
	& \stackrel{\text{\eqref{eq:351n}}}{=} 
		 \phi^{-2}
	\cdot 
	\begin{tikzpicture}[very thick,scale=0.75,color=blue!50!black, baseline=-0.5cm]
	%
	\draw[color=green!50!black] (0,-1.5) -- (0,0.5);
	\draw[color=green!50!black] (0,0.5) node[Odot] (end) {}; 
	%
	\fill[color=black] (0,-0.5) circle (2.9pt) node[left] (meet) {{\scriptsize$\!\psi^{2}$}};
	%
	\fill[color=green!50!black] (0.2,-1.5) circle (0pt) node (M) {{\scriptsize$A$}};
	\end{tikzpicture} 
	= \;
	\sum_k 
	 	 \phi^{-2}
	\psi_{A_k}^2 \cdot \varepsilon_{A_k} 
	\end{align}
	and hence for all~$k$: 
	\be 
	\label{eq:351comp} 
	\sum_{i,j} \psi_{A_i}^2 \psi_{A_j}^2 \frac{d_{{}_kT_{i,j}}}{d_{A_k}}
	= 
		 \phi^{-2} 
	\psi_{A_k}^2 \, . 
	\ee 
	
	Now we compute 
	\begin{align}
	\begin{tikzpicture}[very thick,scale=0.75,color=blue!50!black, baseline]
	%
	\draw[decoration={markings, mark=at position 0.5 with {\arrow{<}},}, postaction={decorate}] 
	([shift=(0:1)]1,0) arc (0:180:1);
	\draw[decoration={markings, mark=at position 0.5 with {\arrow{>}},}, postaction={decorate}] 
	([shift=(0:1)]1,0) arc (0:-180:1);
	%
	\draw[color=green!50!black] (-1,-1.5) to[out=90, in=225] ([shift=(140:1)]1,0);
	%
	\fill[color=black] ([shift=(-180:1)]1,0) circle (2.9pt) node[right] (meet) {{\scriptsize$\!(\psi^X_1)^{2}$}};
	\fill[color=black] ([shift=(-140:1)]1,0) circle (2.9pt) node[right] (meet) {{\scriptsize$\!(\psi^X_2)^{2}$}};
	\fill ([shift=(140:1)]1,0) circle (2.9pt) node[right] (meet) {};
	%
	\fill[color=green!50!black] (-0.8,-1.5) circle (0pt) node (M) {{\scriptsize$B$}};
	\fill (1.75,1.1) circle (0pt) node (M) {{\scriptsize$T^X$}};
	\end{tikzpicture} 
	& = 
	\sum_{i,j,k} \psi_{B_i}^2 \psi_{B_j}^2 
	\begin{tikzpicture}[very thick,scale=0.75,color=blue!50!black, baseline]
	\fill[green!50!black,opacity=0.15] (1.5,0) circle (2.2);
	%
	\draw[color=green!50!black] (-1.5,-2.5) to[out=90, in=225] ($(140:2.2) + (1.5,0)$);
	%
	\draw[color=orange!80!black, decoration={markings, mark=at position 0.5 with {\arrow{>}},}, postaction={decorate}] 
	([shift=(0:1.5)]2.2,0) arc (0:180:2.2);
	\draw[color=orange!80!black, decoration={markings, mark=at position 0.5 with {\arrow{<}},}, postaction={decorate}] 
	([shift=(0:1.5)]2.2,0) arc (0:-180:2.2);
	%
	\draw[decoration={markings, mark=at position 0.5 with {\arrow{<}},}, postaction={decorate}] 
	([shift=(0:1.5)]1.5,0) arc (0:180:1.5);
	\draw[decoration={markings, mark=at position 0.5 with {\arrow{>}},}, postaction={decorate}] 
	([shift=(0:1.5)]1.5,0) arc (0:-180:1.5);
	%
	\draw[color=orange!80!black, decoration={markings, mark=at position 0.5 with {\arrow{<}},}, postaction={decorate}] 
	([shift=(0:1.5)]0.4,0) arc (0:180:0.4);
	\draw[color=orange!80!black, decoration={markings, mark=at position 0.5 with {\arrow{>}},}, postaction={decorate}] 
	([shift=(0:1.5)]0.4,0) arc (0:-180:0.4);
	%
	\draw[color=orange!80!black, decoration={markings, mark=at position 0.5 with {\arrow{<}},}, postaction={decorate}] 
	([shift=(0:1.5)]0.9,0) arc (0:180:0.9);
	\draw[color=orange!80!black, decoration={markings, mark=at position 0.5 with {\arrow{>}},}, postaction={decorate}] 
	([shift=(0:1.5)]0.9,0) arc (0:-180:0.9);
	%
	\fill[color=orange!80!black] ($(140:2.2) + (1.5,0)$) circle (2.9pt) node[right] (meet) {};
	\fill (1.85,1.2) circle (0pt) node (M) {{\scriptsize${}_kT_{i,j}$}};
	\fill[color=orange!80!black] (1.5,0) circle (0pt) node (M) {{\scriptsize$X_j$}};
	\fill[color=orange!80!black] (1.4,-1.15) circle (0pt) node (M) {{\scriptsize$X_i$}};
	\fill[color=orange!80!black] (1.4,1.9) circle (0pt) node (M) {{\scriptsize$X_k$}};
	\fill[color=green!50!black] (-1.2,-2.5) circle (0pt) node (M) {{\scriptsize$B_k$}};
	\end{tikzpicture} 
	\nonumber 
	\\ 
	& \stackrel{\text{\eqref{eq:moddims}}}{=} 
	\sum_{i,j,k} \psi_{B_i}^2 \psi_{B_j}^2 \frac{d_{X_i}}{d_{A_i}} \frac{d_{X_j}}{d_{A_j}} \frac{d_{{}_kT_{i,j}}}{d_{A_k}} \frac{d_{X_k}}{d_{B_k}} \cdot 
	\begin{tikzpicture}[very thick,scale=0.75,color=blue!50!black, baseline=0cm]
	%
	\draw[color=green!50!black] (0,-0.5) -- (0,0.5);
	\draw[color=green!50!black] (0,0.5) node[Odot] (end) {}; 
	%
	\fill[color=green!50!black] (0.3,-0.5) circle (0pt) node (M) {{\scriptsize$B_k$}};
	\end{tikzpicture} 
	\nonumber
	\\ 
	& \stackrel{\text{\eqref{eq:psiratio}}}{=}   
	\sum_k \Big(
	 \sum_{i,j} 
	 \psi_{A_i}^2 \psi_{A_j}^2 \frac{d_{{}_kT_{i,j}}}{d_{A_k}} \Big)  \frac{d_{X_k}}{d_{B_k}} \cdot \varepsilon_{B_k}
	\nonumber 
	\\ 
	& \stackrel{\text{\eqref{eq:351comp}}}{=}  
	\sum_k 
		 \phi^{-2}
	\psi_{A_k}^2 \frac{d_{X_k}}{d_{B_k}} \cdot \varepsilon_{B_k} 
	\stackrel{\text{\eqref{eq:psiratio}}}{=}  
	\sum_k 
	 \phi^{-2}
	\psi_{B_k}^2 \cdot \varepsilon_{B_k} 
	= 
		 \phi^{-2}
	\cdot 
	\begin{tikzpicture}[very thick,scale=0.75,color=blue!50!black, baseline=0cm]
	%
	\draw[color=green!50!black] (0,-0.5) -- (0,0.5);
	\draw[color=green!50!black] (0,0.5) node[Odot] (end) {}; 
	%
	\fill[color=black] (0,0) circle (2.9pt) node[left] (meet) {{\scriptsize$\!(\psi^X)^{2}$}};
	%
	\fill[color=green!50!black] (0.25,-0.5) circle (0pt) node (M) {{\scriptsize$B$}};
	\end{tikzpicture} 
	. 
	\end{align}
	The other identities in~\eqref{eq:351n} are checked analogously for $X(\A)$. 
	
	Next we check the first identity in~\eqref{eq:348n} for $X(\A)$. 
	Using~\eqref{eq:348n} for~$\A$ and~\eqref{eq:psiratio}, its left-hand side is 
	\be 
	\sum_{i,j,k,l,m,n,p} 
	\begin{tikzpicture}[very thick,scale=0.75,color=blue!50!black, baseline=-1.2cm]
	\draw[color=orange!80!black] (-1,-1.5) to[out=90, in=-135] (-0.8,-1.2) to[out=45, in=-90] (2,1.5); 
	\draw[color=orange!80!black, decoration={markings, mark=at position 0.5 with {\arrow{>}},}, postaction={decorate}] (-0.5,-1.5) .. controls +(0,0.75) and +(0,0.75) .. (1,-1.5);
	\draw[color=orange!80!black, decoration={markings, mark=at position 0.5 with {\arrow{<}},}, postaction={decorate}] (-1,1.5) .. controls +(0,-1) and +(0,-1) .. (1,1.5);
	\draw[color=orange!80!black, decoration={markings, mark=at position 0.01 with {\arrow{>}},}, postaction={decorate}] (2,-1.5) to[out=90, in=-90] (2.5,1.5); 
	\draw[color=orange!80!black, decoration={markings, mark=at position 0.5 with {\arrow{>}},}, postaction={decorate}] (-2,1.5) -- (-2,-4.5); 
	\draw (0,0) -- (-1.5,1.5); 
	\draw (0,0) -- (1.5,1.5); 
	\draw (0,0) to[out=-120, in=90] (-1.5,-1.5); 
	\draw[color=white, line width=4pt] (0,0) to[out=-60, in=90] (1.5,-1.5);  
	\draw (0,0) to[out=-60, in=90] (1.5,-1.5); 
	\draw[color=white, line width=4pt] (2.5,-1.5) to[out=90, in=-45] (-0.5,1.5); 
	\draw[color=orange!80!black] (2.5,-1.5) to[out=90, in=-45] (-0.5,1.5); 
	\fill [green!50!black,opacity=0.1] (-2,1.5) -- (-2,-1.5) -- (2,-1.5) -- (2.5,1.5);
	%
	\fill[color=white] (0,0) node[inner sep=3pt,draw, rounded corners=1pt, fill, color=white] (R2) {{\scriptsize$\alpha$}};
	\draw[line width=1pt, color=black] (0,0) node[inner sep=3pt, draw, semithick, rounded corners=1pt] (R) {{\scriptsize$\alpha$}};
	%
	%
	\draw[color=orange!80!black] (2.5,-4.5) to[out=90, in=-90] (2,-1.5); 
	\draw[color=orange!80!black, decoration={markings, mark=at position 0.5 with {\arrow{>}},}, postaction={decorate}] (-1,-4.5) .. controls +(0,1) and +(0,1) .. (1,-4.5);
	\draw[color=orange!80!black, decoration={markings, mark=at position 0.5 with {\arrow{<}},}, postaction={decorate}] (-0.5,-1.5) .. controls +(0,-0.75) and +(0,-0.75) .. (1,-1.5);
	\draw[color=orange!80!black, decoration={markings, mark=at position 0.9 with {\arrow{>}},}, postaction={decorate}] (2,-4.5) to[out=90, in=-45] (-0.8,-1.8) to[out=135, in=-90] (-1,-1.5); 
	\draw (0,-3) to[out=120, in=-90] (-1.5,-1.5); 
	\draw (0,-3) -- (-1.5,-4.5); 
	\draw (0,-3) -- (1.5,-4.5); 
	\draw[color=white, line width=4pt] (0,-3) to[out=60, in=-90] (1.5,-1.5); 
	\draw (0,-3) to[out=60, in=-90] (1.5,-1.5); 
	\draw[color=white, line width=4pt] (-0.5,-4.5) to[out=45, in=-90] (2.5,-1.5); 
	\draw[color=orange!80!black, decoration={markings, mark=at position 0.99 with {\arrow{>}},}, postaction={decorate}] (-0.5,-4.5) to[out=45, in=-90] (2.5,-1.5); 
	\fill [green!50!black,opacity=0.1] (-2,-1.5) -- (-2,-4.5) -- (2.5,-4.5) to[out=90, in=-90] (2,-1.5);
	\fill [green!50!black,opacity=0.1] (2.14,-2.52) to[out=102, in=-102] (2.18,-0.50) to[out=-60, in=55] (2.14,-2.52);
	%
	\fill[color=white] (0,-3) node[inner sep=3pt,draw, rounded corners=1pt, fill, color=white] (R2) {{\scriptsize$\bar\alpha$}};
	\draw[line width=1pt, color=black] (0,-3) node[inner sep=3pt, draw, semithick, rounded corners=1pt] (R) {{\scriptsize$\bar\alpha$}};
	\fill[color=black] (1,-1.5) circle (2.9pt) node[left] (meet) {{\scriptsize$\psi_{B_m}^{2}\!\!\!$}};
	%
	\fill[color=blue!50!black] (-1.05,0.5) circle (0pt) node (M) {{\scriptsize${}_pT_{n,i}$}};
	\fill[color=blue!50!black] (1.3,0.8) circle (0pt) node (M) {{\scriptsize${}_nT_{k,j}$}};%
	\fill[color=blue!50!black] (-1.1,-3.5) circle (0pt) node (M) {{\scriptsize${}_pT_{l,i}$}};
	\fill[color=blue!50!black] (-1.1,-2.5) circle (0pt) node (M) {{\scriptsize${}_pT_{k,m}$}};
	\fill[color=blue!50!black] (1.1,-2.5) circle (0pt) node (M) {{\scriptsize${}_mT_{j,i}$}};
	\fill[color=blue!50!black] (1.3,-3.7) circle (0pt) node (M) {{\scriptsize${}_lT_{k,j}$}};
	\fill[color=orange!80!black] (-1,1.67) circle (0pt) node (M) {{\scriptsize$X_n$}};%
	\fill[color=orange!80!black] (-0.17,-1.5) circle (0pt) node (M) {{\scriptsize$X_m$}};%
	\fill[color=orange!80!black] (-1.9,-4.7) circle (0pt) node (M) {{\scriptsize$X_p$}};
	\fill[color=orange!80!black] (-1,-4.7) circle (0pt) node (M) {{\scriptsize$X_l$}};%
	\fill[color=orange!80!black] (-0.5,-4.7) circle (0pt) node (M) {{\scriptsize$X_i$}};%
	\fill[color=orange!80!black] (2.1,-4.7) circle (0pt) node (M) {{\scriptsize$X_k$}};%
	\fill[color=orange!80!black] (2.6,-4.7) circle (0pt) node (M) {{\scriptsize$X_j$}};%
	\end{tikzpicture} 
	\; = \; 
	\sum_{i,j,k,l,p} \psi_{A_l}^{-2} \cdot 
	\begin{tikzpicture}[very thick,scale=0.75,color=blue!50!black, baseline=-1.2cm]
	\draw[color=orange!80!black, decoration={markings, mark=at position 0.5 with {\arrow{<}},}, postaction={decorate}] (-1,1.5) .. controls +(0,-1) and +(0,-1) .. (1,1.5);
	\draw[color=orange!80!black, decoration={markings, mark=at position 0.5 with {\arrow{>}},}, postaction={decorate}] (2,1.5) -- (2,-4.5); 
	\draw[color=orange!80!black, decoration={markings, mark=at position 0.5 with {\arrow{>}},}, postaction={decorate}] (2.5,1.5) -- (2.5,-4.5); 
	\draw[color=orange!80!black, decoration={markings, mark=at position 0.5 with {\arrow{>}},}, postaction={decorate}] (-2,1.5) -- (-2,-4.5); 
	\draw[color=orange!80!black, decoration={markings, mark=at position 0.5 with {\arrow{>}},}, postaction={decorate}] (-1,-4.5) .. controls +(0,1) and +(0,1) .. (1,-4.5);
	\draw (-1.5,-4.5) -- (-1.5,1.5); 
	\draw (1.5,-4.5) -- (1.5,1.5); 
	\draw[color=white, line width=4pt] (-0.5,-4.5) -- (-0.5,1.5);  
	\draw[color=orange!80!black, decoration={markings, mark=at position 0.5 with {\arrow{>}},}, postaction={decorate}] (-0.5,-4.5) -- (-0.5,1.5); 
	\fill [green!50!black,opacity=0.1] (-2,1.5) -- (-2,-4.5) -- (2.5,-4.5) -- (2.5,1.5);
	%
	%
	\fill[color=blue!50!black] (-1.05,-3) circle (0pt) node (M) {{\scriptsize${}_pT_{l,i}$}};
	\fill[color=blue!50!black] (1.05,-3) circle (0pt) node (M) {{\scriptsize${}_lT_{k,j}$}};
	\fill[color=orange!80!black] (-1,1.67) circle (0pt) node (M) {{\scriptsize$X_l$}};%
	\fill[color=orange!80!black] (-1.9,-4.7) circle (0pt) node (M) {{\scriptsize$X_p$}};
	\fill[color=orange!80!black] (-1,-4.7) circle (0pt) node (M) {{\scriptsize$X_l$}};%
	\fill[color=orange!80!black] (-0.5,-4.7) circle (0pt) node (M) {{\scriptsize$X_i$}};%
	\fill[color=orange!80!black] (2.1,-4.7) circle (0pt) node (M) {{\scriptsize$X_k$}};%
	\fill[color=orange!80!black] (2.6,-4.7) circle (0pt) node (M) {{\scriptsize$X_j$}};%
	\end{tikzpicture} 
	. 
	\ee 
	To see that this indeed equals the right-hand side of the first identity in~\eqref{eq:348n} for $X(\A)$, i.\,e. 
	\be 
	\begin{tikzpicture}[very thick,scale=0.75,color=blue!50!black, baseline=0.75cm]
	\draw (0,-1.5) -- (0,3.5); 
	\draw (2,-1.5) -- (2,3.5); 
	%
	\draw[-dot-, color=green!50!black] (0.5,0) .. controls +(0,-1) and +(0,-1) .. (1.5,0);
	\draw[color=green!50!black] (0.5,0) to[out=90, in=-45] (0,1);
	\draw[color=green!50!black] (1.5,0) to[out=90, in=-135] (2,1); 
	\draw[color=green!50!black] (1,-1.25) node[Odot] (end) {}; 
	\draw[color=green!50!black] (1,-1.2) -- (1,-0.7);
	%
	\fill[color=black] (2,2) circle (2.9pt) node[left] (meet) {{\scriptsize$(\psi_0^X)^{-2}\!$}};
	\fill (0,1) circle (2.9pt) node[left] (meet) {};
	\fill (2,1) circle (2.9pt) node[left] (meet) {};
	%
	\fill[color=black] (0.2,1.1) circle (0pt) node (meet) {{\tiny$1$}};
	\fill (-0.3,-1.5) circle (0pt) node (M) {{\scriptsize$T^X$}};
	\fill (2.4,-1.5) circle (0pt) node (M) {{\scriptsize$T^X$}};
	\fill (-0.3,3.5) circle (0pt) node (M) {{\scriptsize$T^X$}};
	\fill (2.4,3.5) circle (0pt) node (M) {{\scriptsize$T^X$}};
	\fill[color=green!50!black] (0.7,0) circle (0pt) node (M) {{\scriptsize$B$}};
	\end{tikzpicture} 
	= 
	\begin{tikzpicture}[very thick,scale=0.75,color=blue!50!black, baseline=0.75cm]
	\fill [green!50!black,opacity=0.1] (0,-1.5) -- (0,3.5) -- (2,3.5)-- (2,-1.5);
	\draw (0,-1.5) -- (0,3.5); 
	\draw (2,-1.5) -- (2,3.5); 
	%
	\draw[-dot-, color=green!50!black] (0.5,0) .. controls +(0,-1) and +(0,-1) .. (1.5,0);
	\draw[color=green!50!black] (0.5,0) to[out=90, in=-45] (0,1);
	\draw[color=green!50!black] (1.5,0) to[out=90, in=-135] (2,1); 
	\draw[color=green!50!black] (1,-1.25) node[Odot] (end) {}; 
	\draw[color=green!50!black] (1,-1.2) -- (1,-0.7);
	%
	\fill[color=black] (2,2) circle (2.9pt) node[left] (meet) {{\scriptsize$(\psi_0^X)^{-2}\!$}};
	\fill (0,1) circle (2.9pt) node[left] (meet) {};
	\fill (2,1) circle (2.9pt) node[left] (meet) {};
	%
	\fill[color=black] (0.2,1.1) circle (0pt) node (meet) {{\tiny$1$}};
	\fill (-0.3,-1.5) circle (0pt) node (M) {{\scriptsize$T^X$}};
	\fill (2.4,-1.5) circle (0pt) node (M) {{\scriptsize$T^X$}};
	\fill (-0.3,3.5) circle (0pt) node (M) {{\scriptsize$T^X$}};
	\fill (2.4,3.5) circle (0pt) node (M) {{\scriptsize$T^X$}};
	\fill[color=green!50!black] (0.7,0) circle (0pt) node (M) {{\scriptsize$B$}};
	\end{tikzpicture} 
	= 
	\sum_{i,j,k,l,p} \psi_{B_l}^{-2} \cdot 
	\begin{tikzpicture}[very thick,scale=0.75,color=blue!50!black, baseline=-0.75cm]
	\draw[color=orange!80!black, decoration={markings, mark=at position 0.5 with {\arrow{<}},}, postaction={decorate}] (-1,1.5) -- (-1,-3.5);
	\draw[color=orange!80!black, decoration={markings, mark=at position 0.5 with {\arrow{>}},}, postaction={decorate}] (2,1.5) -- (2,-3.5); 
	\draw[color=orange!80!black, decoration={markings, mark=at position 0.5 with {\arrow{>}},}, postaction={decorate}] (2.5,1.5) -- (2.5,-3.5); 
	\draw[color=orange!80!black, decoration={markings, mark=at position 0.5 with {\arrow{>}},}, postaction={decorate}] (-2,1.5) -- (-2,-3.5); 
	\draw[color=orange!80!black, decoration={markings, mark=at position 0.5 with {\arrow{>}},}, postaction={decorate}] (1,1.5) -- (1,-3.5); 
	\draw (-1.5,-3.5) -- (-1.5,1.5); 
	\draw (1.5,-3.5) -- (1.5,1.5); 
	%
	%
	\draw[-dot-, color=green!50!black] (-0.25,0) .. controls +(0,-1) and +(0,-1) .. (0.25,0);
	\draw[color=green!50!black] (-0.25,0) to[out=90, in=-45] (-1,1);
	\draw[color=green!50!black] (0.25,0) to[out=90, in=-135] (1,1); 
	\draw[color=green!50!black] (0,-1.25) node[Odot] (end) {}; 
	\draw[color=green!50!black] (0,-1.2) -- (0,-0.7);
	\fill[color=orange!80!black] (-1,1) circle (2.9pt) node[left] (meet) {};
	\fill[color=orange!80!black] (1,1) circle (2.9pt) node[left] (meet) {};
	\draw[color=white, line width=4pt] (-0.5,-3.5) -- (-0.5,1.5);  
	\draw[color=orange!80!black, decoration={markings, mark=at position 0.5 with {\arrow{<}},}, postaction={decorate}] (-0.5,1.5) -- (-0.5,-3.5); 
	\fill [green!50!black,opacity=0.1] (-2,1.5) -- (-2,-3.5) -- (2.5,-3.5) -- (2.5,1.5);
	%
	%
	\fill[color=blue!50!black] (-1.5,1.7) circle (0pt) node (M) {{\scriptsize${}_pT_{l,i}$}};
	\fill[color=blue!50!black] (1.5,1.7) circle (0pt) node (M) {{\scriptsize${}_lT_{k,j}$}};
	\fill[color=green!50!black] (0.55,0) circle (0pt) node (M) {{\scriptsize$B_l$}};
	\fill[color=orange!80!black] (-1.9,-3.7) circle (0pt) node (M) {{\scriptsize$X_p$}};
	\fill[color=orange!80!black] (-1,-3.7) circle (0pt) node (M) {{\scriptsize$X_l$}};%
	\fill[color=orange!80!black] (1,-3.7) circle (0pt) node (M) {{\scriptsize$X_l$}};%
	\fill[color=orange!80!black] (-0.5,-3.7) circle (0pt) node (M) {{\scriptsize$X_i$}};%
	\fill[color=orange!80!black] (2.1,-3.7) circle (0pt) node (M) {{\scriptsize$X_k$}};%
	\fill[color=orange!80!black] (2.6,-3.7) circle (0pt) node (M) {{\scriptsize$X_j$}};%
	\end{tikzpicture} 
	, 
	\ee
	we pre-compose both sides with 
	\be 
	\sum_{i,j,k,l,p} 
	\begin{tikzpicture}[very thick,scale=0.75,color=blue!50!black, baseline=0cm]
	\draw[color=orange!80!black, decoration={markings, mark=at position 0.5 with {\arrow{>}},}, postaction={decorate}] (2,1) -- (2,-1); 
	\draw[color=orange!80!black, decoration={markings, mark=at position 0.5 with {\arrow{>}},}, postaction={decorate}] (2.5,1) -- (2.5,-1); 
	\draw[color=orange!80!black, decoration={markings, mark=at position 0.5 with {\arrow{>}},}, postaction={decorate}] (-2,1) -- (-2,-1); 
	\draw[color=orange!80!black, decoration={markings, mark=at position 0.5 with {\arrow{<}},}, postaction={decorate}] (-1,1) .. controls +(0,-2) and +(0,-2) .. (1,1);
	\draw (-1.5,-1) -- (-1.5,1); 
	\draw (1.5,-1) -- (1.5,1); 
	\draw[color=white, line width=4pt] (-0.5,-1) -- (-0.5,1);  
	\draw[color=orange!80!black, decoration={markings, mark=at position 0.5 with {\arrow{<}},}, postaction={decorate}] (-0.5,1) -- (-0.5,-1); 
	\fill [green!50!black,opacity=0.1] (-2,1) -- (-2,-1) -- (2.5,-1) -- (2.5,1);
	%
	%
	\fill[color=blue!50!black] (-1.5,1.2) circle (0pt) node (M) {{\scriptsize${}_pT_{l,i}$}};
	\fill[color=blue!50!black] (1.5,1.2) circle (0pt) node (M) {{\scriptsize${}_lT_{k,j}$}};
	\fill[color=orange!80!black] (-1.9,-1.2) circle (0pt) node (M) {{\scriptsize$X_p$}};
	\fill[color=orange!80!black] (0.7,-0.55) circle (0pt) node (M) {{\scriptsize$X_l$}};%
	\fill[color=orange!80!black] (-0.5,-1.2) circle (0pt) node (M) {{\scriptsize$X_i$}};%
	\fill[color=orange!80!black] (2.1,-1.2) circle (0pt) node (M) {{\scriptsize$X_k$}};%
	\fill[color=orange!80!black] (2.6,-1.2) circle (0pt) node (M) {{\scriptsize$X_j$}};%
	\end{tikzpicture} 
	\ee 
	and use~\eqref{eq:psiratio} again to see that they are equal. 
\end{proof}

\begin{remark}
	We note that the expressions for $T^X, \alpha^X, \bar\alpha^X, (\psi^X)^2$ in Definition~\ref{def:Moritatransport} have a simple origin: 
	they are obtained by the rule to ``draw an $X$-line parallel to the $T$-lines on every 2-stratum in the neighbourhoods of $T,\alpha,\bar\alpha$ in~\eqref{eq:ATalphaalphabarpic}.'' 
	This rule immediately produces~\eqref{eq:TX} and~\eqref{eq:alphaX}, while~\eqref{eq:psiXsquared} is motivated by wrapping an $X$-line around a $\psi^2$-insertion on an interior 2-stratum. 
\end{remark}

\begin{remark}
\label{rem:MoritaTransportNonsemisimple}
The proof of Proposition~\ref{prop:Moritatransport} shows that the following more general result holds: 
Let $\B, \A, B, X$ be as above, except that~$\B$ is not necessarily semisimple, but still the algebras $A,B$ decompose into simple summands $A_i, B_i$ of non-zero dimension as in~\eqref{eq:ABdecomp}, and the images of the projectors as in~\eqref{eq:pMNgreen} for $A_i, B_i$ exist in~$\B$. 
Then $X(\A)$ is a special orbifold datum in the (possibly non-semisimple) ribbon category~$\B$. 
\end{remark}

\begin{proposition}
	\label{prop:Mtransportequi}
	Let $\B= \Cat{C}$ be a modular tensor category, and let $\A, X$ be as above. 
	Then the orbifold TQFTs corresponding to~$\A$ and~$X(\A)$ are isomorphic: 
	\be 
	(\zzc)_{\mathcal A} \cong (\zzc)_{X(\A)} \, . 
	\ee 
\end{proposition}

\begin{proof}
We will construct a monoidal natural isomorphism $\nu \colon (\zzc)_{\mathcal A} \to (\zzc)_{X(\A)}$. 
Recall from \cite[Sect.\,3.2]{CRS1} and Section~\ref{sec:orbif-defect-tqfts} that for $\Sigma \in \Bord_3$, the vector space $(\zzc)_\A(\Sigma)$ is defined as the limit of a projective system that is built from $\A$-decorated dual triangulations of cylinders over~$\Sigma$. 
Let~$\tau$ be an oriented triangulation of~$\Sigma$ and decorate the dual stratification with the data~$\A$ to obtain $\Sigma^{\tau,\A} \in \Borddefen{3}(\D^{\Cat{C}})$. 
Extend~$\tau$ to an oriented triangulation~$t$ of the cylinder $C_\Sigma = \Sigma \times [0,1]$ and decorate the stratification dual to~$t$ with~$\A$ to obtain a morphism $C_\Sigma^{t,\A} \colon \Sigma^{\tau,\A} \to \Sigma^{\tau,\A}$ in $\Borddefen{3}(\D^{\Cat{C}})$. 
Then $(\zzc)_\A(\Sigma) \cong \textrm{Im} \zzc(C_\Sigma^{t,\A})$. 
	
We will obtain the components~$\nu_\Sigma$ by modifying $C_\Sigma^{t,\A}$ only near its outgoing boundary $\Sigma \times \{ 1 \}$. 
The 2-strata in $C_\Sigma^{t,\A}$ have the topology of discs and are labelled by $A$. We will only be concerned with 2-strata that intersect the 
outgoing boundary component $\Sigma \times \{ 1 \}$. Let $D$ be such a 2-stratum.

In $D$ insert a semi-circular 1-stratum which starts and ends on $\Sigma \times \{ 1 \}$, which is oriented clockwise with respect to the orientation of $D$, and which is labelled by $X$. 
This splits $D$ into two disc-shaped connected components~$D_{\textrm{i}}$ and~$D_{\textrm{o}}$ (``inner'' and ``outer''). 
The disc~$D_{\textrm{i}}$ is bounded by the $X$-labelled line and a single interval on the boundary $D\cap (\Sigma \times \{ 1 \})$, while $D_{\textrm{o}}$ is bounded by two disjoint intervals in $D\cap (\Sigma \times \{ 1 \})$, as well as by $X$- and $T$-labelled 1-strata and 0-strata labelled by~$\alpha$ or~$\bar\alpha$. 
The 2-stratum $D_{\textrm{o}}$ keeps its label~$A$ while the label of~$D_{\textrm{i}}$ is changed from $A$ to $B$.

Note that by construction, each positively oriented $T$-labelled 0-stratum in the outgoing boundary $\Sigma \times \{ 1 \}$ has one $X^*$- and two $X$-labelled 0-strata in its vicinity, and vice versa for a negatively oriented 0-stratum.

Recall from Section~\ref{sec:resh-tura-defect} the construction of the ribbon  graph corresponding to this stratified bordism. 
This results in an $A$-network in $D_{\textrm{o}}$ and a $B$-network in $D_{\textrm{i}}$. Furthermore, $D_{\textrm{i}}$ gets an insertion of $\psi_B$, while $D_{\textrm{o}}$ does not get any $\psi_A$-insertion since the corresponding Euler characteristic is zero.

We choose $\varepsilon > 0$ and enlarge the underlying cylinder~$C_\Sigma$ to $C_{\Sigma,\varepsilon} = \Sigma \times [0,1+\varepsilon]$ and construct ${\widetilde C}_{\Sigma,\varepsilon}^{t,\A,X}$ as follows: it is identical to the bordism with embedded ribbon graph constructed above along the interval $[0,1]$, and it is a cylinder along $[1,1+\varepsilon]$ except 
that we insert the projectors $X^* \otimes T \otimes X \otimes X \twoheadrightarrow T^X$ 
and embeddings $T^X \hookrightarrow X^* \otimes T \otimes X \otimes X$ for all $T$-lines in $\Sigma \times [1,1+\varepsilon]$, depending on the direction of $T$, such that all ribbons ending on  $\Sigma \times \{ 1+\eps \}$ are labelled either $T^X$ or $B$.

In this way we obtain a bordism-with-ribbon-graph ${\widetilde C}_{\Sigma, \varepsilon}^{t,\A,X}$ in $\Borden{3}$ which represents a defect bordism
$\Sigma^{\tau,\A} \to \Sigma^{\tau,X(\A)}$ in $\Borddefen{3}(\D^{\Cat{C}})$, and which (after applying $\zrt$ and the projection to the limit) defines the component~$\nu_\Sigma$ of the natural transformation~$\nu$.

To verify that this is indeed a natural transformation, one writes out the naturality square and notes that one can pass from one path to the other by replacing each $\psi_A^2$ inserted on an interior 2-stratum by a small circular 1-stratum labelled~$X$ and a $\psi_B^2$-insertion in its interior, using the identity \eqref{eq:psiratio}.
Furthermore, it follows directly from the construction that~$\nu$ is monoidal (and hence an isomorphism, see e.\,g.\ \cite[Lem.\,A.2]{CR4}). 
\end{proof}

Later in Section~\ref{sec:Gextensions} we will have need to combine Morita transports with the following notion of isomorphisms of orbifold data: 

\begin{definition}
	\label{definition:T-iso}
	Let $\A = (A,T,\all,\allb,\psi,\phi)$ and $\widetilde\A = (A,\widetilde{T},\widetilde{\all},\widetilde{\allb},\psi,\phi)$ be special orbifold data in a ribbon fusion category. 
	A \textsl{T-compatible isomorphism} from~$\A$ to~$\widetilde\A$ is an isomorphism $\rho\colon T\rightarrow \widetilde{T}$ of multi-modules such that 
	\begin{equation}
	\label{eq:T-iso}
	(  \rho \otimes \rho) \circ \all= \widetilde{\all} \circ (\rho \otimes \rho) \, . 
	\end{equation}
\end{definition}

\begin{lemma}
	\label{lem:Tisomequi}
	A $T$-compatible isomorphism $\rho$ from $\mathcal A$ to $\widetilde{\mathcal A}$ induces an isomorphism between the corresponding orbifold TQFTs: $(\zzc)_{\mathcal A} \cong (\zzc)_{\widetilde{\mathcal A}}$.
\end{lemma}

\begin{proof}
	The construction of a monoidal natural isomorphism $\kappa \colon (\zzc)_{\mathcal A} \to (\zzc)_{\widetilde{\mathcal A}}$ is analogous to the construction in the proof of the previous proposition:
	Consider for $\Sigma \in \Bord_3$ a morphism $C_\Sigma^{t,\A} \colon \Sigma^{\tau,\A} \to \Sigma^{\tau,\A}$ in $\Borddefen{3}(\D^{\Cat{C}})$ 
	as above. Again we enlarge the cylinder to $C_{\Sigma,\varepsilon} = \Sigma \times [0,1+\varepsilon]$ and construct a morphism ${ C}_{\Sigma,\varepsilon}^{t,\A,\widetilde{\mathcal A}} \colon 
	\Sigma^{\tau,\A} \to \Sigma^{\tau,\widetilde{\mathcal A}}
	$ in $\Borddefen{3}(\D^{\Cat{C}})$ as  $C_\Sigma^{t,\A}$ along $[0,1]$ and as a cylinder along 
	$[1,1+\varepsilon]$, where we insert the isomorphism $\rho$ on each outwards oriented 1-stratum
	and the isomorphism $(\rho^{-1})^{*}$ on each inward oriented 1-stratum. 
	Thus ${ C}_{\Sigma,\varepsilon}^{t,\A,\widetilde{\mathcal A}} \colon 
	\Sigma^{\tau,\A} \to \Sigma^{\tau,\widetilde{\mathcal A}}
	$ is a well-defined morphism in $\Borddefen{3}(\D^{\Cat{C}})$ and we define $\kappa_{\Sigma}$ to be the map from $(\zzc)_\A(\Sigma)$ to $(\zzc)_{\widetilde{\mathcal A}}(\Sigma)$ that is induced by 
	$\zzc( C_{\Sigma,\varepsilon}^{t,\A,{\widetilde{\mathcal A}}} )$. 
	It is monoidal by construction, and for a bordism $M\colon \Sigma \rightarrow \Sigma'$ with triangulation and $\A$-decoration $M^{t',\A}$ we can replace each $\all$ by $ (\rho^{-1} \otimes \rho^{-1})\circ \widetilde{\all} \circ (\rho \otimes  \rho)$ without changing $\zzc(M^{t',\A})$. On each inner $T$-line we can then cancel  the $\rho$ with the $\rho^{-1}$ decoration to obtain a decoration by ${\widetilde{\mathcal A}}$ in the interior composed with a cylinder $ { C}_{\Sigma,\varepsilon}^{t,\A,\widetilde{\mathcal A}}$ as above and its inverse on the boundaries. After evaluating with $\zzc$ this shows the naturality of $\kappa$.
\end{proof}


\corollary\label{cor:morita-general}
	Let $\A$, $\B$, $X$ be as in Definition~\ref{def:Moritatransport}. 
	Let $\Cat{C}$ be a modular tensor category and $F \colon \B \to \Cat{C}$ a ribbon functor. 
	Then
	$(\zzc)_{F(\mathcal A)}
	\cong 
	(\zzc)_{F(X(\mathcal A))}$.
\endcorollary

\begin{proof}
		 This follows from Lemma~\ref{lem:Tisomequi} and Propositions~\ref{proposition:push-forward} and~\ref{prop:Mtransportequi}, 
	together with the observation that there is a $T$-compatible isomorphism $F(X(\mathcal A)) \cong 
	F(X)(F(\mathcal A))$.
\end{proof}

\subsection{Commutative Frobenius algebras give orbifold data}
\label{sec:commSSFR}

A simple example of orbifold data can be obtained as follows. 
Let $\Cat{B}$ be a ribbon category and let $A$ be a commutative \SSFR\ in $\Cat{B}$. Commutativity and symmetry together imply that the twist on $A$ is trivial, $\theta_A=\id_A$. 
For the bimodule $\ATAA$ we take $T=A$ with all actions given by the multiplication on $A$. For the remaining data we choose
\be\label{eq:special-orb-comm}
	\alpha = \bar\alpha = \Delta \circ \mu
	\, , \quad
	\psi = \id_A
	\, , \quad
	\phi = 1 \, .
\ee

\begin{proposition}\label{prop:comm-alg}
The tuple $(A,T,\alpha,\bar\alpha,\psi,\phi)$ described in \eqref{eq:special-orb-comm} is a special orbifold datum in $\Cat{B}$.
\end{proposition}

\begin{proof}
Note that the commutativity of~$A$ is needed for~$A$ to be a right $(A \otimes A)$-module. 
	 The fact that the conditions in Proposition~\ref{prop:internal-orb} are all satisfied follows from commutativity and the \SSFR\ properties of~$A$. 
We will check the first identity in \eqref{eq:348n} and condition \eqref{eq:349n} as examples. 

For the first condition \eqref{eq:348n}, the left-hand side becomes $\Delta \circ \mu \circ \Delta \circ \mu = \Delta \circ \mu$, using $\Delta$-separability. 
The right-hand side can be rewritten as $(\id \otimes \mu) \circ (\Delta \otimes \id)$, and the required equality is then just the Frobenius property.

As for \eqref{eq:349n}, using symmetry one can replace the duality maps on $A$ by $\eps \circ \mu$ and $\Delta \circ \eta$, thereby also replacing the $A^*$-labelled strand by its orientation-reversed version (which is then labelled $A$). 
After this, the computation boils down to using symmetry and commutativity a number of times, as well as $\Delta$-separability.
\end{proof}

In the case~$\Cat{B}$ is a modular tensor category and~$A$ is in addition simple, 
	 the orbifold theory for an orbifold datum as in Proposition~\ref{prop:comm-alg} is equivalent to the Reshetikhin--Turaev TQFT
obtained from the category of local $A$-modules in~$\Cat{B}$, which is again modular
\cite[Thm.\,4.5]{Kirillov:2001ti}. 
	 Based on the work \cite{MuleRunk} this result will be explained in \cite{CMRSS}.

\section{Turaev--Viro theory}
\label{sec:TVtheory}

In this section we explain how ``state sum models are orbifolds of the trivial theory'' in three dimensions: 
We start in Section~\ref{subsec:TVire} by reviewing Turaev--Viro theory (which constructs a state sum TQFT~$\zz^{\text{TV},\Cat{S}}$ for every spherical fusion category~$\Cat{S}$ over an algebraically closed field~$\Bbbk$), in the formulation of Turaev--Virelizier. 
Then, independently of Section~\ref{subsec:TVire}, in Section~\ref{subsec:orbidatatrivial} we define the ``trivial'' 3-dimensional defect TQFT~$\zz^{\text{triv}}$, and for every spherical fusion category~$\Cat{S}$ we construct a special orbifold datum~$\Aca$ for~$\zz^{\text{triv}}$. 
In Section~\ref{subsec:TVorbifolds} we prove that  $\zz^{\text{TV},\Cat{S}} \cong (\zz^{\text{triv}})_{\Aca}$. 
Altogether, this establishes our first main result (Theorem~A from the introduction).

\subsection{Turaev--Virelizier construction}
\label{subsec:TVire}

Let $\Cat{S} \equiv (\Cat{S},\otimes,\one)$ be a spherical fusion category. 
We choose a set~$I$ of representatives of the isomorphism classes of simple objects in~$\Cat{S}$ such that $\one \in I$, 
and we denote their quantum dimensions as 
\be
d_i := \dim(i) \in \End_{\Cat{S}}(\one) = \Bbbk 
	\qquad \text{for all } i \in I \, . 
\ee

For all $i,j,k \in I$ we say that two bases~$\Lambda$ of $\Hom_{\Cat{S}}(i\otimes j,k)$ and $\widehat{\Lambda}$ of $\Hom_{\Cat{S}}(k,i\otimes j)$ are \textsl{dual} to each other if they are dual with respect to the trace pairing 
\be\label{eq:tracepairing}
\Hom_{\Cat{S}}(i\otimes j,k) \times \Hom_{\Cat{S}}(k,i\otimes j) \supset \Lambda \times \widehat{\Lambda} \ni (\lambda, \widehat{\mu}) \lmt 
\tikzzbox{\begin{tikzpicture}[very thick,scale=0.4,color=blue!50!black, baseline=-0.1cm]
	\draw[-dot-] (0,0) .. controls +(0,1) and +(0,1) .. (1,0);
	\draw[-dot-] (1,0) .. controls +(0,-1) and +(0,-1) .. (0,0);
	\draw (0.5,0.7) -- (0.5,1.2);
	\draw[directed] (0.5,1.2) .. controls +(0,1) and +(0,1) .. (-1,1.2);
	\draw (-1,1.2) -- (-1,-1.2);
	\draw[directed] (-1,-1.2) .. controls +(0,-1) and +(0,-1) .. (0.5,-1.2);
	\draw (0.5,-0.7) -- (0.5,-1.2);
	%
	\fill (1,-1.3) circle (0pt) node (meet2) {{\scriptsize$\widehat{\mu}$}};
	\fill (1,1.2) circle (0pt) node (meet2) {{\scriptsize$\lambda$}};
	\end{tikzpicture}}
\in \Bbbk \, . 
\ee
By useful abuse of notation we then write 
\be
\tikzzbox{\begin{tikzpicture}[very thick,scale=0.4,color=blue!50!black, baseline=-0.1cm]
	\draw[-dot-] (0,0) .. controls +(0,1) and +(0,1) .. (1,0);
	\draw[-dot-] (1,0) .. controls +(0,-1) and +(0,-1) .. (0,0);
	\draw (0.5,0.7) -- (0.5,1.2);
	\draw[directed] (0.5,1.2) .. controls +(0,1) and +(0,1) .. (-1,1.2);
	\draw (-1,1.2) -- (-1,-1.2);
	\draw[directed] (-1,-1.2) .. controls +(0,-1) and +(0,-1) .. (0.5,-1.2);
	\draw (0.5,-0.7) -- (0.5,-1.2);
	%
	\fill (1,-1.3) circle (0pt) node (meet2) {{\scriptsize$\widehat{\mu}$}};
	\fill (1,1.2) circle (0pt) node (meet2) {{\scriptsize$\lambda$}};
	\end{tikzpicture}}
= \delta_{\lambda,\mu} 
\qquad \text{for all } \lambda \in \Lambda, \, \widehat{\mu} \in \widehat{\Lambda}
\ee
and we will always denote the dual basis element of~$\lambda$ by $\widehat{\lambda}$. 
Note that the simple objects $i,j,k$ are suppressed in the notation for the basis elements $\lambda, \widehat{\mu}$, and we will infer the former from the context. 

\begin{lemma}\label{lem:recall}
\begin{enumerate}
	\item 
	For all $i,j,j',k,k' \in I$ we have 
	\be
	\label{eq:basis-sum-rule-i}
\tikzzbox{\begin{tikzpicture}[very thick,scale=0.4,color=blue!50!black, baseline=0cm]
	\draw[-dot-] (0,0) .. controls +(0,1) and +(0,1) .. (1,0);
	\draw[-dot-] (1,0) .. controls +(0,-1) and +(0,-1) .. (0,0);
	\draw (0.5,0.7) -- (0.5,2);
	\draw (0.5,-0.7) -- (0.5,-2);
	%
	\fill (1,-1.3) circle (0pt) node (meet2) {{\scriptsize$\widehat{\mu}$}};
	\fill (1,1.2) circle (0pt) node (meet2) {{\scriptsize$\lambda$}};
	\fill (0.1,2) circle (0pt) node (meet2) {{\scriptsize$k'$}};
	\fill (0.1,-2) circle (0pt) node (meet2) {{\scriptsize$k$}};
	\fill (-0.3,0) circle (0pt) node (meet2) {{\scriptsize$i$}};
	\fill (1.35,0) circle (0pt) node (meet2) {{\scriptsize$j$}};
	\end{tikzpicture}}
	= \frac{1}{d_k} \, \delta_{\lambda, \mu} \delta_{k,k'} \cdot 
		\tikzzbox{\begin{tikzpicture}[very thick,scale=0.4,color=blue!50!black, baseline=0cm]
			\draw (0.5,-2) -- (0.5,2);
			\fill (0.1,-2) circle (0pt) node (meet2) {{\scriptsize$k$}};
			\end{tikzpicture}}
	 \, , 
	\qquad 
	\tikzzbox{\begin{tikzpicture}[very thick,scale=0.4,color=blue!50!black, baseline=0cm]
		\draw[directed] (0,-1) .. controls +(0,-1) and +(0,-1) .. (1,-1);
		\draw[-dot-] (1,-1) .. controls +(0,0.5) and +(0,0.5) .. (2,-1);
		\draw (2,-1) -- (2,-2);
		\draw (0,-1) -- (0,1);
		\draw[-dot-] (2,1) .. controls +(0,-0.5) and +(0,-0.5) .. (1,1);
		\draw[directed] (1,1) .. controls +(0,1) and +(0,1) .. (0,1);
		\draw (2,1) -- (2,2);
		\draw (1.5,-0.5) -- (1.5,0.5);
		%
		\fill (1.5,1.1) circle (0pt) node (meet2) {{\scriptsize$\widehat{\mu}$}};
		\fill (1.5,-1.1) circle (0pt) node (meet2) {{\scriptsize$\lambda$}};
		\fill (2.4,2) circle (0pt) node (meet2) {{\scriptsize$j'$}};
		\fill (2.3,-2) circle (0pt) node (meet2) {{\scriptsize$j$}};
		\fill (1.2,0) circle (0pt) node (meet2) {{\scriptsize$k$}};
		\fill (1.05,-1.7) circle (0pt) node (meet2) {{\scriptsize$i$}};
		\end{tikzpicture}}
	= \frac{1}{d_j} \, \delta_{\lambda, \mu} \delta_{j,j'} \cdot 
		\tikzzbox{\begin{tikzpicture}[very thick,scale=0.4,color=blue!50!black, baseline=0cm]
			\draw (0.5,-2) -- (0.5,2);
			\fill (0.1,-2) circle (0pt) node (meet2) {{\scriptsize$j$}};
			\end{tikzpicture}}
	 \, .
	\ee
	\item 
	For all $i,j,a,b \in I$ we have 
	\be
	\label{eq:basis-sum-rule-ii}
	\sum_{k,\lambda} d_k \cdot 
		\tikzzbox{\begin{tikzpicture}[very thick,scale=0.4,color=blue!50!black, baseline=-0.1cm]
			\draw[-dot-] (1,-1) .. controls +(0,0.5) and +(0,0.5) .. (2,-1);
			\draw (2,-1) -- (2,-2);
			\draw[-dot-] (2,1) .. controls +(0,-0.5) and +(0,-0.5) .. (1,1);
			\draw (2,1) -- (2,2);
			\draw (1.5,-0.5) -- (1.5,0.5);
			\draw (1,1) -- (1,2);
			\draw (1,-1) -- (1,-2);
			%
			\fill (1.5,1.2) circle (0pt) node (meet2) {{\scriptsize$\widehat{\lambda}$}};
			\fill (1.5,-1.2) circle (0pt) node (meet2) {{\scriptsize$\lambda$}};
			\fill (2.4,2) circle (0pt) node (meet2) {{\scriptsize$j$}};
			\fill (2.3,-2) circle (0pt) node (meet2) {{\scriptsize$j$}};
			\fill (0.7,2) circle (0pt) node (meet2) {{\scriptsize$i\vphantom{j}$}};
			\fill (0.7,-2) circle (0pt) node (meet2) {{\scriptsize$i\vphantom{j}$}};
			\fill (1.2,0) circle (0pt) node (meet2) {{\scriptsize$k$}};
			\end{tikzpicture}}
	= 
		\tikzzbox{\begin{tikzpicture}[very thick,scale=0.4,color=blue!50!black, baseline=-0.1cm]
		\draw (0.5,-2) -- (0.5,2);
		\fill (0.1,-2) circle (0pt) node (meet2) {{\scriptsize$i\vphantom{j}$}};
		\draw (1.5,-2) -- (1.5,2);
		\fill (2,-2) circle (0pt) node (meet2) {{\scriptsize$j$}};
		\end{tikzpicture}}
	\, , 
	\qquad 
	\sum_{l,\mu} d_l \cdot 
		\tikzzbox{\begin{tikzpicture}[very thick,scale=0.4,color=blue!50!black, baseline=-0.1cm]
			\draw[-dot-] (0,-1) .. controls +(0,-0.5) and +(0,-0.5) .. (1,-1);
			\draw[directed] (1,-1) .. controls +(0,0.5) and +(0,0.5) .. (2,-1);
			\draw (2,-1) -- (2,-2);
			\draw (0,-1) -- (0,1);
			\draw[directed] (2,1) .. controls +(0,-0.5) and +(0,-0.5) .. (1,1);
			\draw[-dot-] (1,1) .. controls +(0,0.5) and +(0,0.5) .. (0,1);
			\draw (2,1) -- (2,2);
			\draw (0.5,1.3) -- (0.5,2);
			\draw (0.5,-1.3) -- (0.5,-2);
			%
			\fill (0.5,0.8) circle (0pt) node (meet2) {{\scriptsize$\mu$}};
			\fill (0.5,-0.8) circle (0pt) node (meet2) {{\scriptsize$\widehat{\mu}$}};
			\fill (2.45,2) circle (0pt) node (meet2) {{\scriptsize$b^*$}};
			\fill (2.45,-2) circle (0pt) node (meet2) {{\scriptsize$b^*$}};
			\fill (0.2,-2) circle (0pt) node (meet2) {{\scriptsize$a\vphantom{b^*}$}};
			\fill (0.2,2) circle (0pt) node (meet2) {{\scriptsize$a\vphantom{b^*}$}};
			\fill (-0.3,0) circle (0pt) node (meet2) {{\scriptsize$l$}};
			\end{tikzpicture}}
	= 
	\tikzzbox{\begin{tikzpicture}[very thick,scale=0.4,color=blue!50!black, baseline=-0.1cm]
		\draw (0.5,-2) -- (0.5,2);
		\fill (0.1,-2) circle (0pt) node (meet2) {{\scriptsize$a\vphantom{b^*}$}};
		\draw (1.5,-2) -- (1.5,2);
		\fill (2.1,-2) circle (0pt) node (meet2) {{\scriptsize$b^*$}};
	\end{tikzpicture}}
	\ee
	where the first sum is over all $k\in I$ and (for fixed~$k$) all elements~$\lambda$ of a chosen basis of $\Hom_{\Cat{S}}(i\otimes j,k)$, and similarly for the second sum. 
	\item 
	For all $i,j,k\in I$, $\Gamma \in \Hom_{\Cat{S}}(\one, k^* \otimes i \otimes j)$ and $\Gamma' \in \Hom_{\Cat{S}}(k^* \otimes i \otimes j, \one)$ we have 
\be 
\label{eq:basis-sum-rule-iii}
\sum_\lambda
\tikzzbox{\begin{tikzpicture}[very thick,scale=0.6,color=blue!50!black, baseline=-1.6cm]
	\draw[-dot-] (1,0) .. controls +(0,-1) and +(0,-1) .. (0,0);
	\draw[directed] (-1,-1.2) .. controls +(0,-1) and +(0,-1) .. (0.5,-1.2);
	\draw (0.5,-0.7) -- (0.5,-1.2);
	\draw (-1,-1.2) -- (-1,0.5);
	\draw (0,-0) -- (0,0.5);
	\draw (1,-0) -- (1,0.5);
	\draw[-dot-] (1,-5) .. controls +(0,1) and +(0,1) .. (0,-5);
	\draw[directed] (0.5,-3.8) .. controls +(0,1) and +(0,1) .. (-1,-3.8);
	\draw (0.5,-4.3) -- (0.5,-3.8);	
	\draw (-1,-3.8) -- (-1,-5.5);
	\draw (0,-5) -- (0,-5.5);
	\draw (1,-5) -- (1,-5.5);
	%
	\fill (0.5,-0.3) circle (0pt) node (meet2) {{\scriptsize$\widehat{\lambda}$}};
	\draw (0,0.7) node[white, fill=white, inner sep=4pt,draw, rounded corners=1pt] (R) {{\scriptsize$\hspace{0.4cm}\Gamma'\hspace{0.4cm}$}};
	\draw[line width=1pt, color=black] (0,0.7) node[inner sep=4pt,draw, rounded corners=1pt] (R) {{\scriptsize$\hspace{0.4cm}\Gamma'\hspace{0.4cm}$}};
	\fill (0.5,-4.6) circle (0pt) node (meet2) {{\scriptsize$\lambda$}};
	\draw (0,-5.7) node[white, fill=white, inner sep=4pt,draw, rounded corners=1pt] (R) {{\scriptsize$\hspace{0.434cm}\Gamma\hspace{0.434cm}$}};
	\draw[line width=1pt, color=black] (0,-5.7) node[inner sep=4pt,draw, rounded corners=1pt] (R) {{\scriptsize$\hspace{0.434cm}\Gamma\hspace{0.434cm}$}};
	\fill (-0.62,-5) circle (0pt) node (meet2) {{\scriptsize$k^*\vphantom{jk}$}};
	\fill (0.2,-5) circle (0pt) node (meet2) {{\scriptsize$i\vphantom{jk}$}};
	\fill (0.8,-5) circle (0pt) node (meet2) {{\scriptsize$j\vphantom{jk}$}};
	\fill (-0.62,0) circle (0pt) node (meet2) {{\scriptsize$k^*\vphantom{jk}$}};
	\fill (-0.2,0) circle (0pt) node (meet2) {{\scriptsize$i\vphantom{jk}$}};
	\fill (1.2,0) circle (0pt) node (meet2) {{\scriptsize$j\vphantom{jk}$}};
	\end{tikzpicture}}
\; = \;\;\;\;
\tikzzbox{\begin{tikzpicture}[very thick,scale=0.6,color=blue!50!black, baseline=-1.6cm]
	\draw (-1,-5.5) -- (-1,0.5);
	\draw (0,-5.5) -- (0,0.5);
	\draw (1,-5.5) -- (1,0.5);
	%
	\draw (0,0.7) node[white, fill=white, inner sep=4pt,draw, rounded corners=1pt] (R) {{\scriptsize$\hspace{0.4cm}\Gamma'\hspace{0.4cm}$}};
	\draw[line width=1pt, color=black] (0,0.7) node[inner sep=4pt,draw, rounded corners=1pt] (R) {{\scriptsize$\hspace{0.4cm}\Gamma'\hspace{0.4cm}$}};
	\draw (0,-5.7) node[white, fill=white, inner sep=4pt,draw, rounded corners=1pt] (R) {{\scriptsize$\hspace{0.434cm}\Gamma\hspace{0.434cm}$}};
	\draw[line width=1pt, color=black] (0,-5.7) node[inner sep=4pt,draw, rounded corners=1pt] (R) {{\scriptsize$\hspace{0.434cm}\Gamma\hspace{0.434cm}$}};
	\fill (-0.62,-5) circle (0pt) node (meet2) {{\scriptsize$k^*\vphantom{jk}$}};
	\fill (0.2,-5) circle (0pt) node (meet2) {{\scriptsize$i\vphantom{jk}$}};
	\fill (0.8,-5) circle (0pt) node (meet2) {{\scriptsize$j\vphantom{jk}$}};
	\end{tikzpicture}}
	\, . 
	\ee
\end{enumerate}
In parts (i)--(iii), the vertically reflected versions of the identities hold as well.
\end{lemma}

\begin{proof}
All these identities follow from simple manipulations with bases: for part (i) take quantum traces on both sides; in part (ii) post-compose both sides with the same basis vector and use part (i); part (iii) follows from inserting the first identity in (ii) applied to $i \otimes j$ on the right-hand side and then using the second identity in (ii) together with the observation that $\Hom_{\Cat{S}}(\one, l)=\{0\}$ unless $l=\one$, and that the $l=\one$ summand in (ii) can be written as $\frac{1}{d_a}\delta_{a,b} \coev_a \circ \tev_a$. 
\end{proof}

We define the \textsl{F-matrix} elements $F^{\lambda \lambda'}_{\mu\mu'}$ in terms of the chosen bases as follows:
\begin{align}
\tikzzbox{\begin{tikzpicture}[very thick,scale=0.4,color=blue!50!black, baseline=0.3cm]
	\draw[-dot-] (1,0) .. controls +(0,1) and +(0,1) .. (2,0);
	\draw[-dot-] (0,0.7) .. controls +(0,1) and +(0,1) .. (1.5,0.7);
	\draw (0,0.7) -- (0,0);
	\draw (0.75,1.4) -- (0.75,2.5);
	\draw (0,0) -- (0,-0.5);
	\draw (1,0) -- (1,-0.5);
	\draw (2,0) -- (2,-0.5);
	%
	\fill (0.75,0.9) circle (0pt) node (meet2) {{\scriptsize$\lambda$}};
	\fill (1.5,0.2) circle (0pt) node (meet2) {{\scriptsize$\mu$}};
	\fill (-0.3,-0.5) circle (0pt) node (meet2) {{\scriptsize$a\vphantom{bj}$}};
	\fill (0.7,-0.5) circle (0pt) node (meet2) {{\scriptsize$b\vphantom{bj}$}};
	\fill (2.3,-0.5) circle (0pt) node (meet2) {{\scriptsize$j\vphantom{bj}$}};
	\fill (1.3,1.6) circle (0pt) node (meet2) {{\scriptsize$c$}};
	\fill (0.4,2.5) circle (0pt) node (meet2) {{\scriptsize$k$}};
	%
	\end{tikzpicture}}
& = 
\sum_{d,\lambda',\mu'} 
	F^{\lambda \lambda'}_{\mu\mu'} 
\cdot 
\tikzzbox{\begin{tikzpicture}[very thick,scale=0.4,color=blue!50!black, baseline=0.3cm]
	\draw[-dot-] (0,0) .. controls +(0,1) and +(0,1) .. (1,0);
	\draw[-dot-] (0.5,0.7) .. controls +(0,1) and +(0,1) .. (2,0.7);
	\draw (2,0.7) -- (2,0);
	\draw (1.25,1.4) -- (1.25,2.5);
	\draw (0,0) -- (0,-0.5);
	\draw (1,0) -- (1,-0.5);
	\draw (2,0) -- (2,-0.5);
	%
	\fill (1.25,1) circle (0pt) node (meet2) {{\scriptsize$\lambda'$}};
	\fill (0.5,0.2) circle (0pt) node (meet2) {{\scriptsize$\mu'$}};
	\fill (-0.3,-0.5) circle (0pt) node (meet2) {{\scriptsize$a\vphantom{bj}$}};
	\fill (1.3,-0.5) circle (0pt) node (meet2) {{\scriptsize$b\vphantom{bj}$}};
	\fill (2.3,-0.5) circle (0pt) node (meet2) {{\scriptsize$j\vphantom{bj}$}};
	\fill (1.6,2.5) circle (0pt) node (meet2) {{\scriptsize$k$}};
	\fill (0.55,1.65) circle (0pt) node (meet2) {{\scriptsize$d$}};
	\end{tikzpicture}}
\, , \label{eq:Fsummation}
\\
\tikzzbox{\begin{tikzpicture}[very thick,scale=0.4,color=blue!50!black, baseline=0.3cm]
	\draw[-dot-] (0,0) .. controls +(0,1) and +(0,1) .. (1,0);
	\draw[-dot-] (0.5,0.7) .. controls +(0,1) and +(0,1) .. (2,0.7);
	\draw (2,0.7) -- (2,0);
	\draw (1.25,1.4) -- (1.25,2.5);
	\draw (0,0) -- (0,-0.5);
	\draw (1,0) -- (1,-0.5);
	\draw (2,0) -- (2,-0.5);
	%
	\fill (1.25,1) circle (0pt) node (meet2) {{\scriptsize$\lambda'$}};
	\fill (0.5,0.2) circle (0pt) node (meet2) {{\scriptsize$\mu'$}};
	\fill (-0.3,-0.5) circle (0pt) node (meet2) {{\scriptsize$a\vphantom{bj}$}};
	\fill (1.3,-0.5) circle (0pt) node (meet2) {{\scriptsize$b\vphantom{bj}$}};
	\fill (2.3,-0.5) circle (0pt) node (meet2) {{\scriptsize$j\vphantom{bj}$}};
	\fill (0.55,1.65) circle (0pt) node (meet2) {{\scriptsize$d$}};
	\fill (1.6,2.5) circle (0pt) node (meet2) {{\scriptsize$k$}};
	\end{tikzpicture}}
& = 
\sum_{c,\lambda'',\mu''} (F^{-1})^{\lambda' \lambda''}_{\mu'\mu''}
\cdot 
\tikzzbox{\begin{tikzpicture}[very thick,scale=0.4,color=blue!50!black, baseline=0.3cm]
	\draw[-dot-] (1,0) .. controls +(0,1) and +(0,1) .. (2,0);
	\draw[-dot-] (0,0.7) .. controls +(0,1) and +(0,1) .. (1.5,0.7);
	\draw (0,0.7) -- (0,0);
	\draw (0.75,1.4) -- (0.75,2.5);
	\draw (0,0) -- (0,-0.5);
	\draw (1,0) -- (1,-0.5);
	\draw (2,0) -- (2,-0.5);
	%
	\fill (0.75,0.9) circle (0pt) node (meet2) {{\scriptsize$\lambda''$}};
	\fill (1.5,0.2) circle (0pt) node (meet2) {{\scriptsize$\mu''$}};
	\fill (-0.3,-0.5) circle (0pt) node (meet2) {{\scriptsize$a\vphantom{bj}$}};
	\fill (0.7,-0.5) circle (0pt) node (meet2) {{\scriptsize$b\vphantom{bj}$}};
	\fill (2.3,-0.5) circle (0pt) node (meet2) {{\scriptsize$j\vphantom{bj}$}};
	\fill (0.4,2.5) circle (0pt) node (meet2) {{\scriptsize$k$}};
	\fill (1.3,1.6) circle (0pt) node (meet2) {{\scriptsize$c$}};
	\end{tikzpicture}}
\, .
\label{eq:Finvsummation}
\end{align}
Using Lemma~\ref{lem:recall}(i), these can be expressed in terms of closed string diagrams as
\begin{align}
\label{eq:Fmat}
F^{\lambda \lambda'}_{\mu\mu'} = d_d \cdot 
\tikzzbox{\begin{tikzpicture}[very thick,scale=0.4,color=blue!50!black, baseline=0cm]
	\draw[-dot-] (0,0) .. controls +(0,1) and +(0,1) .. (-1,0);
	\draw[-dot-] (-1,0) .. controls +(0,-1) and +(0,-1) .. (-2,0);
	\draw[-dot-] (-0.5,0.7) .. controls +(0,1.25) and +(0,1.25) .. (-2,0.7);
	\draw (-2,0.7) -- (-2,0);
	\draw[-dot-] (0,-0.7) .. controls +(0,-1.25) and +(0,-1.25) .. (-1.5,-0.7);
	\draw (0,-0.7) -- (0,0);
	\draw[directed] (-1.25,1.5) .. controls +(0,1) and +(0,1) .. (-2.75,1.5);
	\draw (-2.75,1.5) -- (-2.75,-1.6);
	\draw[directed] (-2.75,-1.6) .. controls +(0,-1.25) and +(0,-1.25) .. (-0.75,-1.6);
	\fill (-0.2,-2.05) circle (0pt) node (meet2) {{\scriptsize$\widehat{\lambda'}$}};
	\fill (-2.1,-1.1) circle (0pt) node (meet2) {{\scriptsize$\widehat{\mu'}$}};
	\fill (-0.8,2.1) circle (0pt) node (meet2) {{\scriptsize$\lambda$}};
	\fill (-0.5,0.2) circle (0pt) node (meet2) {{\scriptsize$\mu$}};
        \fill (-1,-1.2) circle (0pt) node (meet2) {{\scriptsize$d$}};
	\end{tikzpicture}}
\, , \qquad 
(F^{-1})^{\lambda' \lambda''}_{\mu'\mu''} = d_c \cdot 
\tikzzbox{\begin{tikzpicture}[very thick,scale=0.4,color=blue!50!black, baseline=0cm]
	\draw[-dot-] (0,0) .. controls +(0,1) and +(0,1) .. (1,0);
	\draw[-dot-] (1,0) .. controls +(0,-1) and +(0,-1) .. (2,0);
	\draw[-dot-] (0.5,0.7) .. controls +(0,1.25) and +(0,1.25) .. (2,0.7);
	\draw (2,0.7) -- (2,0);
	\draw[-dot-] (0,-0.7) .. controls +(0,-1.25) and +(0,-1.25) .. (1.5,-0.7);
	\draw (0,-0.7) -- (0,0);
	\draw[directed] (1.25,1.5) .. controls +(0,1.25) and +(0,1.25) .. (-0.75,1.5);
	\draw (-0.75,1.5) -- (-0.75,-1.6);
	\draw[directed] (-0.75,-1.6) .. controls +(0,-1) and +(0,-1) .. (0.75,-1.6);
	%
	\fill (1.2,-2.05) circle (0pt) node (meet2) {{\scriptsize$\widehat{\lambda''}$}};
	\fill (2.1,-1.1) circle (0pt) node (meet2) {{\scriptsize$\widehat{\mu''}$}};
	\fill (1.7,2.1) circle (0pt) node (meet2) {{\scriptsize$\lambda'$}};
	\fill (0.5,0.2) circle (0pt) node (meet2) {{\scriptsize$\mu'$}};
	\fill (1,-1.2) circle (0pt) node (meet2) {{\scriptsize$c$}};
	\end{tikzpicture}}
\, . 
\end{align} 
The pentagon identity satisfied by the associator translates into an identity for $F$-matrix elements as follows. 
One computes the change-of-basis matrix $M_{\lambda\mu\nu,\lambda'\mu'\nu'}$ in 
\be\label{eq:pentagon-two-bases-1}
\tikzzbox{\begin{tikzpicture}[very thick,scale=0.4,color=blue!50!black, baseline=0.8cm]
	\draw (0,0) -- (4,4);
	\draw (6,0) -- (7,1);
	\draw (4,0) -- (6,2);
	\draw (8,0) -- (4,4);
	\draw (4,4) -- (4,5);
	\fill (0,-0.4) circle (0pt) node (meet2) {{\scriptsize$a\vphantom{bp}$}};
	\fill (4,-0.4) circle (0pt) node (meet2) {{\scriptsize$b\vphantom{bp}$}};
	\fill (6,-0.4) circle (0pt) node (meet2) {{\scriptsize$c\vphantom{bp}$}};
	\fill (8,-0.4) circle (0pt) node (meet2) {{\scriptsize$d\vphantom{bp}$}};
	\fill (4,5.3) circle (0pt) node (meet2) {{\scriptsize$e\vphantom{bp}$}};
	\fill (6.15,1.15) circle (0pt) node (meet2) {{\scriptsize$x$}};
	\fill (4.65,2.65) circle (0pt) node (meet2) {{\scriptsize$y$}};
	\fill (7,1) circle (5pt) node (meet2) {};
	\fill (7.6,1.3) circle (0pt) node (meet2) {{\scriptsize$\nu$}};
	\fill (6,2) circle (5pt) node (meet2) {};
	\fill (6.6,2.3) circle (0pt) node (meet2) {{\scriptsize$\mu$}};
	\fill (4,4) circle (5pt) node (meet2) {};
	\fill (4.6,4.2) circle (0pt) node (meet2) {{\scriptsize$\lambda$}};
	\end{tikzpicture}}
= \sum_{x',y',\lambda',\mu',\nu'} M_{\lambda\mu\nu,\lambda'\mu'\nu'} \cdot 
\tikzzbox{\begin{tikzpicture}[very thick,scale=0.4,color=blue!50!black, baseline=0.8cm]
	\draw (0,0) -- (4,4);
	\draw (2,0) -- (1,1);
	\draw (4,0) -- (2,2);
	\draw (8,0) -- (4,4);
	\draw (4,4) -- (4,5);
	\fill (0,-0.4) circle (0pt) node (meet2) {{\scriptsize$a\vphantom{bp}$}};
	\fill (2,-0.4) circle (0pt) node (meet2) {{\scriptsize$b\vphantom{bp}$}};
	\fill (4,-0.4) circle (0pt) node (meet2) {{\scriptsize$c\vphantom{bp}$}};
	\fill (8,-0.4) circle (0pt) node (meet2) {{\scriptsize$d\vphantom{bp}$}};
	\fill (4,5.3) circle (0pt) node (meet2) {{\scriptsize$e\vphantom{bp}$}};
	\fill (1.75,1.25) circle (0pt) node (meet2) {{\scriptsize$x'$}};
	\fill (3.4,2.75) circle (0pt) node (meet2) {{\scriptsize$y'$}};
	\fill (1,1) circle (5pt) node (meet2) {};
	\fill (0.6,1.4) circle (0pt) node (meet2) {{\scriptsize$\nu'$}};
	\fill (2,2) circle (5pt) node (meet2) {};
	\fill (1.7,2.3) circle (0pt) node (meet2) {{\scriptsize$\mu'$}};
	\fill (4,4) circle (5pt) node (meet2) {};
	\fill (3.5,4.1) circle (0pt) node (meet2) {{\scriptsize$\lambda'$}};
	\end{tikzpicture}}
\ee
in two ways. The resulting two expressions for $M_{\lambda\mu\nu,\lambda'\mu'\nu'}$  must be equal, giving the identity
\be\label{eq:pentagon-two-bases-2}
\sum_\delta F_{\mu\nu'}^{\lambda\delta} F_{\nu\mu'}^{\delta\lambda'}
=
\sum_{z,\eps,\phi,\kappa} 
F_{\nu\phi}^{\mu\eps} 
F_{\eps\kappa}^{\lambda\lambda'}
F_{\phi\nu'}^{\kappa\mu'} \, ,
\ee
where the indices take values as prescribed by \eqref{eq:pentagon-two-bases-1}, and~$z$ labels the edge between the vertices with labels $\phi$ and $\varepsilon$.

\bigskip

In the remainder of Section~\ref{subsec:TVire} we review the Turaev--Virelizier construction 
	 \cite{TVire} 
of the Turaev--Viro TQFT
\be 
\zz^{\text{TV},\Cat{S}} \colon \Bord_3 \lra \Vectk 
\ee
for a spherical fusion category~${\Cat{S}}$. 
We only provide the details we need for the proof of Theorem~\ref{thm:ZTVZtrivorb} in Section~\ref{subsec:TVorbifolds}. 

Let $\Sigma \in \Bord_3$ and let~$M$ be a 3-bordism. 
Recall from 
	 \cite[Ch.\,11]{TVire}
the notions of an \textsl{oriented graph}~$\Gamma$ in $\Sigma$, and of an \textsl{oriented stratified 2-polyhedron}~$P$ in~$M$. 
We will exclusively consider the special cases where~$\Gamma$ is the Poincar\'{e} dual of a triangulation of~$\Sigma$ with chosen orientations for the 1-strata of~$\Gamma$ (called edges), and where~$P$ is dual to a triangulation of~$M$ with chosen orientations for the 2-strata of~$P$ (called regions). 
We will denote the sets of $j$-strata of~$\Gamma$ and~$P$ by~$\Gamma_j$ and~$P_j$, respectively. 	

For an oriented graph~$\Gamma$ in~$\Sigma$ as above, let~$c$ be an \textsl{${\Cat{S}}$-colouring} of~$\Gamma$, i.\,e.\ a map $c\colon \Gamma_1 \to I$. 
For a vertex $v\in \Gamma_0$ consider the cyclic set of edges $(e_1,\dots,e_n)$ incident on~$v$ as determined by the opposite orientation of~$\Sigma$. 
Set $\varepsilon(e_i) = +$ if $e_i$ is oriented towards~$v$, and $\varepsilon(e_i) = -$ otherwise, and then 
\be \label{eq:H_ei-def}
H_{e_i} = \Hom_{\Cat{S}}\big(\one, c(e_i)^{\varepsilon(e_i)} \otimes \dots  \otimes  c(e_n)^{\varepsilon(e_n)} \otimes c(e_1)^{\varepsilon(e_1)} \otimes \dots \otimes c(e_{i-1})^{\varepsilon(e_{i-1})} \big)
\ee
where we use the notation $u^+=u$ and $u^-=u^*$ for all $u\in{\Cat{S}}$. 

The duality morphisms of~${\Cat{S}}$ induce isomorphisms~$\{f\}$ between the $H_{e_i}$ which form a projective system, and in \cite{TVire} the vector space $H(E_c^v)$ assigned to the data $E_c^v = ((e_1,\dots,e_n), c, \varepsilon)$ is defined to be the projective limit.
One can also use the duality morphisms of~${\Cat{S}}$ to obtain isomorphisms~$\{g\}$ that move tensor factors between the arguments of $\Hom_{\Cat{S}}(-,-)$, for example 
\be 
\Hom_{\Cat{S}}\big( \one, c(e_k)^+ \otimes c(e_j)^- \otimes c(e_i)^- \big) 
\cong 
\Hom_{\Cat{S}}\big( c(e_i) \otimes c(e_j), c(e_k) \big) \, . 
\ee
The projective limit of the system $\{ f,g \}$ is uniquely isomorphic to the limit of $\{ f \}$, hence we can and will work with the former as $H(E_c^v)$. 
In terms of these we set 
\be\label{eq:HGcS}
H(\Gamma, c; \Sigma) = \bigotimes_{v\in\Gamma_0} H(E_c^v)
\ee 
for $\Sigma, \Gamma, c$ as above, and we note that there is a canonical isomorphism $H(\Gamma^{\text{op}}, c; \Sigma^{\text{op}})^* \cong H(\Gamma, c; \Sigma)$, where $(-)^{\text{op}}$ signifies opposite orientation. 
For example in \eqref{eq:H_ei-def} for $i=1$ one pairs 
$\Hom_{\Cat{S}}(\one, c(e_1)^{\varepsilon(e_1)} \otimes \cdots \otimes c(e_{n})^{\varepsilon(e_{n})} )$ 
with 
$\Hom_{\Cat{S}}(\one, c(e_n)^{-\varepsilon(e_n)} \otimes \cdots \otimes c(e_{1})^{-\varepsilon(e_{1})})$ using the duality morphisms.

Let now $\Sigma = S^2$ be endowed with an oriented graph~$\Gamma$ with an ${\Cat{S}}$-colouring~$c$. 
Using the cone isomorphisms $\{ f \}$ and sphericality one can associate to these data a functional (see
	 \cite[Sect.\,12.2]{TVire})
\be 
\mathds{F}_{\Cat{S}}(\Gamma, c) \in H(\Gamma, c; \Sigma)^* \, . 
\ee
The idea is that locally around every vertex of~$\Gamma$ one can interpret it as a slot into which one can insert elements of the tensor factors in~\eqref{eq:HGcS}. 
This tensor product is then evaluated to the corresponding string diagram in~${\Cat{S}}$. 
We will have need for only two types of graph~$\Gamma$, to which we turn next; the associated functionals will be defined in~\eqref{eq:FCGinv} below. 

Consider a 3-bordism~$M$ with an oriented stratified 2-polyhedron~$P$ that comes from a triangulation of~$M$ as discussed above. 
All vertices in~$P$ correspond to oriented tetrahedra, hence locally they look like 
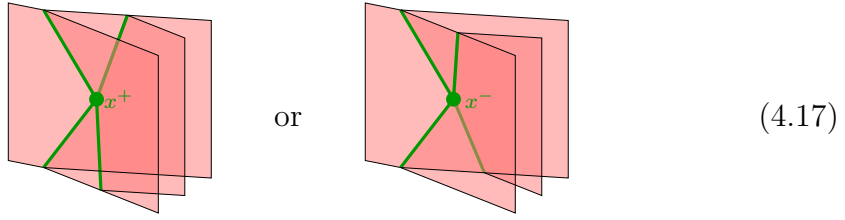
\be 
\tikzzbox{\begin{tikzpicture}[thick,scale=2.321,color=blue!50!black, baseline=0.0cm, >=stealth, 
	style={x={(-0.6cm,-0.4cm)},y={(1cm,-0.2cm)},z={(0cm,0.9cm)}}]
	\pgfmathsetmacro{\yy}{0.2}
	\coordinate (P) at (0.5, \yy, 0);
	\coordinate (R) at (0.625, 0.5 + \yy/2, 0);
	\coordinate (L) at (0.5, 0, 0);
	\coordinate (R1) at (0.25, 1, 0);
	\coordinate (R2) at (0.5, 1, 0);
	\coordinate (R3) at (0.75, 1, 0);
	\coordinate (Pt) at (0.5, \yy, 1);
	\coordinate (Rt) at (0.375, 0.5 + \yy/2, 1);
	\coordinate (Lt) at (0.5, 0, 1);
	\coordinate (R1t) at (0.25, 1, 1);
	\coordinate (R2t) at (0.5, 1, 1);
	\coordinate (R3t) at (0.75, 1, 1);
	\coordinate (alpha) at (0.5, 0.5, 0.5);
	%
	\draw[green!60!black, very thick] (alpha) -- (Rt);
	\fill [red!50,opacity=0.545] (L) -- (P) -- (alpha) -- (Pt) -- (Lt);
	\fill [red!50,opacity=0.545] (Pt) -- (Rt) -- (alpha);
	\fill [red!50,opacity=0.545] (Rt) -- (R1t) -- (R1) -- (P) -- (alpha);
	\fill [red!50,opacity=0.545] (Rt) -- (R2t) -- (R2) -- (R) -- (alpha);
	\draw[green!60!black, very thick] (alpha) -- (Rt);
	\fill [red!50,opacity=0.545] (Pt) -- (R3t) -- (R3) -- (R) -- (alpha);
	\fill [red!50,opacity=0.545] (P) -- (R) -- (alpha);
	%
	\draw[green!60!black, very thick] (P) -- (alpha);
	\draw[green!60!black, very thick] (R) -- (alpha);
	\draw[green!60!black, very thick] (alpha) -- (Pt);
	%
	\fill[color=green!60!black] (alpha) circle (1.2pt) node[right] (0up) { {\scriptsize$\!x^+$} };
	%
	\draw [black,opacity=1, very thin] (Pt) -- (Lt) -- (L) -- (P);
	\draw [black,opacity=1, very thin] (Pt) -- (Rt);
	\draw [black,opacity=1, very thin] (Rt) -- (R1t) -- (R1) -- (P);
	\draw [black,opacity=1, very thin] (Rt) -- (R2t) -- (R2) -- (R);
	\draw [black,opacity=1, very thin] (Pt) -- (R3t) -- (R3) -- (R);
	\draw [black,opacity=1, very thin] (P) -- (R);
	\end{tikzpicture}}
\qquad\text{or}\qquad
\tikzzbox{\begin{tikzpicture}[thick,scale=2.321,color=blue!50!black, baseline=0.0cm, >=stealth, 
	style={x={(-0.6cm,-0.4cm)},y={(1cm,-0.2cm)},z={(0cm,0.9cm)}}]
	\pgfmathsetmacro{\yy}{0.2}
	\coordinate (P) at (0.5, \yy, 0);
	\coordinate (R) at (0.375, 0.5 + \yy/2, 0);
	\coordinate (L) at (0.5, 0, 0);
	\coordinate (R1) at (0.25, 1, 0);
	\coordinate (R2) at (0.5, 1, 0);
	\coordinate (R3) at (0.75, 1, 0);
	\coordinate (Pt) at (0.5, \yy, 1);
	\coordinate (Rt) at (0.625, 0.5 + \yy/2, 1);
	\coordinate (Lt) at (0.5, 0, 1);
	\coordinate (R1t) at (0.25, 1, 1);
	\coordinate (R2t) at (0.5, 1, 1);
	\coordinate (R3t) at (0.75, 1, 1);
	\coordinate (alpha) at (0.5, 0.5, 0.5);
	%
	\draw[green!60!black, very thick] (alpha) -- (Rt);
	\fill [red!50,opacity=0.545] (L) -- (P) -- (alpha) -- (Pt) -- (Lt);
	\fill [red!50,opacity=0.545] (Pt) -- (R1t) -- (R1) -- (R) -- (alpha);
	\fill [red!50,opacity=0.545] (Rt) -- (R2t) -- (R2) -- (R) -- (alpha);
	\fill [red!50,opacity=0.545] (Pt) -- (Rt) -- (alpha);
	\fill [red!50,opacity=0.545] (P) -- (R) -- (alpha);
	\draw[green!60!black, very thick] (R) -- (alpha);
	\fill [red!50,opacity=0.545] (Rt) -- (R3t) -- (R3) -- (P) -- (alpha);
	%
	\draw[green!60!black, very thick] (alpha) -- (Rt);
	%
	\draw[green!60!black, very thick] (P) -- (alpha);
	\draw[green!60!black, very thick] (alpha) -- (Pt);
	%
	\fill[color=green!60!black] (alpha) circle (1.2pt) node[right] (0up) { {\scriptsize$x^{-}$} };
	%
	\draw [black,opacity=1, very thin] (Pt) -- (Lt) -- (L) -- (P) ;
	\draw [black,opacity=1, very thin] (Pt) -- (R1t) -- (R1) -- (R);
	\draw [black,opacity=1, very thin] (Rt) -- (R2t) -- (R2) -- (R);
	\draw [black,opacity=1, very thin] (Pt) -- (Rt);
	\draw [black,opacity=1, very thin] (P) -- (R);
	\draw [black,opacity=1, very thin] (Rt) -- (R3t) -- (R3) -- (P);
	\end{tikzpicture}}
\ee 
where only the 2-dimensional regions are oriented (all inducing the counter-clockwise orientation in the paper plane). 
On the boundary $\partial B_{x} \cong S^{2}$ of a small ball~$B_{x}$ around such vertices, oriented as induced from $M \setminus B_{x}$ according to the construction of \cite{TVire}, the above stratifications induce the oriented graphs 
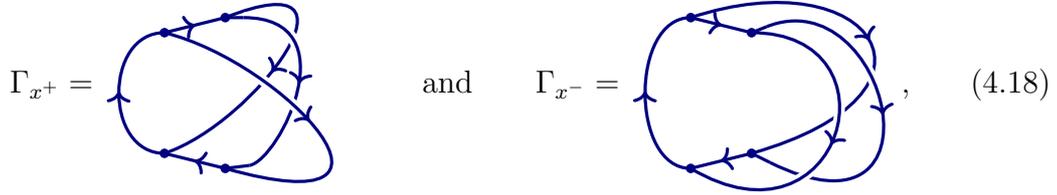
\be 
\Gamma_{x^+} = 
\tikzzbox{\begin{tikzpicture}[very thick,scale=0.4,color=blue!50!black, baseline=0cm]
	\coordinate (ol) at (0,2);
	\coordinate (or) at (2,2.5);
	\coordinate (ul) at (0,-2);
	\coordinate (ur) at (2,-2.5);
	\draw[directed] (ul) .. controls +(-2,0) and +(-2,0) .. (ol);
	\draw[directed] (or) .. controls +(5,2) and +(3,1) .. (ul);
	\draw[color=white, line width=4pt] (2.6,2.5) .. controls +(3.5,0) and +(1,0.2) .. (2.8,-2.4);
	\draw (or) -- (2.6,2.5);
	\draw[directed] (2.6,2.5) .. controls +(3.5,0) and +(1,0.2) .. (2.8,-2.4);
	\draw[color=white, line width=4pt] (ol) .. controls +(4,-1) and +(7,-2) .. (ur);
	\draw[directed] (ol) .. controls +(4,-1) and +(7,-2) .. (ur);
	\draw (2.8,-2.4) -- (ur);
	\draw[directed] (ol) -- (or);
	\draw[directed] (ur) -- (ul);
	\fill (ol) circle (4.5pt) node (ols) {};
	\fill (or) circle (4.5pt) node (osdr) {};
	\fill (ul) circle (4.5pt) node (udsl) {};
	\fill (ur) circle (4.5pt) node (usr) {};
	\end{tikzpicture}}
\!\!\!\!\text{and}\qquad
\Gamma_{x^-} = 
\tikzzbox{\begin{tikzpicture}[very thick,scale=0.4,color=blue!50!black, baseline=0cm]
	\coordinate (ol) at (0,2.5);
	\coordinate (or) at (2,2);
	\coordinate (ul) at (0,-2.5);
	\coordinate (ur) at (2,-2);
	\draw[directed] (ul) .. controls +(-2,0) and +(-2,0) .. (ol);
	\draw[directed] (ol) .. controls +(6,2) and +(7,2) .. (ur);
	\draw[color=white, line width=4pt] (2.5,2.2) .. controls +(4,1.6) and +(6,-3) .. (2.6,-2.3);
	\draw[directed] (2.5,2.2) .. controls +(4,1.6) and +(6,-3) .. (2.6,-2.3);
	\draw[color=white, line width=4pt] (or) .. controls +(5,0) and +(5,-3) .. (ul);
	\draw[directed] (or) .. controls +(5,0) and +(5,-3) .. (ul);
	\draw[directed] (ol) -- (or);
	\draw[directed] (ur) -- (ul);
	\draw (2.5,2.2) --(or);
	\draw (ur) --(2.6,-2.3);
	\fill (ol) circle (4.5pt) node (ols) {};
	\fill (or) circle (4.5pt) node (osdr) {};
	\fill (ul) circle (4.5pt) node (udsl) {};
	\fill (ur) circle (4.5pt) node (usr) {};
	\end{tikzpicture}}
\hspace{-0.85cm} ,
\ee
respectively. 
For every ${\Cat{S}}$-colouring~$c$ of the edges in the graphs $\Gamma_{x^\pm}$, we obtain functionals $\mathds{F}_{\Cat{S}}(\Gamma_{x^\pm}, c) \in H(\Gamma_{x^\pm}, c; S^2)^*$. 
Denoting the values of the colouring by  $i,j,k,l,m,n \in I$ we may employ the cone isomorphisms of the projective system defining $H(\Gamma_{x^\pm}, c; S^2)$ to obtain the explicit functionals (where here and below we suppress the choice of $i,j,k,l,m,n$ in the notation)
\be\label{eq:FCGinv}
\mathds{F}_{\Cat{S}}(\Gamma_{x^+}) = 
\tikzzbox{\begin{tikzpicture}[very thick,scale=1,color=blue!50!black, baseline=0cm]
	\draw[-dot-] (0,0) .. controls +(0,1) and +(0,1) .. (-1,0);
	\draw[-dot-] (-1,0) .. controls +(0,-1) and +(0,-1) .. (-2,0);
	\draw[-dot-] (-0.5,0.7) .. controls +(0,1.25) and +(0,1.25) .. (-2,0.7);
	\draw (-2,0.7) -- (-2,0);
	\draw[-dot-] (0,-0.7) .. controls +(0,-1.25) and +(0,-1.25) .. (-1.5,-0.7);
	\draw (0,-0.7) -- (0,0);
	\draw[directed] (-1.25,1.5) .. controls +(0,1) and +(0,1) .. (-2.75,1.5);
	\draw (-2.75,1.5) -- (-2.75,-1.6);
	\draw[directed] (-2.75,-1.6) .. controls +(0,-1.25) and +(0,-1.25) .. (-0.75,-1.6);
	%
	\fill (-0.85,0.11) circle (0pt) node (meet2) {{\scriptsize$n$}};
	\fill (-1.5,-1.4) circle (0pt) node (meet2) {{\scriptsize$m$}};
	\fill (-0.15,-0.5) circle (0pt) node (meet2) {{\scriptsize$l$}};
	\fill (-1.4,2.2) circle (0pt) node (meet2) {{\scriptsize$k$}};
	\fill (-1.85,0.5) circle (0pt) node (meet2) {{\scriptsize$j$}};
	\fill (-0.5,1.4) circle (0pt) node (meet2) {{\scriptsize$i$}};
	\draw (-1.25,1.6) node[white, fill=white, inner sep=4pt,draw, rounded corners=1pt] (R) {{\scriptsize$\;-\;$}};
	\draw[line width=1pt, color=black] (-1.25,1.6) node[inner sep=4pt,draw, rounded corners=1pt] (R) {{\scriptsize$\;-\;$}};
	\draw (-0.5,0.7) node[white, fill=white, inner sep=4pt,draw, rounded corners=1pt] (R) {{\scriptsize$\;-\;$}};
	\draw[line width=1pt, color=black] (-0.5,0.7) node[inner sep=4pt,draw, rounded corners=1pt] (R) {{\scriptsize$\;-\;$}};
	\draw (-1.5,-0.75) node[white, fill=white, inner sep=4pt,draw, rounded corners=1pt] (R) {{\scriptsize$\;-\;$}};
	\draw[line width=1pt, color=black] (-1.5,-0.75) node[inner sep=4pt,draw, rounded corners=1pt] (R) {{\scriptsize$\;-\;$}};
	\draw (-0.75,-1.6) node[white, fill=white, inner sep=4pt,draw, rounded corners=1pt] (R) {{\scriptsize$\;-\;$}};
	\draw[line width=1pt, color=black] (-0.75,-1.6) node[inner sep=4pt,draw, rounded corners=1pt] (R) {{\scriptsize$\;-\;$}};
	\end{tikzpicture}}
\; , \qquad 
\mathds{F}_{\Cat{S}}(\Gamma_{x^-}) = 
\tikzzbox{\begin{tikzpicture}[very thick,scale=1,color=blue!50!black, baseline=0cm]
	\draw[-dot-] (0,0) .. controls +(0,1) and +(0,1) .. (1,0);
	\draw[-dot-] (1,0) .. controls +(0,-1) and +(0,-1) .. (2,0);
	\draw[-dot-] (0.5,0.7) .. controls +(0,1.25) and +(0,1.25) .. (2,0.7);
	\draw (2,0.7) -- (2,0);
	\draw[-dot-] (0,-0.7) .. controls +(0,-1.25) and +(0,-1.25) .. (1.5,-0.7);
	\draw (0,-0.7) -- (0,0);
	\draw[directed] (1.25,1.5) .. controls +(0,1.25) and +(0,1.25) .. (-0.75,1.5);
	\draw (-0.75,1.5) -- (-0.75,-1.6);
	\draw[directed] (-0.75,-1.6) .. controls +(0,-1) and +(0,-1) .. (0.75,-1.6);
	%
	\fill (0.85,0-0.11) circle (0pt) node (meet2) {{\scriptsize$n$}};
	\fill (1.5,-1.4) circle (0pt) node (meet2) {{\scriptsize$m$}};
	\fill (0.15,-0.5) circle (0pt) node (meet2) {{\scriptsize$l$}};
	\fill (1.1,2.3) circle (0pt) node (meet2) {{\scriptsize$k$}};
	\fill (1.85,0.5) circle (0pt) node (meet2) {{\scriptsize$j$}};
	\fill (0.5,1.4) circle (0pt) node (meet2) {{\scriptsize$i$}};
	\draw (1.25,1.6) node[white, fill=white, inner sep=4pt,draw, rounded corners=1pt] (R) {{\scriptsize$\;-\;$}};
	\draw[line width=1pt, color=black] (1.25,1.6) node[inner sep=4pt,draw, rounded corners=1pt] (R) {{\scriptsize$\;-\;$}};
	\draw (0.5,0.7) node[white, fill=white, inner sep=4pt,draw, rounded corners=1pt] (R) {{\scriptsize$\;-\;$}};
	\draw[line width=1pt, color=black] (0.5,0.7) node[inner sep=4pt,draw, rounded corners=1pt] (R) {{\scriptsize$\;-\;$}};
	\draw (1.5,-0.75) node[white, fill=white, inner sep=4pt,draw, rounded corners=1pt] (R) {{\scriptsize$\;-\;$}};
	\draw[line width=1pt, color=black] (1.5,-0.75) node[inner sep=4pt,draw, rounded corners=1pt] (R) {{\scriptsize$\;-\;$}};
	\draw (0.75,-1.6) node[white, fill=white, inner sep=4pt,draw, rounded corners=1pt] (R) {{\scriptsize$\;-\;$}};
	\draw[line width=1pt, color=black] (0.75,-1.6) node[inner sep=4pt,draw, rounded corners=1pt] (R) {{\scriptsize$\;-\;$}};
\end{tikzpicture}}
\ee 
which are respectively elements of 
\begin{align}
& \big( \!\Hom_{\Cat{S}}(j\otimes i, k) \otimes_\Bbbk \Hom_{\Cat{S}}(n\otimes l,i) \otimes_\Bbbk \Hom_{\Cat{S}}(m,j\otimes n) \otimes_\Bbbk \Hom_{\Cat{S}}(k,m\otimes l) \big)^* \, , 
\nonumber
\\ 
& \big( \!\Hom_{\Cat{S}}(i\otimes j, k) \otimes_\Bbbk \Hom_{\Cat{S}}(l\otimes n,i) \otimes_\Bbbk \Hom_{\Cat{S}}(m,n\otimes j) \otimes_\Bbbk \Hom_{\Cat{S}}(k,l\otimes m) \big)^* \, . 
\end{align}

The final ingredient for the Turaev--Viro invariant in the construction of \cite{TVire} are the contraction maps~$*_e$. 
We describe them for the case of interest to us: 
Let $M,P$ be as before. 
Choose an ${\Cat{S}}$-colouring $c\colon P_2 \to I$ and decorate each region $r\in P_2$ with the object $c(r)$. 
Hence every internal edge~$e$ of~$P$ looks like
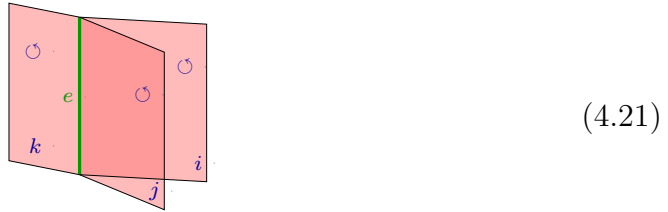
\be\label{eq:unorie}
\tikzzbox{\begin{tikzpicture}[thick,scale=2.321,color=blue!50!black, baseline=0.0cm, >=stealth, 
	style={x={(-0.6cm,-0.4cm)},y={(1cm,-0.2cm)},z={(0cm,0.9cm)}}]
	\pgfmathsetmacro{\yy}{0.2}
	\coordinate (T) at (0.5, 0.4, 0);
	\coordinate (L) at (0.5, 0, 0);
	\coordinate (R1) at (0.3, 1, 0);
	\coordinate (R2) at (0.7, 1, 0);
	\coordinate (1T) at (0.5, 0.4, 1);
	\coordinate (1L) at (0.5, 0, 1);
	\coordinate (1R1) at (0.3, 1, );
	\coordinate (1R2) at (0.7, 1, );
	%
	\coordinate (p3) at (0.1, 0.1, 0.5);
	\coordinate (p2) at (0.5, 0.95, 0.5);
	\coordinate (p1) at (0.9, 0.1, 0.5);
	%
	\fill [red!50,opacity=0.545] (L) -- (T) -- (1T) -- (1L);
	\fill [red!50,opacity=0.545] (R1) -- (T) -- (1T) -- (1R1);
	\fill [red!50,opacity=0.545] (R2) -- (T) -- (1T) -- (1R2);
	\fill[color=blue!60!black] (0.5,0.25,0.15) circle (0pt) node[left] (0up) { {\scriptsize$k$} };
	\fill[color=blue!60!black] (0.15,0.95,0.04) circle (0pt) node[left] (0up) { {\scriptsize$i$} };
	\fill[color=blue!60!black] (0.55,0.95,0.04) circle (0pt) node[left] (0up) { {\scriptsize$j$} };
	\fill[color=blue!60!black] (0.5,0.25,0.75) circle (0pt) node[left] (0up) { {\scriptsize$\circlearrowleft$} };
	\fill[color=blue!60!black] (0.15,0.9,0.64) circle (0pt) node[left] (0up) { {\scriptsize$\circlearrowleft$} };
	\fill[color=blue!60!black] (0.55,0.9,0.64) circle (0pt) node[left] (0up) { {\scriptsize$\circlearrowleft$} };
	%
	\draw[green!60!black, very thick] (T) -- (1T);
	\fill[color=green!60!black] (0.5,0.43,0.5) circle (0pt) node[left] (0up) { {\scriptsize$e$} };
	%
	\draw [black,opacity=1, very thin] (1T) -- (1L) -- (L) -- (T);
	\draw [black,opacity=1, very thin] (1T) -- (1R1) -- (R1) -- (T);
	\draw [black,opacity=1, very thin] (1T) -- (1R2) -- (R2) -- (T);
	\end{tikzpicture}}
\ee
for some $i,j,k\in I$. 
Recall that regions of~$P$ are oriented, but edges are not. 
The region coloured by~$k$ in~\eqref{eq:unorie} induces an orientation on the edge~$e$
(upwards in the diagram in our convention), and we denote the corresponding \textsl{oriented} edge by~$e^+$. The oppositely oriented edge is denoted~$e^-$.
To these the construction of \cite{TVire} associates vector spaces 
\be 
H_c(e^+) \cong \Hom_{\Cat{S}}(i\otimes j,k) \, , 
\qquad 
H_c(e^-) \cong \Hom_{\Cat{S}}(k, i\otimes j) \, .
\ee 
Since the pairing $H_c(e^+) \otimes_\Bbbk H_c(e^-) \to \Bbbk$ of~\eqref{eq:tracepairing} is nondegenerate, there is a unique dual copairing $\gamma\colon \Bbbk \to H_c(e^-) \otimes_\Bbbk H_c(e^+)$ given by $\gamma(1) = \sum_\lambda \widehat{\lambda} \otimes \lambda$. 
The \textsl{contraction map}  
$*_e \colon H_c(e^+)^* \otimes_\Bbbk H_c(e^-)^*  \to \Bbbk $
is defined to be the dual map~$\gamma^*$ composed with the canonical isomorphisms $\Bbbk \cong \Bbbk^*$ and $(V\otimes_\Bbbk W)^*\cong W^* \otimes_\Bbbk V^*$ for all $V,W \in \text{vect}_\Bbbk$. 
Thus we have (for basis elements $\lambda,\mu \in \Hom_{\Cat{S}}(i\otimes j,k)$) 
\begin{align}\label{eq:starcontract}
*_e \colon H_c(e^+)^* \otimes_\Bbbk H_c(e^-)^*  \lra \Bbbk 
\, , \qquad 
\lambda^*\otimes \widehat{\mu}^* \lmt \delta_{\lambda,\mu} \, ,
\end{align}
where we use 
\be\label{eq:lstar}
\lambda^* := 
\tikzzbox{\begin{tikzpicture}[very thick,scale=0.6,color=blue!50!black, baseline=-0.1cm]
	\draw[-dot-] (0,0) .. controls +(0,1) and +(0,1) .. (1,0);
	\draw[-dot-] (1,0) .. controls +(0,-1) and +(0,-1) .. (0,0);
	\draw (0.5,0.7) -- (0.5,1.2);
	\draw[directed] (0.5,1.2) .. controls +(0,1) and +(0,1) .. (-1,1.2);
	\draw (-1,1.2) -- (-1,-1.2);
	\draw[directed] (-1,-1.2) .. controls +(0,-1) and +(0,-1) .. (0.5,-1.2);
	\draw (0.5,-0.7) -- (0.5,-1.2);
	%
	\fill (0.5,-0.3) circle (0pt) node (meet2) {{\scriptsize$\widehat{\lambda}$}};
	\draw (0.5,0.7) node[white, fill=white, inner sep=4pt,draw, rounded corners=1pt] (R) {{\scriptsize$\;-\;$}};
	\draw[line width=1pt, color=black] (0.5,0.7) node[inner sep=4pt,draw, rounded corners=1pt] (R) {{\scriptsize$\;-\;$}};
	\end{tikzpicture}}
\, , \qquad 
\widehat{\mu}^* := 
\tikzzbox{\begin{tikzpicture}[very thick,scale=0.6,color=blue!50!black, baseline=-0.1cm]
	\draw[-dot-] (0,0) .. controls +(0,1) and +(0,1) .. (1,0);
	\draw[-dot-] (1,0) .. controls +(0,-1) and +(0,-1) .. (0,0);
	\draw (0.5,0.7) -- (0.5,1.2);
	\draw[directed] (0.5,1.2) .. controls +(0,1) and +(0,1) .. (-1,1.2);
	\draw (-1,1.2) -- (-1,-1.2);
	\draw[directed] (-1,-1.2) .. controls +(0,-1) and +(0,-1) .. (0.5,-1.2);
	\draw (0.5,-0.7) -- (0.5,-1.2);
	%
	\fill (0.5,0.3) circle (0pt) node (meet2) {{\scriptsize$\mu$}};
	\draw (0.5,-0.7) node[white, fill=white, inner sep=4pt,draw, rounded corners=1pt] (R) {{\scriptsize$\;-\;$}};
	\draw[line width=1pt, color=black] (0.5,-0.7) node[inner sep=4pt,draw, rounded corners=1pt] (R) {{\scriptsize$\;-\;$}};
	\end{tikzpicture}}
\, . 
\ee

We can now describe the Turaev--Viro invariants $\zz_\cat^{\text{TV}}(M) \in \Bbbk$ for closed 3-manifolds~$M$ \cite{TVire}: 
\be\label{eq:ZTVclosed}
\zz^{\text{TV},\Cat{S}}(M) = 
\big(\! \dim({\Cat{S}})\big)^{-|P_3|} 
	\sum_{c\colon P_2\to I} \Big( \prod_{r\in P_2} d_{c(r)}^{\chi(r)} \Big)
	\Big( \bigotimes_{e\in P_1} *_e \Big) \Big( \bigotimes_{x\in P_0} \mathds{F}_{\Cat{S}}(\Gamma_x) \Big)
\ee
for any oriented stratified 2-polyhedron~$P$ of~$M$, where~$P_3$ denotes the set of connected components in $M\setminus P$ and 
$\chi(r)$ denotes the Euler characteristic of the 2-stratum~$r$. 
Note that for every edge $e\in P_1$ its two oriented versions $e^+,e^-$ correspond to precisely two tensor factors in $\bigotimes_{x\in P_0} \mathds{F}_{\Cat{S}}(\Gamma_x)$, since every edge has two endpoints in~$P_0$, and for every $x\in P_0$ its incident edges are all treated as outgoing in the definition of $\mathds{F}_{\Cat{S}}(\Gamma_x)$, 

On objects and arbitrary morphisms the functor $\zz^{\text{TV},\Cat{S}}$ is defined  along the same lines. 
For a surface~$\Sigma$ with an embedded oriented graph~$\Gamma$ as above, we set 
\be
\big| \Gamma; \Sigma \big|^0 = \bigoplus_{c\colon \Gamma_1\to I} H(\Gamma, c; \Sigma)
\ee
with $H(\Gamma, c; \Sigma)$ as in~\eqref{eq:HGcS}. 
For a bordism $M\colon \emptyset \to \Sigma$ we choose an extension of the graph~$\Gamma$ to an oriented stratified 2-polyhedron~$P$ of~$M$. 
Let $\text{Col}(P,c)$ be the set of ${\Cat{S}}$-colourings $\widetilde{c}\colon P_2 \to I$ with $\widetilde{c}(r_e) = c(e)$ for the region~$r_e$ with $e \subset r_e\cap \partial M$ and all edges~$e$ in~$\Gamma$.  
We write $P_1^{\text{int}}$ for the set of edges in~$P$ without endpoints in $\partial M$, and $P_1^\partial$ for the set of edges in~$P$ with precisely one endpoint in $\partial M$. 
Such endpoints correspond to vertices $v\in \Gamma_0$ and we have 
\be 
\bigotimes_{e\in P_1^\partial} H_{\widetilde{c}}(e^{\text{out}})^* 
\cong
H(\Gamma^{\text{op}}, c; \Sigma^{\text{op}}) 
\ee
where $e^{\text{out}}$ denotes the edge~$e$ with orientation towards $\partial M$. 
Thus by contracting along interior edges $e\in P_1^{\text{int}}$ we obtain a vector 
\be 
\big| M; \Gamma, c \big|^0 
= 
\big(\! \dim({\Cat{S}})\big)^{-|P_3|} 
\sum_{\widetilde{c}\in \text{Col}(P,c)} \Big( \prod_{r\in P_2} d_{\widetilde{c}(r)}^{\chi(r)} \Big)
\Big( \bigotimes_{e\in P_1^{\text{int}}} *_e \Big) \Big( \bigotimes_{x\in P_0} \mathds{F}_{\Cat{S}}(\Gamma_x) \Big)
\ee 
in $H(\Gamma^{\text{op}}, c; \Sigma^{\text{op}})^* \cong H(\Gamma, c; \Sigma)$, generalising~\eqref{eq:ZTVclosed}. 

If $(\Sigma,\Gamma) = (\Sigma'^{\text{op}},\Gamma'^{\text{op}}) \sqcup (\Sigma'', \Gamma'')$ we can view~$M$ as a bordism $\Sigma' \to \Sigma''$. 
Writing
\be 
	 \Upsilon \colon H(\Gamma', c; \Sigma')^* \otimes_\Bbbk H(\Gamma''^{\text{op}}, c; \Sigma''^{\text{op}})^*
\stackrel{\cong}{\lra} 
	 \Hom_\Bbbk\big( H(\Gamma', c; \Sigma'), H(\Gamma'', c; \Sigma'') \big) 
\ee 
for the canonical isomorphism, we obtain a linear map 
\be
\label{eq:lin-mapTV}
\big| M; \Sigma', \Gamma', \Sigma'', \Gamma'', c \big|^0 
= 
\frac{\big(\! \dim({\Cat{S}})\big)^{|\Sigma''\setminus \Gamma''|}}{\prod_{e\in\Gamma''_1} d_{c(e)}} 
	\cdot \Upsilon\Big( \big| M; \Gamma, c\big|^0\Big) \, ,
\ee 
where $|\Sigma''\setminus \Gamma''|$ denotes the number of components of $\Sigma''\setminus \Gamma''$.
Restricting to $\Sigma' = \Sigma''$ and the cylinder $M=\Sigma' \times [0,1]$, by summing over all ${\Cat{S}}$-colourings $c\colon (\Gamma' \sqcup \Gamma'')_1 \to I$, we obtain a projective system 
\be
\label{eq:proj-sys-TV}
p(\Gamma', \Gamma'') \colon \big| \Gamma', \Sigma' \big|^0 \lra \big|\Gamma'', \Sigma'\big|^0 \, . 
\ee 
Then by definition 
\be
\zz^{\text{TV},\Cat{S}}(\Sigma') = \varprojlim  \, p(\Gamma', \Gamma'')
\ee
and $\zz^{\text{TV},\Cat{S}}$ acts on arbitrary bordism classes as the induced linear maps.

\subsection{Orbifold data for the trivial Reshetikhin--Turaev theory}
\label{subsec:orbidatatrivial}

By the (3-dimensional) \textsl{trivial defect TQFT} $\zz^{\text{triv}} \colon \Bord_3(\D^{\text{triv}}) \to \Vectk$ we mean the Reshetikhin--Turaev defect TQFT constructed (in Section~\ref{sec:resh-tura-defect}) from the ``trivial'' modular tensor category $\Vectk$: 
\be
\zz^{\text{triv}} := \zz^{\Vectk} \, . 
\ee
Hence $\zz^{\text{triv}}(\Sigma) = \Bbbk$ for every unstratified surface $\Sigma \in \Bord_3(\D^{\text{triv}})$, while 2- and 1-strata of bordisms in $\Bord_3(\D^{\text{triv}})$ are labelled by $\Delta$-separable symmetric Frobenius $\Bbbk$-algebras and their cyclic modules in $\Vectk$, respectively. 
In this section we will construct orbifold data for $\zz^{\text{triv}}$: 

\begin{proposition}
\label{prop:orbidata}
	 Given a spherical fusion category~${\Cat{S}}$, the following is an orbifold datum for $\zz^{\text{{\normalfont triv}}}$, denoted~$\Aca$: 
\begin{align}
\Cat{C} & := \Vectk \, , 
\\
A & := \bigoplus_{i\in I} \Bbbk 
	\qquad (\text{direct sum of trivial Frobenius algebras } \Bbbk) \, , 
\\
T & := \bigoplus_{i,j,k\in I} \Hom_{\Cat{S}}(i\otimes j,k) \, , 
\\ 
\alpha & \colon \lambda \otimes \mu \lmt \sum_{d,\lambda',\mu'} d_d^{-1} F^{\lambda\lambda'}_{\mu\mu'} \cdot \lambda' \otimes \mu' \, , 
	\label{eq:A0plussum}
\\
\bar\alpha & \colon \lambda' \otimes \mu' \lmt \sum_{c,\lambda'',\mu''} d_c^{-1} (F^{-1})^{\lambda'\lambda''}_{\mu'\mu''} \cdot \lambda'' \otimes \mu'' \, , 
	\label{eq:A0minussum}
\\
\psi^2 & := \text{diag}(d_1,d_2,\dots,d_{|I|})
\qquad (\text{$\psi$ is a choice of square root}) \, , 
\label{eq:psi-TVorb}
  \\
	 \phi^2 & := \Big( \sum_{i\in I} d_i^2 \Big)^{-1} = (\dim\,\Cat{S})^{-1} \qquad (\text{$\phi$ is a choice of square root}) \, , 
\end{align}
where the basis elements and sums in~\eqref{eq:A0plussum} and~\eqref{eq:A0minussum} are as in~\eqref{eq:Fsummation} and~\eqref{eq:Finvsummation} while $\alpha(\lambda\otimes\mu) \stackrel{\text{def}}{=} 0 \stackrel{\text{def}}{=} \bar\alpha(\lambda'\otimes\mu')$ if $\lambda,\mu$ and $\lambda',\mu'$ are not compatible as in ~\eqref{eq:Fsummation} and~\eqref{eq:Finvsummation}, 
respectively, 
and the action of~$A$ on $\Hom_{\Cat{S}}(i\otimes j,k)$ in~$T$ is such that only the $k$-th summand~$\Bbbk_k$ acts non-trivially from the left, and only $\Bbbk_i \otimes \Bbbk_j$ acts non-trivially from the right.
Different choices of square root in~\eqref{eq:psi-TVorb} give equivalent orbifold data in the sense of Definition~\ref{def:Moritatransport}.
\end{proposition}

\medskip 

As preparation for the proof of Proposition~\ref{prop:orbidata} we spell out composition and adjunctions for 2-morphisms in $\mathcal{T}_{\zz^{\text{triv}}}$. 
Using the isomorphism 
\be
\Hom_{\Cat{S}}(k, i\otimes j) \stackrel{\cong}{\lra} \Hom_{\Cat{S}} (i\otimes j,k)^*
\, , \qquad 
\widehat{\lambda} \lmt 
\tikzzbox{\begin{tikzpicture}[very thick,scale=0.6,color=blue!50!black, baseline=-0.1cm]
	\draw[-dot-] (0,0) .. controls +(0,1) and +(0,1) .. (1,0);
	\draw[-dot-] (1,0) .. controls +(0,-1) and +(0,-1) .. (0,0);
	\draw (0.5,0.7) -- (0.5,1.2);
	\draw[directed] (0.5,1.2) .. controls +(0,1) and +(0,1) .. (-1,1.2);
	\draw (-1,1.2) -- (-1,-1.2);
	\draw[directed] (-1,-1.2) .. controls +(0,-1) and +(0,-1) .. (0.5,-1.2);
	\draw (0.5,-0.7) -- (0.5,-1.2);
	%
	\fill (0.5,-0.3) circle (0pt) node (meet2) {{\scriptsize$\widehat{\lambda}$}};
	\draw (0.5,0.7) node[white, fill=white, inner sep=4pt,draw, rounded corners=1pt] (R) {{\scriptsize$\;-\;$}};
	\draw[line width=1pt, color=black] (0.5,0.7) node[inner sep=4pt,draw, rounded corners=1pt] (R) {{\scriptsize$\;-\;$}};
	\end{tikzpicture}}
\ee
we exhibit the $(A\otimes A)$-$A$-bimodule 
\be
{T}^\dagger := \bigoplus_{i,j,k\in I} \Hom_{\Cat{S}}(k,i\otimes j)
\ee
as the adjoint of~${T}$ via the maps 
\begin{align}
\text{ev}_{{T}} \colon {T}^\dagger \otimes_{{A}} {T} = \bigoplus_{a,b,i,j,k} \Hom_{\Cat{S}}(k,a\otimes b) \otimes_\Bbbk \Hom_{\Cat{S}}(i\otimes j,k) 
	& \lra {A} \otimes_\Bbbk {A}
	\nonumber
\\
\Hom_{\Cat{S}}(k,a\otimes b) \otimes_\Bbbk \Hom_{\Cat{S}}(i\otimes j,k) \;\ni\; \widehat{\mu} \otimes \lambda 
	& \lmt \delta_{a,i} \, \delta_{b,j} \, \delta_{\lambda,\mu} \cdot 1_i \otimes 1_j
	\nonumber
\end{align}
and 
\begin{align}
\text{coev}_{{T}} \colon {A} 
	& \lra {T} \otimes_{{A} \otimes_\Bbbk {A}} {T}^\dagger = \bigoplus_{i,j,k} \Hom_{\Cat{S}}(i\otimes j,k) \otimes_\Bbbk \Hom_{\Cat{S}}(k,i\otimes j) 
\\ 
1_k 
	& \lmt \sum_{i,j, \lambda} \lambda \otimes \widehat{\lambda}
\end{align}
where $1_k$ denotes $1\in \Bbbk$ in the $k$-th copy of~$\Bbbk$ in~${A}$. 
Note how tensor products over the direct sum algebra~${A}$ turn into tensor products over~$\Bbbk$ of matching summands. 
Similarly, we have adjunction maps 
\begin{align}
\widetilde{\text{ev}}_{{T}} \colon {T} \otimes_{{A} \otimes_\Bbbk {A}} {T}^\dagger 
	 \lra {A} \, , 
\qquad
\widetilde{\text{coev}}_{{T}} \colon {A} \otimes_\Bbbk {A} 
	 \lra {T}^\dagger \otimes_{{A}} {T} \, . 
\end{align}

\medskip 

\begin{proof}[Proof of Proposition~\ref{prop:orbidata}]
We will show that the data $(\Vectk, {A}, {T}, \alpha, \bar\alpha, \psi, \phi)$ satisfy the constraints in Proposition~\ref{prop:internal-orb}.

The constraints \eqref{eq:347n} and \eqref{eq:348n} reduce to the pentagon axiom for~$\Cat{S}$ (expressed in terms of $F$-symbols) and to the fact that up to dimension factors, $\alpha$ is the inverse of~$\bar\alpha$. 
For example, writing~$c$ for the symmetric braiding of $\Vect$, the left-hand side of \eqref{eq:347n} becomes
\be
\begin{tikzpicture}[
baseline=(current bounding box.base), 
>=stealth,
descr/.style={fill=white,inner sep=3.5pt}, 
normal line/.style={->}
] 
\matrix (m) [matrix of math nodes, row sep=4.0em, column sep=4.0em, text height=1.8ex, text depth=0.1ex] {%
	{T \otimes T \otimes T}  
	& 
	\displaystyle{\sum_{x',y',\lambda',\mu',\nu'} d_{x'}^{-1} d_{y'}^{-1} 
		\sum_\delta F_{\mu\nu'}^{\lambda\delta} F_{\nu\mu'}^{\delta\lambda'} \cdot
		\lambda' \otimes \mu' \otimes \nu'}
	\\
	{T \otimes T \otimes T}  
	&
	\displaystyle{\sum_{x',\nu',\delta} d_{x'}^{-1} F^{\lambda\delta}_{\mu\nu'}
		\cdot \delta \otimes \nu \otimes \nu'} 
	\\
	{T \otimes T \otimes T}
	&
	\displaystyle{\sum_{x',\nu',\delta} d_{x'}^{-1} F^{\lambda\delta}_{\mu\nu'}
		\cdot \delta \otimes \nu' \otimes \nu} 
	\\
	{T \otimes T \otimes T}  
	& 
	\lambda \otimes \mu \otimes \nu  
	\\
};
\path[font=\footnotesize] (m-4-1) edge[->] node[auto] {$\alpha \otimes \id$} (m-3-1);
\path[font=\footnotesize] (m-3-1) edge[->] node[auto] {$\id_T \otimes c_{T,T}^{-1}$} (m-2-1);
\path[font=\footnotesize] (m-2-1) edge[->] node[auto] {$\alpha \otimes \id$} (m-1-1);
\path[font=\footnotesize] (m-4-2) edge[|->] node[auto] {} (m-3-2);
\path[font=\footnotesize] (m-3-2) edge[|->] node[auto] {} (m-2-2);
\path[font=\footnotesize] (m-2-2) edge[|->] node[auto] {} (m-1-2);
\end{tikzpicture}
\ee
which is the left-hand side of \eqref{eq:pentagon-two-bases-2}, up to the dimension factors which cancel against corresponding factors on the right-hand side of \eqref{eq:347n} together with the factors coming from $\psi^{2}$.

\begin{figure}[!htbp]
	\begin{center} 
		$$
		\begin{tikzpicture}[very thick, x=2cm, y=5cm, scale=0.75, color=blue!50!black, baseline=0.75cm]
		\draw (-0.15,2) -- (-0.15,3.5); 
		\draw[directed] (0.15,2.0) .. controls +(0,1) and +(0,1) .. (1.25,2.0);
		\draw[redirected] (0.15,2) .. controls +(0,-1) and +(0,-1) .. (1.75,2);
		\draw (1.75,2) -- (1.75,3.5);
		\draw (-0.15,0) -- (-0.15,2); 
		\draw[redirected] (0.15,0) .. controls +(0,-1) and +(0,-1) .. (1.25,0);
		\draw[directed] (0.15,0) .. controls +(0,1) and +(0,1) .. (1.75,0);
		\draw (1.75,0) -- (1.75,-1.5);
		\draw[color=white, line width=4pt] (1.25,0) -- (1.25,2);
		\draw (1.25,0) -- (1.25,2);
		\draw (-0.15,-1.5) -- (-0.15,0); 
		%
		%
		\fill[color=white] (0,0) node[inner sep=5pt,draw, rounded corners=1pt, fill, color=white] (R2) {{\scriptsize$\;\alpha\;$}};
		\draw[line width=1pt, color=black] (0,0) node[inner sep=5pt, draw, semithick, rounded corners=1pt] (R) {{\scriptsize$\;\alpha\;$}};
		%
		%
		\fill[color=white] (0,2) node[inner sep=5pt,draw, rounded corners=1pt, fill, color=white] (R2) {{\scriptsize$\;\alpha\;$}};
		\draw[line width=1pt, color=black] (0,2) node[inner sep=5pt, draw, semithick, rounded corners=1pt] (R) {{\scriptsize$\;\bar\alpha\;$}};
		%
		%
		\fill[color=black] (-0.15,0.75) circle (2.9pt) node[left] (meet) {{\scriptsize$\!\psi_2^2\!$}};
		%
		\fill (0.15,3.5) circle (0pt) node (M) {{\scriptsize${}_l T_{a,k}$}};
		\fill (1.35,3.5) circle (0pt) node (M) {{\scriptsize$({}_iT_{a,b})^\dagger$}};
		\fill (0.5,2.3) circle (0pt) node (M) {{\scriptsize${}_{z'} T_{b,x}$}};
		\fill (0.5,-0.3) circle (0pt) node (M) {{\scriptsize${}_k T_{j,x}$}};
		\fill (0.85,1) circle (0pt) node (M) {{\scriptsize$({}_k T_{j,x})^\dagger$}};
		\fill (0.5,0.3) circle (0pt) node (M) {{\scriptsize${}_y T_{m,j}$}};
		\fill (0.15,-1.5) circle (0pt) node (M) {{\scriptsize${}_l T_{m,k}$}};
		\fill (1.35,-1.5) circle (0pt) node (M) {{\scriptsize$({}_iT_{m,j})^\dagger$}};
		\fill (-0.4,0.3) circle (0pt) node (M) {{\scriptsize${}_l T_{y,x}$}};
		\fill[black] (2.5,3.25) circle (0pt) node[right] (M) 
		{$\displaystyle{\sum_{x,\mu,\lambda'} \sum_{a,\nu,\lambda''}} d_x d_i^{-1} d_{k}^{-1} F_{\mu\widetilde{\lambda}}^{\lambda\lambda'} (F^{-1})_{\nu\mu}^{\lambda'\lambda''} \cdot \lambda'' \otimes \widehat{\nu}$};
		\draw[semithick, black, |->] (2.9,2.6) -- (2.9,3.05);
		\fill[black] (2.5,2.4) circle (0pt) node[right] (M) 
		{$\displaystyle{\sum_{x,\mu,\lambda'} \sum_{a,b,\nu} \sum_{z',\lambda'',\nu'}} d_x d_i^{-1} d_{z'}^{-1} F_{\mu\widetilde{\lambda}}^{\lambda\lambda'} (F^{-1})_{\nu\nu'}^{\lambda'\lambda''} \cdot \lambda'' \otimes \nu' \otimes \widehat{\mu} \otimes \widehat{\nu}$};
		\draw[semithick, black, |->] (2.9,1.8) -- (2.9,2.2);
		\fill[black] (2.5,1.6) circle (0pt) node[right] (M) 
		{$\displaystyle{\sum_{x,\mu,\lambda'} \sum_{a,b,\nu}} d_x d_i^{-1} F_{\mu\widetilde{\lambda}}^{\lambda\lambda'} \cdot \lambda' \otimes \nu \otimes \widehat{\mu} \otimes \widehat{\nu}$};
		\draw[semithick, black, |->] (2.9,1.2) -- (2.9,1.4);
		\fill[black] (2.5,1) circle (0pt) node[right] (M) 
		{$\displaystyle{\sum_{x,\mu,\lambda'}} d_x d_i^{-1} F_{\mu\widetilde{\lambda}}^{\lambda\lambda'} \cdot \lambda' \otimes \widehat{\mu}$};
		\draw[semithick, black, |->] (2.9,0.6) -- (2.9,0.8);
		\fill[black] (2.5,0.4) circle (0pt) node[right] (M) 
		{$\displaystyle{\sum_{x,\mu} \sum_{y,\lambda',\mu'}} d_y^{-1} F_{\mu\mu'}^{\lambda\lambda'} \cdot \lambda' \otimes \mu' \otimes \widehat{\mu} \otimes \widehat{\widetilde{\lambda}}$};
		\draw[semithick, black, |->] (2.9,-0.2) -- (2.9,0.2);
		\fill[black] (2.5,-0.4) circle (0pt) node[right] (M) 
		{$\displaystyle{\sum_{x,\mu}} \lambda \otimes \mu \otimes \widehat{\mu} \otimes \widehat{\widetilde{\lambda}}$};
		\draw[semithick, black, |->] (2.9,-1.15) -- (2.9,-0.6);
		\fill[black] (2.5,-1.25) circle (0pt) node[right] (M) 
		{$\lambda \otimes \widehat{\widetilde{\lambda}}$};
		\end{tikzpicture} 
		$$
	\end{center}
	\caption{Computing the left-hand side of condition~\eqref{eq:349n} for $\Aca$.
		Note that this is a string diagram in $\Vectk$, so there is no need to distinguish between over- and under-crossings, but we prefer to keep the notation from Proposition~\ref{prop:internal-orb}.} 
	\label{fig:349} 
\end{figure}
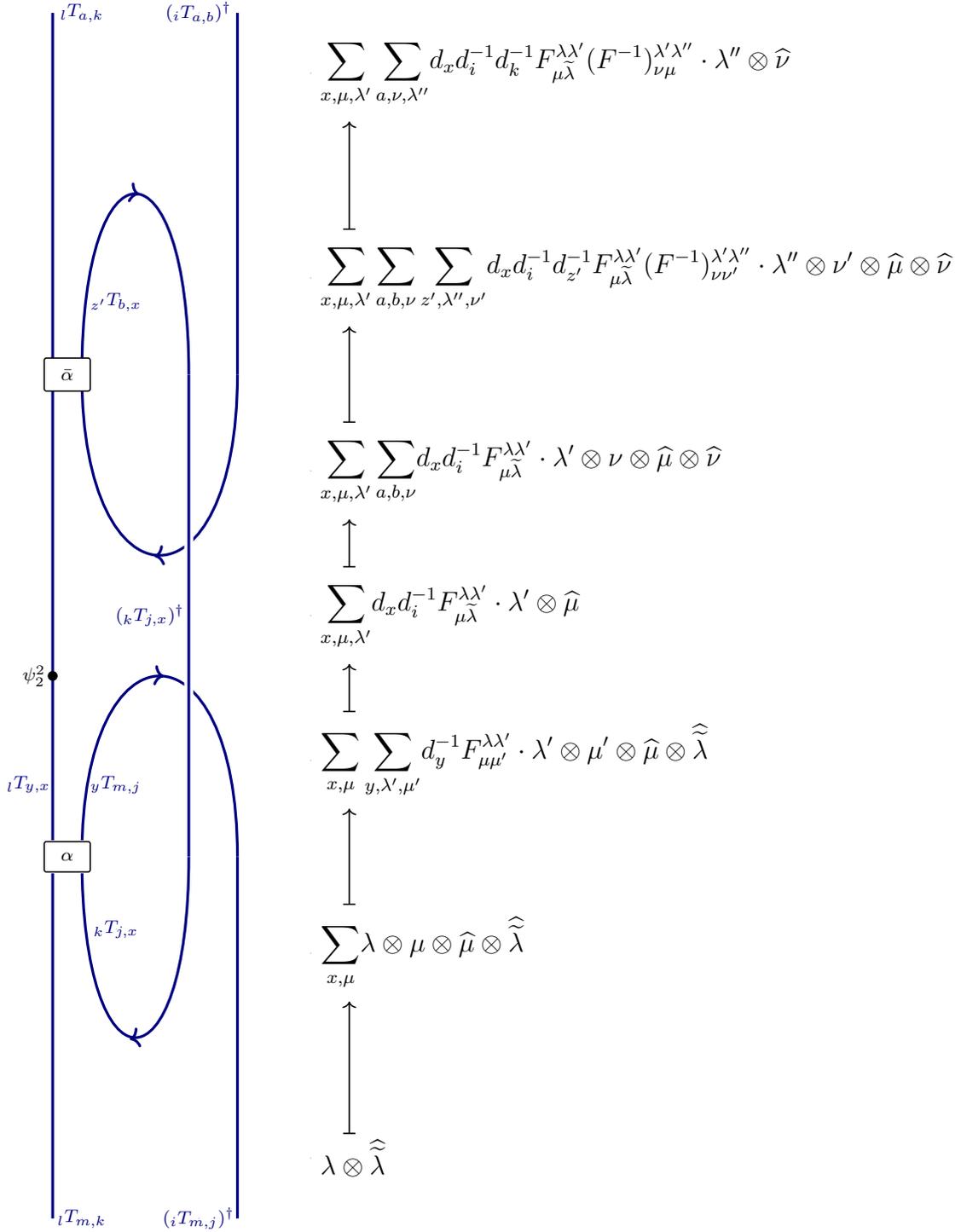

We now turn to the first condition of~\eqref{eq:349n}, which is an identity of linear maps on 
\begin{align}
{T} \otimes_{{A}} {T}^\dagger 
& = 
\bigoplus_{i,j,k,l,m} \Hom_{\Cat{S}}(m\otimes k,l) \otimes_\Bbbk \Hom_{\Cat{S}}(i,m\otimes j) 
\nonumber \\
& = 
\bigoplus_{i,j,k,l,m}
{}_l T_{m,k} \otimes_\Bbbk ({}_i T_{m,j})^\dagger
\, .
\end{align}
Here and below we use the abbreviations 
\be 
{}_k T_{i,j} := \Hom_{\Cat{S}}(i \otimes j, k) 
\, , \quad 
({}_k T_{i,j})^\dagger := \Hom_{\Cat{S}}(k, i \otimes j) \, . 
\ee 
Hence it is sufficient to show that for all fixed $i,j,k,l,m \in I$ and for all basis elements $\lambda \in \Hom_{\Cat{S}}(m\otimes k,l)$, $\widehat{\widetilde{\lambda}} \in \Hom_{\Cat{S}}(i,m\otimes j)$ the left-hand side of~\eqref{eq:349n} acts as the identity times~$d_m^{-1}$ on $\lambda\otimes \widehat{\widetilde{\lambda}}$. 
This action on $\lambda\otimes \widehat{\widetilde{\lambda}}$ is computed in Figure~\ref{fig:349}, where the Roman summation indices $a,b,x,y,z'$ range over~$I$ while Greek indices range over chosen bases elements: 
\begin{align}
\mu \in &  \Hom_{\Cat{S}}(j\otimes x, k) \, , \nonumber
\\
\lambda' \in & \Hom_{\Cat{S}}(y\otimes x, l) \, , \nonumber
\\
\mu' \in & \Hom_{\Cat{S}}(m\otimes j, y) \, , \nonumber
\\
\nu \in & \Hom_{\Cat{S}}(a\otimes b, i) \, ,  \nonumber
\\
\lambda'' \in & \Hom_{\Cat{S}}(a\otimes z', l) \, , \nonumber
\\
\nu' \in & \Hom_{\Cat{S}}(b\otimes x, z') \, . 
\end{align}
The outcome 
\be 
\sum_{x,\mu,\lambda'} \sum_{a,\nu,\lambda''} d_x \big(d_i^{-1}  F_{\mu\widetilde{\lambda}}^{\lambda\lambda'}\big) \cdot \big(d_{k}^{-1} (F^{-1})_{\nu\mu}^{\lambda'\lambda''}\big) \cdot \lambda'' \otimes \widehat{\nu}
\label{eq:TV-orb-aux1}
\ee 
of the computation in Figure~\ref{fig:349} can be further simplified:
\begin{align*}
\eqref{eq:TV-orb-aux1} \stackrel{\eqref{eq:Fmat}}{=}  
& 
\sum_{a,x,\nu,\lambda',\lambda'',\mu} 
d_x 
\cdot 
\tikzzbox{\begin{tikzpicture}[very thick,scale=0.5,color=blue!50!black, baseline=0cm]
	\draw[-dot-] (0,0) .. controls +(0,1) and +(0,1) .. (-1,0);
	\draw[-dot-] (-1,0) .. controls +(0,-1) and +(0,-1) .. (-2,0);
	\draw[-dot-] (-0.5,0.7) .. controls +(0,1.25) and +(0,1.25) .. (-2,0.7);
	\draw (-2,0.7) -- (-2,0);
	\draw[-dot-] (0,-0.7) .. controls +(0,-1.25) and +(0,-1.25) .. (-1.5,-0.7);
	\draw (0,-0.7) -- (0,0);
	\draw[directed] (-1.25,1.5) .. controls +(0,1) and +(0,1) .. (-2.75,1.5);
	\draw (-2.75,1.5) -- (-2.75,-1.6);
	\draw[directed] (-2.75,-1.6) .. controls +(0,-1.25) and +(0,-1.25) .. (-0.75,-1.6);
	\fill (-0.2,-2.05) circle (0pt) node (meet2) {{\scriptsize$\widehat{\lambda'}$}};
	\fill (-1.5,-0.25) circle (0pt) node (meet2) {{\scriptsize$\widehat{\widetilde{\lambda}}$}};
	\fill (-0.9,2) circle (0pt) node (meet2) {{\scriptsize$\lambda$}};
	\fill (-0.5,0.4) circle (0pt) node (meet2) {{\scriptsize$\mu$}};
	%
	%
	\fill (-1.15,0.35) circle (0pt) node (meet2) {{\scriptsize$j$}};
	\fill (-1.6,-1.4) circle (0pt) node (meet2) {{\scriptsize$i$}};
	\fill (-0.25,-0.5) circle (0pt) node (meet2) {{\scriptsize$x$}};
	\fill (-1.3,2.4) circle (0pt) node (meet2) {{\scriptsize$l$}};
	\fill (-2.25,0.5) circle (0pt) node (meet2) {{\scriptsize$m$}};
	\fill (-0.45,1.45) circle (0pt) node (meet2) {{\scriptsize$k$}};
	\end{tikzpicture}}
\; 
\tikzzbox{\begin{tikzpicture}[very thick,scale=0.5,color=blue!50!black, baseline=0cm]
	\draw[-dot-] (0,0) .. controls +(0,1) and +(0,1) .. (1,0);
	\draw[-dot-] (1,0) .. controls +(0,-1) and +(0,-1) .. (2,0);
	\draw[-dot-] (0.5,0.7) .. controls +(0,1.25) and +(0,1.25) .. (2,0.7);
	\draw (2,0.7) -- (2,0);
	\draw[-dot-] (0,-0.7) .. controls +(0,-1.25) and +(0,-1.25) .. (1.5,-0.7);
	\draw (0,-0.7) -- (0,0);
	\draw[directed] (1.25,1.5) .. controls +(0,1.25) and +(0,1.25) .. (-0.75,1.5);
	\draw (-0.75,1.5) -- (-0.75,-1.6);
	\draw[directed] (-0.75,-1.6) .. controls +(0,-1) and +(0,-1) .. (0.75,-1.6);
	%
	\fill (1.2,-2.05) circle (0pt) node (meet2) {{\scriptsize$\widehat{\lambda''}$}};
	\fill (2,-1.1) circle (0pt) node (meet2) {{\scriptsize$\widehat{\mu}$}};
	\fill (1.6,2) circle (0pt) node (meet2) {{\scriptsize$\lambda'$}};
	\fill (0.5,0.3) circle (0pt) node (meet2) {{\scriptsize$\nu$}};
	%
	%
	\fill (0.8,-0.3) circle (0pt) node (meet2) {{\scriptsize$j$}};
	\fill (1.15,-1.15) circle (0pt) node (meet2) {{\scriptsize$k$}};
	\fill (0.25,-0.5) circle (0pt) node (meet2) {{\scriptsize$a$}};
	\fill (1.1,2.5) circle (0pt) node (meet2) {{\scriptsize$l$}};
	\fill (1.7,0.5) circle (0pt) node (meet2) {{\scriptsize$x$}};
	\fill (0.45,1.55) circle (0pt) node (meet2) {{\scriptsize$i$}};
	\end{tikzpicture}}
\cdot \lambda'' \otimes \widehat{\nu}
\\ 
\stackrel{\eqref{eq:basis-sum-rule-iii}}{=} 
& 
\sum_{a,x,\nu,\lambda'',\mu} d_x 
\cdot 
\tikzzbox{\begin{tikzpicture}[very thick,scale=0.5,color=blue!50!black, baseline=-1cm, xscale=-1]
	\draw[-dot-] (0,0) .. controls +(0,1) and +(0,1) .. (1,0);
	\draw[-dot-] (1,0) .. controls +(0,-1) and +(0,-1) .. (2,0);
	\draw[-dot-] (0.5,0.7) .. controls +(0,1.25) and +(0,1.25) .. (2,0.7);
	\draw (2,0.7) -- (2,0);
	\draw (0,-3) -- (0,0);
	\draw (1.5,-0.7) -- (1.5,-2.3);
	\draw[-dot-] (1,-3) .. controls +(0,1) and +(0,1) .. (2,-3);
	\draw[-dot-] (1,-3) .. controls +(0,-1) and +(0,-1) .. (0,-3);
	\draw[-dot-] (0.5,-3.7) .. controls +(0,-1.25) and +(0,-1.25) .. (2,-3.7);
	\draw (2,-3.7) -- (2,-3);
	%
	\draw[directed] (1.25,1.5) .. controls +(0,1.25) and +(0,1.25) .. (3,1.5);
	\draw[redirected] (1.25,-4.5) .. controls +(0,-1.25) and +(0,-1.25) .. (3,-4.5);	
	\draw (3,1.5) -- (3,-4.5);
	%
	\fill (1.5,-0.2) circle (0pt) node (meet2) {{\scriptsize$\widehat{\widetilde{\lambda}}$}};
	\fill (1.5,-2.6) circle (0pt) node (meet2) {{\scriptsize$\nu$}};
	\fill (0.5,-3.3) circle (0pt) node (meet2) {{\scriptsize$\widehat{\mu}$}};
	\fill (1.25,1.25) circle (0pt) node (meet2) {{\scriptsize$\lambda$}};
	\fill (1.25,-4.1) circle (0pt) node (meet2) {{\scriptsize$\widehat{\lambda''}$}};
	\fill (0.5,0.3) circle (0pt) node (meet2) {{\scriptsize$\mu$}};
	%
	%
	\fill (0.8,-0.3) circle (0pt) node (meet2) {{\scriptsize$j$}};
	\fill (1.2,-3.3) circle (0pt) node (meet2) {{\scriptsize$j$}};
	\fill (2.25,-3.3) circle (0pt) node (meet2) {{\scriptsize$a$}};
	\fill (0.3,-4.25) circle (0pt) node (meet2) {{\scriptsize$k$}};
	\fill (1.7,-1.5) circle (0pt) node (meet2) {{\scriptsize$i$}};
	\fill (0.25,-1.5) circle (0pt) node (meet2) {{\scriptsize$x$}};
	\fill (1.4,2.5) circle (0pt) node (meet2) {{\scriptsize$l$}};
	\fill (2.4,0.5) circle (0pt) node (meet2) {{\scriptsize$m$}};
	\fill (0.45,1.55) circle (0pt) node (meet2) {{\scriptsize$k$}};
	\end{tikzpicture}}
 \cdot \lambda'' \otimes \widehat{\nu}
\;\; \stackrel{\eqref{eq:basis-sum-rule-ii}}{=} \;\;
\sum_{a,\nu,\lambda''} 
\tikzzbox{\begin{tikzpicture}[very thick,scale=0.5,color=blue!50!black, baseline=-1cm,xscale=-1]
	\draw[-dot-] (1,0) .. controls +(0,-1) and +(0,-1) .. (2,0);
	\draw[-dot-] (0,0.7) .. controls +(0,1.25) and +(0,1.25) .. (2,0.7);
	\draw (2,0.7) -- (2,0);
	\draw (0,-3.7) -- (0,0.7);
	\draw (1.5,-0.7) -- (1.5,-2.3);
	\draw[-dot-] (1,-3) .. controls +(0,1) and +(0,1) .. (2,-3);
	\draw[-dot-] (0,-3.7) .. controls +(0,-1.25) and +(0,-1.25) .. (2,-3.7);
	\draw (2,-3.7) -- (2,-3);
    \draw[directed] (1,1.5) .. controls +(0,1.25) and +(0,1.25) .. (3,1.5);
	\draw[redirected] (1,-4.5) .. controls +(0,-1.25) and +(0,-1.25) .. (3,-4.5);	
	\draw (3,1.5) -- (3,-4.5);
	\draw[directed] (1,0) .. controls +(0,0.75) and +(0,0.75) .. (0.25,0);
	\draw[directed] (0.25,-3,0) .. controls +(0,-0.75) and +(0,-0.75) .. (1,-3);
	\draw (0.25,0) -- (0.25,-3);
	%
	\fill (1.5,-0.2) circle (0pt) node (meet2) {{\scriptsize$\widehat{\widetilde{\lambda}}$}};
	\fill (1.5,-2.6) circle (0pt) node (meet2) {{\scriptsize$\nu$}};
	\fill (1.25,1.25) circle (0pt) node (meet2) {{\scriptsize$\lambda$}};
	\fill (1.25,-4.1) circle (0pt) node (meet2) {{\scriptsize$\widehat{\lambda''}$}};
	%
	%
	\fill (0.8,-0.3) circle (0pt) node (meet2) {{\scriptsize$j$}};
	\fill (2.25,-3.3) circle (0pt) node (meet2) {{\scriptsize$a$}};
	\fill (1.7,-1.5) circle (0pt) node (meet2) {{\scriptsize$i$}};
	\fill (-0.25,-1.5) circle (0pt) node (meet2) {{\scriptsize$k$}};
	\fill (1.1,2.5) circle (0pt) node (meet2) {{\scriptsize$l$}};
	\fill (2.4,0.5) circle (0pt) node (meet2) {{\scriptsize$m$}};
	\end{tikzpicture}}
\!\! \cdot \lambda'' \otimes \widehat{\nu}
\\
\stackrel{\eqref{eq:basis-sum-rule-i}}{=} 
& 
\; \sum_{a,\nu,\lambda''} 
\frac{\delta_{a,m}}{d_m}  \, 
\delta_{\nu,\widetilde{\lambda}} \cdot 
\tikzzbox{\begin{tikzpicture}[very thick,scale=0.4,color=blue!50!black, baseline=-0.1cm]
	\draw[-dot-] (0,0) .. controls +(0,1) and +(0,1) .. (1,0);
	\draw[-dot-] (1,0) .. controls +(0,-1) and +(0,-1) .. (0,0);
	\draw (0.5,0.7) -- (0.5,1.2);
	\draw[directed] (0.5,1.2) .. controls +(0,1) and +(0,1) .. (-1,1.2);
	\draw (-1,1.2) -- (-1,-1.2);
	\draw[directed] (-1,-1.2) .. controls +(0,-1) and +(0,-1) .. (0.5,-1.2);
	\draw (0.5,-0.7) -- (0.5,-1.2);
	%
	\fill (1,-1.3) circle (0pt) node (meet2) {{\scriptsize$\widehat{\lambda''}$}};
	\fill (1,1.2) circle (0pt) node (meet2) {{\scriptsize$\lambda$}};
	\fill (-0.3,0) circle (0pt) node (meet2) {{\scriptsize$m\vphantom{k}$}};
	\fill (1.4,0) circle (0pt) node (meet2) {{\scriptsize$k\vphantom{k}$}};
	\fill (0.6,2) circle (0pt) node (meet2) {{\scriptsize$l$}};
	\end{tikzpicture}}
\! \cdot \lambda'' \otimes \widehat{\nu} \,
\stackrel{\eqref{eq:basis-sum-rule-i}}{=} 
\, 
\sum_{\lambda''} \frac{\delta_{\lambda,\lambda''}}{d_m} \cdot \lambda'' \otimes \widehat{\widetilde{\lambda}}
= 
\frac{1}{d_m} \cdot \lambda \otimes \widehat{\widetilde{\lambda}}
\ .
\end{align*}
Hence we have shown that the first identity in~\eqref{eq:349n} holds. 
The other identity is checked similarly. 

\begin{figure}[!htbp]
	\begin{center} 
		$$
		\begin{tikzpicture}[very thick, x=2cm, y=5cm, scale=0.75, color=blue!50!black, baseline=0.75cm]
		\draw[directed] (0.15,2) .. controls +(0,1) and +(0,1) .. (1.25,2);
		\draw (0.15,0) -- (0.15,2); 
		\draw[redirected] (0.15,0) .. controls +(0,-1) and +(0,-1) .. (1.25,0);
		\draw[color=white, line width=4pt] (1.25,0) -- (1.25,2);
		\draw (1.25,0) -- (1.25,2);
		\draw (-0.15,-1.5) -- (-0.15,0); 
		\draw (-0.15,2) -- (-0.15,3.5); 
		\draw[redirected] (-0.15,2) .. controls +(0,-1) and +(0,-1) .. (-1.25,2);
		\draw (-1.25,2) -- (-1.25,3.5);
		\draw[directed] (-0.15,0) .. controls +(0,1) and +(0,1) .. (-1.25,0);
		\draw (-1.25,0) -- (-1.25,-1.5);
		%
		\fill[color=white] (0,0) node[inner sep=5pt,draw, rounded corners=1pt, fill, color=white] (R2) {{\scriptsize$\;\alpha\;$}};
		\draw[line width=1pt, color=black] (0,0) node[inner sep=5pt, draw, semithick, rounded corners=1pt] (R) {{\scriptsize$\;\alpha\;$}};
		%
		\fill[color=white] (0,2) node[inner sep=5pt,draw, rounded corners=1pt, fill, color=white] (R2) {{\scriptsize$\;\alpha\;$}};
		\draw[line width=1pt, color=black] (0,2) node[inner sep=5pt, draw, semithick, rounded corners=1pt] (R) {{\scriptsize$\;\bar\alpha\;$}};
		%
		\fill[color=black] (0.15,0.75) circle (2.9pt) node[right] (meet) {{\scriptsize$\!\psi_2^2\!$}};
		%
		%
		\fill (0.15,3.5) circle (0pt) node (M) {{\scriptsize${}_z T_{i,j}$}};
		\fill (-0.85,3.5) circle (0pt) node (M) {{\scriptsize$({}_z T_{a,b})^\dagger$}};
		\fill (0.5,2.3) circle (0pt) node (M) {{\scriptsize${}_{z''} T_{c,b}$}};
		\fill (0.5,-0.3) circle (0pt) node (M) {{\scriptsize${}_j T_{c,d}$}};
		\fill (0.45,0.3) circle (0pt) node (M) {{\scriptsize${}_{y'} T_{i,c}$}};
		\fill (-0.5,0.3) circle (0pt) node (M) {{\scriptsize${}_{k} T_{y',d}$}};
		\fill (0.15,-1.5) circle (0pt) node (M) {{\scriptsize${}_k T_{i,j}$}};
		\fill (-0.85,-1.5) circle (0pt) node (M) {{\scriptsize$({}_k T_{a,b})^\dagger$}};
		\fill[black] (2.5,3.25) circle (0pt) node[right] (M) 
		{$\displaystyle{\sum_{c,\mu,\mu'} \sum_{z,\nu} \sum_{\nu''}} d_c d_a^{-1} d_{j}^{-1} F_{\mu\mu'}^{\lambda\widetilde{\lambda}} (F^{-1})_{\mu'\mu}^{\nu\nu''} \cdot \widehat{\nu} \otimes \nu''$};
		\draw[semithick, black, |->] (2.9,2.6) -- (2.9,3.05);
		\fill[black] (2.5,2.4) circle (0pt) node[right] (M) 
		{$\displaystyle{\sum_{c,\mu,\mu'} \sum_{z,\nu} \sum_{z'',\nu'',\mu''}} d_c d_a^{-1} d_{z''}^{-1} F_{\mu\mu'}^{\lambda\widetilde{\lambda}} (F^{-1})_{\mu'\mu''}^{\nu\nu''}$};
		\fill[black] (5.5,2.25) circle (0pt) node[right] (M) 
		{$\cdot \widehat{\nu} \otimes \nu'' \otimes \mu'' \otimes \widehat{\mu}$};
		\draw[semithick, black, |->] (2.9,1.8) -- (2.9,2.2);
		\fill[black] (2.5,1.6) circle (0pt) node[right] (M) 
		{$\displaystyle{\sum_{c,\mu,\mu'} \sum_{z,\nu}} d_c d_a^{-1} F_{\mu\mu'}^{\lambda\widetilde{\lambda}} \cdot \widehat{\nu} \otimes \nu \otimes \mu' \otimes \widehat{\mu}$};
		\draw[semithick, black, |->] (2.9,1.2) -- (2.9,1.4);
		\fill[black] (2.5,1) circle (0pt) node[right] (M) 
		{$\displaystyle{\sum_{c,\mu,\mu'}} d_c d_a^{-1} F_{\mu\mu'}^{\lambda\widetilde{\lambda}} \cdot \mu' \otimes \widehat{\mu}$};
		\draw[semithick, black, |->] (2.9,0.6) -- (2.9,0.8);
		\fill[black] (2.5,0.4) circle (0pt) node[right] (M) 
		{$\displaystyle{\sum_{c,d,\mu} \sum_{y',\lambda',\mu'}} d_{y'}^{-1} F_{\mu\mu'}^{\lambda\lambda'} \cdot \widehat{\widetilde{\lambda}} \otimes \lambda' \otimes \mu' \otimes \widehat{\mu}$};
		\draw[semithick, black, |->] (2.9,-0.2) -- (2.9,0.2);
		\fill[black] (2.5,-0.4) circle (0pt) node[right] (M) 
		{$\displaystyle{\sum_{c,d,\mu}} \widehat{\widetilde{\lambda}} \otimes \lambda \otimes \mu \otimes \widehat{\mu}$};
		\draw[semithick, black, |->] (2.9,-1.15) -- (2.9,-0.6);
		\fill[black] (2.5,-1.25) circle (0pt) node[right] (M) 
		{$\widehat{\widetilde{\lambda}} \otimes \lambda$};
		\end{tikzpicture} 
		$$
	\end{center}
	\caption{Computing the left-hand side of condition~\eqref{eq:350n} for $\Aca$.} 
	\label{fig:350} 
\end{figure}
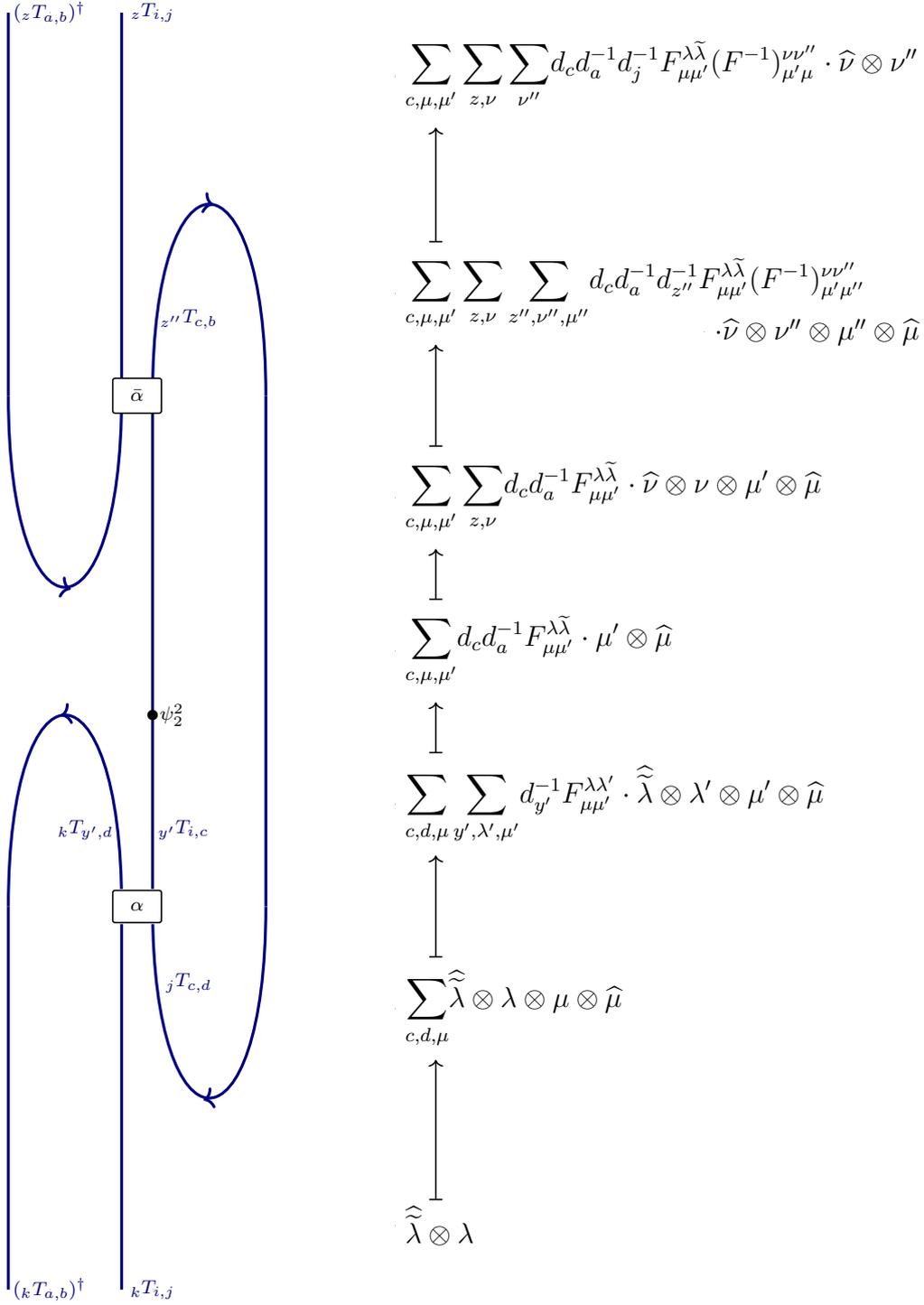

Next we turn to the first constraint in~\eqref{eq:350n}. 
Verifying that it holds for our orbifold data is similar to the case of~\eqref{eq:349n}: 
We have to show that the left-hand side of~\eqref{eq:350n} acts as the identity times $d_k^{-1}$ on $\widehat{\widetilde{\lambda}} \otimes \lambda$ for all elements $\lambda \in \Hom_{\Cat{S}}(i\otimes j, k)$ and $\widehat{\widetilde{\lambda}} \in \Hom_{\Cat{S}}(k,a\otimes b)$ of chosen bases for all $a,b,i,j,k \in I$. 
This action is computed in Figure~\ref{fig:350} to produce 
\be\label{eq:TV-orb-aux2}
\sum_{c,z,\nu,\nu',\mu,\mu'} d_c 
\big( d_a^{-1} F^{\lambda\widetilde{\lambda}}_{\mu\mu'} \big) \cdot 
\big( d_j^{-1} (F^{-1})^{\nu\nu''}_{\mu'\mu} \big) 
\cdot \widehat{\nu} \otimes \nu'' \, . 
\ee
This can be simplified to 
\begin{align*}
\eqref{eq:TV-orb-aux2} 
& 
= 
\sum_{c,z,\nu,\nu'',\mu,\mu'} d_c 
\cdot 
\tikzzbox{\begin{tikzpicture}[very thick,scale=0.5,color=blue!50!black, baseline=0cm]
	\draw[-dot-] (0,0) .. controls +(0,1) and +(0,1) .. (-1,0);
	\draw[-dot-] (-1,0) .. controls +(0,-1) and +(0,-1) .. (-2,0);
	\draw[-dot-] (-0.5,0.7) .. controls +(0,1.25) and +(0,1.25) .. (-2,0.7);
	\draw (-2,0.7) -- (-2,0);
	\draw[-dot-] (0,-0.7) .. controls +(0,-1.25) and +(0,-1.25) .. (-1.5,-0.7);
	\draw (0,-0.7) -- (0,0);
	\draw[directed] (-1.25,1.5) .. controls +(0,1) and +(0,1) .. (-2.75,1.5);
	\draw (-2.75,1.5) -- (-2.75,-1.6);
	\draw[directed] (-2.75,-1.6) .. controls +(0,-1.25) and +(0,-1.25) .. (-0.75,-1.6);
	\fill (-0.2,-2.05) circle (0pt) node (meet2) {{\scriptsize$\widehat{\widetilde{\lambda}}$}};
	\fill (-1.5,-0.25) circle (0pt) node (meet2) {{\scriptsize$\widehat{\mu'}$}};
	\fill (-0.9,2) circle (0pt) node (meet2) {{\scriptsize$\lambda$}};
	\fill (-0.5,0.3) circle (0pt) node (meet2) {{\scriptsize$\mu$}};
	%
	\fill (-1.15,0.35) circle (0pt) node (meet2) {{\scriptsize$c$}};
	\fill (-1.6,-1.4) circle (0pt) node (meet2) {{\scriptsize$a$}};
	\fill (-0.25,-0.5) circle (0pt) node (meet2) {{\scriptsize$b$}};
	\fill (-1.3,2.4) circle (0pt) node (meet2) {{\scriptsize$k$}};
	\fill (-2.25,0.5) circle (0pt) node (meet2) {{\scriptsize$i$}};
	\fill (-0.45,1.45) circle (0pt) node (meet2) {{\scriptsize$j$}};
	\end{tikzpicture}}
\; 
\tikzzbox{\begin{tikzpicture}[very thick,scale=0.5,color=blue!50!black, baseline=0cm]
	\draw[-dot-] (0,0) .. controls +(0,1) and +(0,1) .. (1,0);
	\draw[-dot-] (1,0) .. controls +(0,-1) and +(0,-1) .. (2,0);
	\draw[-dot-] (0.5,0.7) .. controls +(0,1.25) and +(0,1.25) .. (2,0.7);
	\draw (2,0.7) -- (2,0);
	\draw[-dot-] (0,-0.7) .. controls +(0,-1.25) and +(0,-1.25) .. (1.5,-0.7);
	\draw (0,-0.7) -- (0,0);
	\draw[directed] (1.25,1.5) .. controls +(0,1.25) and +(0,1.25) .. (-0.75,1.5);
	\draw (-0.75,1.5) -- (-0.75,-1.6);
	\draw[directed] (-0.75,-1.6) .. controls +(0,-1) and +(0,-1) .. (0.75,-1.6);
	%
	\fill (1.2,-2.05) circle (0pt) node (meet2) {{\scriptsize$\widehat{\nu''}$}};
	\fill (2,-1.1) circle (0pt) node (meet2) {{\scriptsize$\widehat{\mu}$}};
	\fill (1.6,2) circle (0pt) node (meet2) {{\scriptsize$\nu$}};
	\fill (0.5,0.3) circle (0pt) node (meet2) {{\scriptsize$\mu'$}};
	%
	%
	\fill (0.8,-0.3) circle (0pt) node (meet2) {{\scriptsize$c$}};
	\fill (1.15,-1.15) circle (0pt) node (meet2) {{\scriptsize$j$}};
	\fill (0.2,-0.5) circle (0pt) node (meet2) {{\scriptsize$i$}};
	\fill (1.1,2.5) circle (0pt) node (meet2) {{\scriptsize$z$}};
	\fill (1.75,0.5) circle (0pt) node (meet2) {{\scriptsize$b$}};
	\fill (0.45,1.55) circle (0pt) node (meet2) {{\scriptsize$a$}};
	\end{tikzpicture}}
\cdot 
\widehat{\nu} \otimes \nu'' 
\\ 
& 
\overset{(*)}= 
\sum_{c,z,\nu,\nu'',\mu'} d_c 
\cdot
\tikzzbox{\begin{tikzpicture}[very thick,scale=0.5,color=blue!50!black, baseline=0cm]
	\draw[directed] (-1,0.5) .. controls +(0,-0.35) and +(0,-0.35) .. (0,0.5);
	\draw[-dot-] (0,0.5) .. controls +(0,1) and +(0,1) .. (1,0.5);
	\draw[-dot-] (0.5,1.2) .. controls +(0,1) and +(0,1) .. (1.75,1.2);
	\draw[-dot-] (0.5,3.5) .. controls +(0,-1) and +(0,-1) .. (1.75,3.5);
	\draw[directed] (0.5,3.5) .. controls +(0,0.5) and +(0,0.5) .. (-1,3.5);
	\draw[directed] (1.75,3.5) .. controls +(0,1.5) and +(0,1.5) .. (-1.5,3.5);
	\draw (-1,3.5) -- (-1,0.5);
	\draw (1.125,2.8) -- (1.125,2);
	\draw (-1.5,3.5) -- (-1.5,-3.5);
	\draw (1,0.5) -- (1,-0.5);
	\draw (1.75,1.2) -- (1.75,-1.2);
	\draw[redirected] (-1,-0.5) .. controls +(0,0.35) and +(0,0.35) .. (0,-0.5);
	\draw[-dot-] (0,-0.5) .. controls +(0,-1) and +(0,-1) .. (1,-0.5);
	\draw[-dot-] (0.5,-1.2) .. controls +(0,-1) and +(0,-1) .. (1.75,-1.2);
	\draw[-dot-] (0.5,-3.5) .. controls +(0,1) and +(0,1) .. (1.75,-3.5);
	\draw[redirected] (0.5,-3.5) .. controls +(0,-0.5) and +(0,-0.5) .. (-1,-3.5);
	\draw[redirected] (1.75,-3.5) .. controls +(0,-1.5) and +(0,-1.5) .. (-1.5,-3.5);
	\draw (-1,-3.5) -- (-1,-0.5);
	\draw (1.125,-2.8) -- (1.125,-2);
	%
	\fill (0.5,0.8) circle (0pt) node (meet2) {{\scriptsize$\mu'$}};
	\fill (1.126,1.55) circle (0pt) node (meet2) {{\scriptsize$\nu$}};
	\fill (1.126,3.3) circle (0pt) node (meet2) {{\scriptsize$\widehat{\nu''}$}};
	\fill (0.5,-0.8) circle (0pt) node (meet2) {{\scriptsize$\widehat{\mu'}$}};
	\fill (1.126,-1.5) circle (0pt) node (meet2) {{\scriptsize$\widehat{\widetilde{\lambda}}$}};
	\fill (1.126,-3.1) circle (0pt) node (meet2) {{\scriptsize$\lambda$}};
	\fill (0.75,0) circle (0pt) node (meet2) {{\scriptsize$c$}};
	\fill (1.5,0) circle (0pt) node (meet2) {{\scriptsize$b$}};
	\fill (1.8,2.8) circle (0pt) node (meet2) {{\scriptsize$j$}};
	\fill (1.8,-2.8) circle (0pt) node (meet2) {{\scriptsize$j$}};
	\fill (1.4,2.4) circle (0pt) node (meet2) {{\scriptsize$z$}};
	\fill (1.4,-2.4) circle (0pt) node (meet2) {{\scriptsize$k$}};
	\fill (0.5,1.9) circle (0pt) node (meet2) {{\scriptsize$a$}};
	\fill (0.5,-2) circle (0pt) node (meet2) {{\scriptsize$a$}};
	\fill (0.5,2.8) circle (0pt) node (meet2) {{\scriptsize$i$}};
	\fill (0.5,-2.8) circle (0pt) node (meet2) {{\scriptsize$i$}};
	\end{tikzpicture}}
\cdot
\widehat{\nu} \otimes \nu'' 
=
\sum_{z,\nu,\nu''} 
\tikzzbox{\begin{tikzpicture}[very thick,scale=0.5,color=blue!50!black, baseline=0cm]
	\draw[-dot-] (0.5,1.2) .. controls +(0,1) and +(0,1) .. (1.75,1.2);
	\draw[-dot-] (0.5,3.5) .. controls +(0,-1) and +(0,-1) .. (1.75,3.5);
	\draw[directed] (0.5,3.5) .. controls +(0,0.5) and +(0,0.5) .. (-1,3.5);
	\draw[directed] (1.75,3.5) .. controls +(0,1.5) and +(0,1.5) .. (-1.5,3.5);
	\draw (-1,3.5) -- (-1,0.5);
	\draw (1.125,2.8) -- (1.125,2);
	\draw (-1.5,3.5) -- (-1.5,-3.5);
	\draw (0.5,1.2) -- (0.5,-1.2);
	\draw (-1,1.2) -- (-1,-1.2);
	\draw (1.75,1.2) -- (1.75,-1.2);
	\draw[-dot-] (0.5,-1.2) .. controls +(0,-1) and +(0,-1) .. (1.75,-1.2);
	\draw[-dot-] (0.5,-3.5) .. controls +(0,1) and +(0,1) .. (1.75,-3.5);
	\draw[redirected] (0.5,-3.5) .. controls +(0,-0.5) and +(0,-0.5) .. (-1,-3.5);
	\draw[redirected] (1.75,-3.5) .. controls +(0,-1.5) and +(0,-1.5) .. (-1.5,-3.5);
	\draw (-1,-3.5) -- (-1,-0.5);
	\draw (1.125,-2.8) -- (1.125,-2);
	%
	\fill (1.126,1.55) circle (0pt) node (meet2) {{\scriptsize$\nu$}};
	\fill (1.126,3.3) circle (0pt) node (meet2) {{\scriptsize$\widehat{\nu''}$}};
	\fill (1.126,-1.5) circle (0pt) node (meet2) {{\scriptsize$\widehat{\widetilde{\lambda}}$}};
	\fill (1.126,-3.1) circle (0pt) node (meet2) {{\scriptsize$\lambda$}};
	\fill (0.75,0) circle (0pt) node (meet2) {{\scriptsize$a\vphantom{b}$}};
	\fill (1.5,0) circle (0pt) node (meet2) {{\scriptsize$b$}};
	\fill (-1.75,0) circle (0pt) node (meet2) {{\scriptsize$j$}};
	\fill (1.4,2.4) circle (0pt) node (meet2) {{\scriptsize$z$}};
	\fill (1.4,-2.4) circle (0pt) node (meet2) {{\scriptsize$k$}};
	\fill (0.5,2.8) circle (0pt) node (meet2) {{\scriptsize$i$}};
	\end{tikzpicture}}
\cdot
\widehat{\nu} \otimes \nu'' 
\\ 
& = 
\sum_{z,\nu,\nu''} 
\frac{1}{d_k} \,\delta_{\nu,\widetilde{\lambda}}\, \delta_{k,z} \cdot  
\tikzzbox{\begin{tikzpicture}[very thick,scale=0.5,color=blue!50!black, baseline=0cm]
	\draw[-dot-] (0.5,1.5) .. controls +(0,-1) and +(0,-1) .. (1.75,1.5);
	\draw[directed] (0.5,1.5) .. controls +(0,0.5) and +(0,0.5) .. (-0.5,1.5);
	\draw[directed] (1.75,1.5) .. controls +(0,1.5) and +(0,1.5) .. (-1.5,1.5);
	\draw (1.125,0.8) -- (1.125,0);
	\draw (-1.5,1.5) -- (-1.5,-1.5);
	\draw (-0.5,1.5) -- (-0.5,-1.5);
	%
	\draw[-dot-] (0.5,-1.5) .. controls +(0,1) and +(0,1) .. (1.75,-1.5);
	\draw[redirected] (0.5,-1.5) .. controls +(0,-0.5) and +(0,-0.5) .. (-0.5,-1.5);
	\draw[redirected] (1.75,-1.5) .. controls +(0,-1.5) and +(0,-1.5) .. (-1.5,-1.5);
	\draw (1.125,-0.8) -- (1.125,0);
	%
	\fill (1.126,1.3) circle (0pt) node (meet2) {{\scriptsize$\widehat{\nu''}$}};
	\fill (1.126,-1.1) circle (0pt) node (meet2) {{\scriptsize$\lambda$}};
	\fill (-1.25,0) circle (0pt) node (meet2) {{\scriptsize$j\vphantom{b}$}};
	\fill (1.4,0) circle (0pt) node (meet2) {{\scriptsize$z$}};
	\fill (-0.25,0) circle (0pt) node (meet2) {{\scriptsize$i$}};
	\end{tikzpicture}}
\cdot 
\widehat{\nu} \otimes \nu'' 
= \frac{1}{d_k} 
\sum_\nu 
\tikzzbox{\begin{tikzpicture}[very thick,scale=0.4,color=blue!50!black, baseline=-0.1cm]
	\draw[-dot-] (0,0) .. controls +(0,1) and +(0,1) .. (1,0);
	\draw[-dot-] (1,0) .. controls +(0,-1) and +(0,-1) .. (0,0);
	\draw (0.5,0.7) -- (0.5,1.2);
	\draw[directed] (0.5,1.2) .. controls +(0,1) and +(0,1) .. (-1,1.2);
	\draw (-1,1.2) -- (-1,-1.2);
	\draw[directed] (-1,-1.2) .. controls +(0,-1) and +(0,-1) .. (0.5,-1.2);
	\draw (0.5,-0.7) -- (0.5,-1.2);
	%
	\fill (1,-1.3) circle (0pt) node (meet2) {{\scriptsize$\widehat{\nu''}$}};
	\fill (1,1.2) circle (0pt) node (meet2) {{\scriptsize$\lambda$}};
	\fill (-0.3,0) circle (0pt) node (meet2) {{\scriptsize$i$}};
	\fill (1.4,0) circle (0pt) node (meet2) {{\scriptsize$j\vphantom{k}$}};
	\fill (0.6,2) circle (0pt) node (meet2) {{\scriptsize$z$}};
	\end{tikzpicture}}
\cdot \widehat{\widetilde{\lambda}} \otimes \nu'' 
\\
& = \frac{1}{d_k} \cdot \widehat{\widetilde{\lambda}} \otimes \lambda \, , 
\end{align*}
where in $(*)$ first the basis element $\lambda$ in the left diagram and the element $\widehat{\nu''}$ in the right diagram are ``taken around'' by using the cyclicity of the trace, and then \eqref{eq:basis-sum-rule-iii} is used. 
The second identity in~\eqref{eq:350n} follows analogously. 

It remains to verify the constraints in~\eqref{eq:351n}. 
Writing again $1_k$ for the unit in the $k$-th copy of~$\Bbbk$ in ${A} = \bigoplus_{k\in I} \Bbbk$, the left-hand side of the first identity in~\eqref{eq:351n} is 
\be
\begin{tikzpicture}[
baseline=(current bounding box.base), 
>=stealth,
descr/.style={fill=white,inner sep=3.5pt}, 
normal line/.style={->}
] 
\matrix (m) [matrix of math nodes, row sep=1.0em, column sep=4.2em, text height=1.1ex, text depth=0.1ex] {%
	{A}  &  {T}\otimes_{{A} \otimes_\Bbbk {A}} {T}^\dagger  &  {A} 
	\\
	1_k  & \displaystyle{\sum_{i,j,\lambda} d_i d_j \cdot \lambda \otimes \widehat{\lambda}}  &  \displaystyle{\sum_{i,j} d_i d_j N_{ij}^k \cdot \id_{A_k}}
	\\
};
\path[font=\footnotesize] (m-1-1) edge[->] node[auto] {$\psi_1^2 \psi_2^2 \text{coev}_{{T}}$} (m-1-2);
\path[font=\footnotesize] (m-1-2) edge[->] node[auto] {$\widetilde{\text{coev}}_{{T}}$} (m-1-3);
\path[font=\footnotesize] (m-2-1) edge[|->] node[auto] {} (m-2-2);
\path[font=\footnotesize] (m-2-2) edge[|->] node[auto] {} (m-2-3);
\end{tikzpicture}
\ee
where $N_{ij}^k := \dim_\Bbbk \Hom_{\Cat{S}}(i\otimes j,k)$. 
We further compute 
\begin{align}
\sum_{i,j} d_i d_j  N_{ij}^k
&
 = \sum_{i,j} d_i d_{j^*}  N_{ik^*}^{j^*}
 = \sum_{i} d_i d_i d_{k}
 = 
	 \phi^{-2}
\cdot (\eta_A \circ \psi^2)  \big|_{A_k} \, ,
\end{align}
where in the first step we used that $N_{ij}^k = N_{ik^*}^{j^*}$ and that $d_j = d_{j^*}$. The second step is $\dim(i \otimes k^*) = \sum_l N_{ik^*}^l \dim(l)$.
\end{proof}

\begin{remark}
\label{rem:SFCinternal}
The orbifold datum~$\Aca$ of Definition~\ref{prop:orbidata} constructed from a spherical fusion category~$\mathcal S$ is expressed internally to the modular tensor category~$\Vectk$, in line with the general setup of Section~\ref{subsec:SODribbon}. 
Equivalently, $\Aca$ can be described internal to the 3-category with duals $\textrm{Bimod}_\Bbbk$ of spherical fusion categories, bimodule categories with module traces, bimodule functors and their natural transformations (studied in \cite{GregorDiss}), along the general lines of \cite[Sect.\,4.2]{CRS1}. 
In this formulation, 3-, 2-, 1- and 0-strata are labelled by $\Vectk$, the $\Vectk$-$\Vectk$-bimodule~$\mathcal S$, the functor $\otimes\colon \mathcal{S} \boxtimes \mathcal{S} \to \mathcal{S}$ and natural transformations constructed from the associator, respectively. 
Similarly, the $\phi$- and $\psi$-insertions are also natural transformations; for example, one can compute the right quantum dimension $\dim_{\textrm{r}}(\otimes)$ to be $\dim \mathcal S$ times the identity, which fixes 
	 $\phi^{2}$
to be $(\dim \mathcal S)^{-1} \cdot \id_{\textrm{Id}_{\Vectk}}$. 

In this way the orbifold datum~$\Aca$ is a spherical fusion category internal to the 3-category with duals $\mathcal T_{\zztriv} \subset \textrm{Bimod}_\Bbbk$ constructed from~$\zztriv$ as in \cite{CMS}. 
A related idea to use (spherical fusion categories viewed as) ``2-algebras'' to construct Turaev--Viro theory was outlined in \cite{prehistory}. 

In general, we can think of orbifold data for a 3-dimensional defect TQFT~$\zz$ as spherical fusion categories internal to~$\tz$. 
\end{remark}

\subsection{Turaev--Viro theory is an orbifold}
\label{subsec:TVorbifolds}

In this section we prove that for every spherical fusion category~${\Cat{S}}$, Turaev--Viro theory $\zz^{\text{TV},\Cat{S}}$ and the orbifold theory $\zz^{\text{triv}}_{\A^{{\Cat{S}}}}$ are isomorphic as TQFTs. 

\medskip 

We first show that $\zz^{\text{TV},\Cat{S}}$ and $\zz^{\text{triv}}_{\A^{{\Cat{S}}}}$ assign \textsl{identical} invariants to closed 3-manifolds. 
Let~$M$ be such a closed manifold, and let~$t$ be an oriented triangulation of~$M$. 
As recalled from \cite{CRS1} in Section~\ref{sec:orbif-defect-tqfts}, by decorating the Poincar\'{e} dual stratification with the orbifold datum $\Aca$ from Definition~\ref{prop:orbidata}, we obtain a morphism $M^{t,\Aca}\colon \emptyset \to \emptyset$ in $\Borddef_3(\mathds{D}^{\text{triv}})$. By definition, 
\be\label{eq:ZMZM}
\zz^{\text{triv}}_{\Aca}(M) = \zz^{\text{triv}}\big( M^{t,\Aca}\big) \, . 
\ee

To compute the right-hand side of~\eqref{eq:ZMZM}, we will denote the set of $j$-strata of $M^{t,\Aca}$ by $M_j$ for $j \in \{ 1,2,3\}$, while the sets of positively and negatively oriented 0-strata are denoted $M_0^+$ and $M_0^-$, respectively. 
By construction, the invariant $\zz^{\text{triv}}(M^{t,\Aca})$ is a single string diagram~$D$ in $\Vectk$. 
Using the decompositions ${A} = \bigoplus_{i} \Bbbk$ and ${T} = \bigoplus_{i,j,k} \Hom_{\Cat{S}} (i\otimes j,k)$, the diagram~$D$ can be written as a sum of string diagrams whose strings are labelled by simple objects in~$I$. 
The morphisms in these diagrams are either point insertions 
	 $\psi^2, \phi^2$, 
or duality maps 
\begin{align}
& 
\Hom_{\Cat{S}}(l,a\otimes b) \otimes_\Bbbk \Hom_{\Cat{S}}(i\otimes j,k) \ni \widehat{\lambda'} \otimes \lambda \lmt 
\delta_{a,i} \delta_{b,j} \delta_{k,l} \delta_{\lambda,\lambda'} \, , 
\\
& 
1_k \lmt \sum_\lambda \lambda \otimes \widehat{\lambda} \in \Hom_{\Cat{S}}(i\otimes j,k) \otimes_\Bbbk \Hom_{\Cat{S}}(k,i\otimes j) \, , 
\end{align}
or their tilded versions 
\begin{align}
& 
\Hom_{\Cat{S}}(i\otimes j,k) \otimes_\Bbbk \Hom_{\Cat{S}}(l,a\otimes b) \ni \lambda \otimes \widehat{\lambda'} \lmt 
\delta_{a,i} \delta_{b,j} \delta_{k,l} \delta_{\lambda,\lambda'} \, , 
\\
& 
1_k \lmt \sum_\lambda \widehat{\lambda} \otimes \lambda \in \Hom_{\Cat{S}}(k,i\otimes j) \otimes_\Bbbk \Hom_{\Cat{S}}(i\otimes j,k) 
\end{align}
as in Section~\ref{subsec:TVire}, or the component maps 
\begin{align}
\alpha \colon \lambda \otimes \mu 
	& \lmt \sum_{d,\lambda',\mu'} d_{d}^{-1} F^{\lambda\lambda'}_{\mu\mu'} \cdot \lambda' \otimes \mu' \, , 
\\
\bar\alpha \colon \lambda' \otimes \mu'
& \lmt \sum_{c,\lambda'',\mu''} d_{c}^{-1} (F^{-1})^{\lambda'\lambda''}_{\mu'\mu''} \cdot \lambda'' \otimes \mu''  
\end{align}
of Definition~\ref{prop:orbidata}, corresponding to 0-strata in $M_0^\pm$. 

As we will explain in the following, $\zz^{\text{triv}}( M^{t,\Aca})$ is equal to 
\be\label{eq:CIA}
	 \phi^{2|M_3|} 
\cdot
\sum_{\mathcal{I}=(i_1,\dots,i_{|M_2|}) \in I^{|M_2|}} d_{i_1} \ldots d_{i_{|M_2|}}
\sum_{\lambda_e^{\mathcal{I}}}
\Big( \prod_{x\in M_0} \mathds{F}_{\mathcal{C}}(\Gamma_x) \Big(\bigotimes_{e\in E_x} \lambda_e^{\mathcal{I}} \Big) \Big) \, . 
\ee
To arrive at this expression, first note that each 3-stratum in $M^{t,\Aca}$ carries a
	 $\phi^2$-insertion, 
leading to the global factor 
	 $\phi^{2|M_3|}$. 
Secondly, each ${A}$-labelled 2-stratum carries an insertion of $\psi^2 = \text{diag}(d_1,\dots,d_{|I|})$, leading to $\sum_{\mathcal{I}} d_{i_1} \ldots d_{i_{|M_2|}}$ when decomposing ${A} = \bigoplus_{i} \Bbbk$. 
Thirdly, for fixed $\mathcal{I} \in I^{|M_2|}$ and $e\in M_1$, $\lambda_e^{\mathcal{I}}$ ranges over a basis of $\Hom_{\Cat{S}} (i\otimes j,k)$ if a neighbourhood of~$e$ looks like\footnote{When we say that a 2-stratum is labelled with $i\in I$, here and below we mean that we consider the contribution of the $i$-th copy of~$\Bbbk$ in ${A}$.}
\be
\tikzzbox
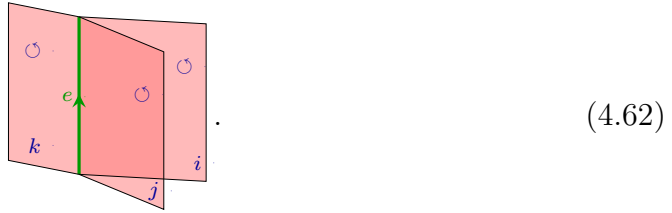
{\begin{tikzpicture}[thick,scale=2.321,color=blue!50!black, baseline=0.0cm, >=stealth, 
	style={x={(-0.6cm,-0.4cm)},y={(1cm,-0.2cm)},z={(0cm,0.9cm)}}]
	\pgfmathsetmacro{\yy}{0.2}
	\coordinate (T) at (0.5, 0.4, 0);
	\coordinate (L) at (0.5, 0, 0);
	\coordinate (R1) at (0.3, 1, 0);
	\coordinate (R2) at (0.7, 1, 0);
	\coordinate (1T) at (0.5, 0.4, 1);
	\coordinate (1L) at (0.5, 0, 1);
	\coordinate (1R1) at (0.3, 1, );
	\coordinate (1R2) at (0.7, 1, );
	%
	\coordinate (p3) at (0.1, 0.1, 0.5);
	\coordinate (p2) at (0.5, 0.95, 0.5);
	\coordinate (p1) at (0.9, 0.1, 0.5);
	%
	\fill [red!50,opacity=0.545] (L) -- (T) -- (1T) -- (1L);
	\fill [red!50,opacity=0.545] (R1) -- (T) -- (1T) -- (1R1);
	\fill [red!50,opacity=0.545] (R2) -- (T) -- (1T) -- (1R2);
	\fill[color=blue!60!black] (0.5,0.25,0.15) circle (0pt) node[left] (0up) { {\scriptsize$k$} };
	\fill[color=blue!60!black] (0.5,0.25,0.75) circle (0pt) node[left] (0up) { {\scriptsize$\circlearrowleft$} };
	\fill[color=blue!60!black] (0.15,0.95,0.04) circle (0pt) node[left] (0up) { {\scriptsize$i$} };
	\fill[color=blue!60!black] (0.15,0.9,0.64) circle (0pt) node[left] (0up) { {\scriptsize$\circlearrowleft$} };
	\fill[color=blue!60!black] (0.55,0.95,0.04) circle (0pt) node[left] (0up) { {\scriptsize$j$} };
	\fill[color=blue!60!black] (0.55,0.9,0.64) circle (0pt) node[left] (0up) { {\scriptsize$\circlearrowleft$} };
	%
	\draw[string=green!60!black, very thick] (T) -- (1T);
	\fill[color=green!60!black] (0.5,0.43,0.5) circle (0pt) node[left] (0up) { {\scriptsize$e$} };
	%
	\draw [black,opacity=1, very thin] (1T) -- (1L) -- (L) -- (T);
	\draw [black,opacity=1, very thin] (1T) -- (1R1) -- (R1) -- (T);
	\draw [black,opacity=1, very thin] (1T) -- (1R2) -- (R2) -- (T);
	\end{tikzpicture}}
. 
\ee
Fourthly, for $x\in M_0$ we write $E_x$ for the list of edges incident on~$x$. 
Then,
if for a fixed colouring~$\mathcal I$ the neighbourhood of $x \in M_0^+$ looks like
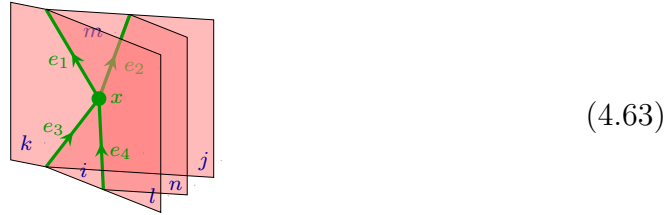
\be 
\tikzzbox{\begin{tikzpicture}[thick,scale=2.321,color=blue!50!black, baseline=0.0cm, >=stealth, 
	style={x={(-0.6cm,-0.4cm)},y={(1cm,-0.2cm)},z={(0cm,0.9cm)}}]
	\pgfmathsetmacro{\yy}{0.2}
	\coordinate (P) at (0.5, \yy, 0);
	\coordinate (R) at (0.625, 0.5 + \yy/2, 0);
	\coordinate (L) at (0.5, 0, 0);
	\coordinate (R1) at (0.25, 1, 0);
	\coordinate (R2) at (0.5, 1, 0);
	\coordinate (R3) at (0.75, 1, 0);
	\coordinate (Pt) at (0.5, \yy, 1);
	\coordinate (Rt) at (0.375, 0.5 + \yy/2, 1);
	\coordinate (Lt) at (0.5, 0, 1);
	\coordinate (R1t) at (0.25, 1, 1);
	\coordinate (R2t) at (0.5, 1, 1);
	\coordinate (R3t) at (0.75, 1, 1);
	\coordinate (alpha) at (0.5, 0.5, 0.5);
	%
	\draw[string=green!60!black, very thick] (alpha) -- (Rt);
	\fill [red!50,opacity=0.545] (L) -- (P) -- (alpha) -- (Pt) -- (Lt);
	\fill [red!50,opacity=0.545] (Pt) -- (Rt) -- (alpha);
	\fill [red!50,opacity=0.545] (Rt) -- (R1t) -- (R1) -- (P) -- (alpha);
	\fill [red!50,opacity=0.545] (Rt) -- (R2t) -- (R2) -- (R) -- (alpha);
	\draw[string=green!60!black, very thick] (alpha) -- (Rt);
	\fill[color=blue!60!black] (0.5,0.59,0.94) circle (0pt) node[left] (0up) { {\scriptsize$m$} };
	\fill[color=green!60!black] (0.5,0.83,0.77) circle (0pt) node[left] (0up) { {\scriptsize$e_2$} };
	\fill [red!50,opacity=0.545] (Pt) -- (R3t) -- (R3) -- (R) -- (alpha);
	\fill [red!50,opacity=0.545] (P) -- (R) -- (alpha);
	%
	\draw[string=green!60!black, very thick] (P) -- (alpha);
	\draw[string=green!60!black, very thick] (R) -- (alpha);
	\draw[string=green!60!black, very thick] (alpha) -- (Pt);
	%
	\fill[color=green!60!black] (alpha) circle (1.2pt) node[right] (0up) { {\scriptsize$x$} };
	\fill[color=blue!60!black] (0.5,0.19,0.15) circle (0pt) node[left] (0up) { {\scriptsize$k$} };
	\fill[color=blue!60!black] (0.5,0.5,0.05) circle (0pt) node[left] (0up) { {\scriptsize$i$} };
	\fill[color=blue!60!black] (0.5,0.89,-0.03) circle (0pt) node[left] (0up) { {\scriptsize$l$} };
	\fill[color=blue!60!black] (0.5,1.04,0.07) circle (0pt) node[left] (0up) { {\scriptsize$n$} };
	\fill[color=blue!60!black] (0.5,1.19,0.26) circle (0pt) node[left] (0up) { {\scriptsize$j$} };
	\fill[color=green!60!black] (0.5,0.37,0.28) circle (0pt) node[left] (0up) { {\scriptsize$e_3$} };
	\fill[color=green!60!black] (0.5,0.75,0.21) circle (0pt) node[left] (0up) { {\scriptsize$e_4$} };
	\fill[color=green!60!black] (0.5,0.4,0.71) circle (0pt) node[left] (0up) { {\scriptsize$e_1$} };
	%
	\draw [black,opacity=1, very thin] (Pt) -- (Lt) -- (L) -- (P);
	\draw [black,opacity=1, very thin] (Pt) -- (Rt);
	\draw [black,opacity=1, very thin] (Rt) -- (R1t) -- (R1) -- (P);
	\draw [black,opacity=1, very thin] (Rt) -- (R2t) -- (R2) -- (R);
	\draw [black,opacity=1, very thin] (Pt) -- (R3t) -- (R3) -- (R);
	\draw [black,opacity=1, very thin] (P) -- (R);
	\end{tikzpicture}}
\ee 
we have $E_x = (e_1, e_2, e_3, e_4)$, and
for fixed $\lambda^{\mathcal I}_{e_1}, \lambda^{\mathcal I}_{e_2}, \lambda^{\mathcal I}_{e_3}, \lambda^{\mathcal I}_{e_4}$ we have 
\be\label{eq:dF}
\mathds{F}_{\mathcal{C}}(\Gamma_x) \Big(\bigotimes_{e\in E_x} \lambda_e^{\mathcal{I}} \Big) = 
\tikzzbox{\begin{tikzpicture}[very thick,scale=0.5,color=blue!50!black, baseline=0cm]
	\draw[-dot-] (0,0) .. controls +(0,1) and +(0,1) .. (-1,0);
	\draw[-dot-] (-1,0) .. controls +(0,-1) and +(0,-1) .. (-2,0);
	\draw[-dot-] (-0.5,0.7) .. controls +(0,1.25) and +(0,1.25) .. (-2,0.7);
	\draw (-2,0.7) -- (-2,0);
	\draw[-dot-] (0,-0.7) .. controls +(0,-1.25) and +(0,-1.25) .. (-1.5,-0.7);
	\draw (0,-0.7) -- (0,0);
	\draw[directed] (-1.25,1.5) .. controls +(0,1) and +(0,1) .. (-2.75,1.5);
	\draw (-2.75,1.5) -- (-2.75,-1.6);
	\draw[directed] (-2.75,-1.6) .. controls +(0,-1.25) and +(0,-1.25) .. (-0.75,-1.6);
	\fill (-0.2,-2.1) circle (0pt) node (meet2)  {{\scriptsize$\widehat{\lambda_{e_1}^{\mathcal{I}}}$}};
	\fill (-1.95,-1.3) circle (0pt) node (meet2){{\scriptsize$\widehat{\lambda_{e_2}^{\mathcal{I}}}$}};};
	\fill (-0.8,2) circle (0pt) node (meet2) {{\scriptsize$\lambda_{e_3}^{\mathcal{I}}$}};
	\fill (-0.6,0.85) circle (0pt) node[right] (meet2) {{\scriptsize$\lambda_{e_4}^{\mathcal{I}}$}};
	%
	%
	\fill (-1.15,0.35) circle (0pt) node (meet2) {{\scriptsize$n$}};
	\fill (-1,-1.2) circle (0pt) node (meet2) {{\scriptsize$m$}};
	\fill (-0.25,-0.5) circle (0pt) node (meet2) {{\scriptsize$l$}};
	\fill (-1.35,2.45) circle (0pt) node (meet2) {{\scriptsize$k$}};
	\fill (-2.25,0.5) circle (0pt) node (meet2) {{\scriptsize$j$}};
	\fill (-0.45,1.45) circle (0pt) node (meet2) {{\scriptsize$i$}};
	\end{tikzpicture}
\, . 
\ee 
Similarly, for $y\in M_0^-$ we have that $\mathds{F}_{\mathcal{C}}(\Gamma_y) (\bigotimes_{e\in E_y} \lambda_e^{\mathcal{I}})$ is given by an appropriate evaluation of a functional as in~\eqref{eq:FCGinv}, i.\,e.\ a diagram of the form
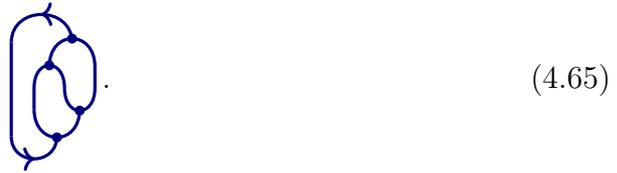
\be\label{eq:dFinv}
\tikzzbox{\begin{tikzpicture}[very thick,scale=0.4,color=blue!50!black, baseline=0cm]
	\draw[-dot-] (0,0) .. controls +(0,1) and +(0,1) .. (1,0);
	\draw[-dot-] (1,0) .. controls +(0,-1) and +(0,-1) .. (2,0);
	\draw[-dot-] (0.5,0.7) .. controls +(0,1.25) and +(0,1.25) .. (2,0.7);
	\draw (2,0.7) -- (2,0);
	\draw[-dot-] (0,-0.7) .. controls +(0,-1.25) and +(0,-1.25) .. (1.5,-0.7);
	\draw (0,-0.7) -- (0,0);
	\draw[directed] (1.25,1.5) .. controls +(0,1.25) and +(0,1.25) .. (-0.75,1.5);
	\draw (-0.75,1.5) -- (-0.75,-1.6);
	\draw[directed] (-0.75,-1.6) .. controls +(0,-1) and +(0,-1) .. (0.75,-1.6);
	\end{tikzpicture}}
\, . 
\ee

In summary, the invariant $\zz^{\text{triv}}_{\Aca}(M)$ has the form 
\be 
\big( 
\!\dim(\mathcal{S})\big)^{-|M_3|} 
\sum_{M_2} d_{i_1} \ldots d_{i_{|M_2|}}
\sum_{M_1} 
\sum_{\{\lambda\}}
\Big( \prod_{M_0^+}  
\tikzzbox{\begin{tikzpicture}[thick,scale=0.25,color=blue!50!black, baseline=-0.04cm]
	\draw[-dot-] (0,0) .. controls +(0,1) and +(0,1) .. (-1,0);
	\draw[-dot-] (-1,0) .. controls +(0,-1) and +(0,-1) .. (-2,0);
	\draw[-dot-] (-0.5,0.7) .. controls +(0,1.25) and +(0,1.25) .. (-2,0.7);
	\draw (-2,0.7) -- (-2,0);
	\draw[-dot-] (0,-0.7) .. controls +(0,-1.25) and +(0,-1.25) .. (-1.5,-0.7);
	\draw (0,-0.7) -- (0,0);
	\draw[directed] (-1.25,1.5) .. controls +(0,1) and +(0,1) .. (-2.75,1.5);
	\draw (-2.75,1.5) -- (-2.75,-1.6);
	\draw[directed] (-2.75,-1.6) .. controls +(0,-1.25) and +(0,-1.25) .. (-0.75,-1.6);
	\end{tikzpicture}}
\Big)
\Big( \prod_{M_0^-}  
\tikzzbox{\begin{tikzpicture}[thick,scale=0.25,color=blue!50!black, baseline=0cm]
	\draw[-dot-] (0,0) .. controls +(0,1) and +(0,1) .. (1,0);
	\draw[-dot-] (1,0) .. controls +(0,-1) and +(0,-1) .. (2,0);
	\draw[-dot-] (0.5,0.7) .. controls +(0,1.25) and +(0,1.25) .. (2,0.7);
	\draw (2,0.7) -- (2,0);
	\draw[-dot-] (0,-0.7) .. controls +(0,-1.25) and +(0,-1.25) .. (1.5,-0.7);
	\draw (0,-0.7) -- (0,0);
	\draw[directed] (1.25,1.5) .. controls +(0,1.25) and +(0,1.25) .. (-0.75,1.5);
	\draw (-0.75,1.5) -- (-0.75,-1.6);
	\draw[directed] (-0.75,-1.6) .. controls +(0,-1) and +(0,-1) .. (0.75,-1.6);
	\end{tikzpicture}}
\Big) . 
\ee 

\medskip

\begin{proposition}
We have $\zz^{\text{triv}}_{\Aca}(M) = \zz^{\text{TV},\Cat{S}}(M)$ for all closed 3-manifolds~$M$. 
\end{proposition}

\begin{proof} 
Recall from~\eqref{eq:ZTVclosed} that $\zz^{\text{TV},\Cat{S}}(M)$ is given by 
\be 
\zz^{\text{TV},\Cat{S}}(M) = 
\big(\! \dim({\Cat{S}})\big)^{-|P_3|} 
\sum_{c\colon P_2\to I} \Big( \prod_{r\in P_2} d_{c(r)}^{\chi(r)} \Big)
\Big( \bigotimes_{e\in P_1} *_e \Big) \Big( \bigotimes_{x\in P_0} \mathds{F}_{\Cat{S}}(\Gamma_x) \Big) , 
\ee 
where we choose the oriented stratified 2-polyhedron~$P$ associated to $M^{t,\Aca}$. 
In this case we have $\chi(r)=1$ for all $r\in P_2 = M_2$, so what remains to be verified is that for fixed $c\colon P \to I$ (and $\mathcal I \in I^{|M_2|}$) the number 
\be 
\Big( \bigotimes_{e\in P_1} *_e \Big) \Big( \bigotimes_{x\in P_0} \mathds{F}_{\Cat{S}}(\Gamma_x) \Big)
\ee 
is indeed the sum over all decorations of the string diagrams as in~\eqref{eq:dF} and~\eqref{eq:dFinv} corresponding to all vertices of~$P$. 

We first note that for a vertex $x\in M_0^+$ and a fixed colouring of the edges of~$\Gamma_x$, we have
\be 
\mathds{F}_{\Cat{S}}(\Gamma_x) 
= 
\sum_{\lambda, \lambda', \mu, \mu'} \;
\tikzzbox{\begin{tikzpicture}[very thick,scale=0.4,color=blue!50!black, baseline=0cm]
	\draw[-dot-] (0,0) .. controls +(0,1) and +(0,1) .. (-1,0);
	\draw[-dot-] (-1,0) .. controls +(0,-1) and +(0,-1) .. (-2,0);
	\draw[-dot-] (-0.5,0.7) .. controls +(0,1.25) and +(0,1.25) .. (-2,0.7);
	\draw (-2,0.7) -- (-2,0);
	\draw[-dot-] (0,-0.7) .. controls +(0,-1.25) and +(0,-1.25) .. (-1.5,-0.7);
	\draw (0,-0.7) -- (0,0);
	\draw[directed] (-1.25,1.5) .. controls +(0,1) and +(0,1) .. (-2.75,1.5);
	\draw (-2.75,1.5) -- (-2.75,-1.6);
	\draw[directed] (-2.75,-1.6) .. controls +(0,-1.25) and +(0,-1.25) .. (-0.75,-1.6);
	\fill (-0.2,-2.1) circle (0pt) node (meet2)  {{\scriptsize$\widehat{\lambda'}$}};
	\fill (-1.95,-1.3) circle (0pt) node (meet2) {{\scriptsize$\widehat{\mu'}$}};
	\fill (-0.8,2) circle (0pt) node (meet2) {{\scriptsize$\lambda$}};
	\fill (-0.6,0.85) circle (0pt) node[right] (meet2) {{\scriptsize$\mu$}};
	\end{tikzpicture}}
\cdot \lambda^* \otimes \mu^* \otimes \widehat{\mu'}^* \otimes \widehat{\lambda'}^* \, , 
\ee 
with $\lambda^*, \widehat{\mu'}^*$ defined in~\eqref{eq:lstar}. 
If $y\in M_0^-$ is a negatively oriented vertex, there is an analogous expression for $\mathds{F}_{\Cat{S}}(\Gamma_y)$. 
Each basis element $\lambda, \lambda', \mu, \mu'$ above corresponds to one of the edges incident on~$x$. 
For example, if~$\lambda$ corresponds to an edge~$e$ which has~$x$ as one endpoint and some vertex $z \in P_0$ as the other endpoint, and if the basis element corresponding to~$e$ at~$z$ is~$\widehat{\kappa}$, then the contraction map~$*_e$ of~\eqref{eq:starcontract} acts as $\lambda \otimes \widehat{\kappa} \mapsto \delta_{\lambda,\kappa}$. 
Hence for a given ${\Cat{S}}$-colouring~$c$, $(\bigotimes_{e\in P_1} *_e) ( \bigotimes_{x\in P_0} \mathds{F}_{\Cat{S}}(\Gamma_x))$ is the sum, over all elements of a basis which can be inserted at the vertices of all~$\Gamma_x$, of the product of the respective evaluations of all $\mathds{F}_{\Cat{S}}(\Gamma_x)$. 
Thus we see that indeed $\zz^{\text{triv}}_{\Aca}(M) = \zz^{\text{TV},\Cat{S}}(M)$.
\end{proof}

\medskip 

Let now~$M$ be an arbitrary 3-bordism. 
By comparing $\zz^{\text{triv}}_{\Aca}(M)$ and $\zz^{\text{TV},\Cat{S}}(M)$ analogously to the above discussion, one finds that the two constructions are identical, except for how they treat 2- and 3-strata of $M^{t,\Aca}$ (or the corresponding 2-polyhedron~$P$) which intersect with the boundary~$\partial M$: 
while the orbifold construction $\zz^{\text{triv}}_{\Aca}$ treats incoming and outgoing boundaries on an equal footing (leading to factors of $d_i^{1/2}$ and 
	 $(\dim\Cat{S})^{-1/2}$
for 2- and 3-strata, respectively), the construction $\zz^{\text{TV},\Cat{S}}$ of \cite{TVire} involves contributions only from the incoming boundary (leading to factors of~$d_i$ 
	 and~$(\dim\Cat{S})^{-1}$). 

This mismatch can be formalised in terms of Euler defect TQFTs, see \cite{Quinnlectures, CRS1}. 
Indeed, in the language of \cite[Ex.\,2.14]{CRS1} the choices for $\zz^{\text{TV},\Cat{S}}$ favouring the incoming boundary correspond to the choice $\lambda=1$ for the Euler TQFT, while the choice for $\zz^{\text{triv}}_{\Aca}$ corresponds to $\lambda=\tfrac{1}{2}$. 
Since both Euler TQFTs are isomorphic \cite{Quinnlectures}, Lemma~2.30 and Remark~3.14 of \cite{CRS1} imply that this isomorphism lifts directly to $\zz^{\text{TV},\Cat{S}}$ and $\zz^{\text{triv}}_{\Aca}$. 
	
	To describe the isomorphism in detail, let $M\colon \Sigma' \rightarrow \Sigma''$ be as in~\eqref{eq:lin-mapTV} with embedded graphs $\Gamma', \Gamma''$ on $\Sigma', \Sigma''$. 
	Let~$t$ be an oriented triangulation of~$M$ extending the duals of the graphs on~$\Sigma'$ and~$\Sigma''$, and for any surface~$\Sigma$ with embedded graph~$\Gamma$ set $f(\Sigma,\Gamma) := (\dim({\Cat{S}}))^{|\Sigma\setminus \Gamma|/2} \prod_{e\in\Gamma_1} d_{c(e)}^{-1/2}$. 
	Then by construction
	\be
	\label{eq:Ztriv-factor-ZTV}
	\zz( M^{t,\orb}) = \frac{f(\Sigma',\Gamma')}{f(\Sigma'',\Gamma'')} \, p(\Gamma', \Gamma'') \, ,
	\ee 
	so the factors $f(\Sigma,\Gamma)$ form an isomorphism between the projective system~\eqref{eq:proj-sys-TV} for $\zz^{\text{TV},\Cat{S}}$ and the corresponding projective system~\eqref{eq:proj-system-orb-bord} for $\zz^{\text{triv}}_{\Aca}$. 
	Thus we obtain an isomorphism between the corresponding limits, which by \eqref{eq:Ztriv-factor-ZTV} is the $\Sigma$-component of a natural isomorphism $\zz^{\text{triv}}_{\Aca} \to \zz^{\text{TV},\Cat{S}}$. 
	Since the map~$f$ is multiplicative under disjoint union by definition, the natural isomorphism is also monoidal. 
We have thus shown: 

\begin{theorem}
\label{thm:ZTVZtrivorb}
For any spherical fusion category~${\Cat{S}}$, there is a
	monoidal 
natural isomorphism between the Turaev--Viro TQFT $\zz^{\text{TV},\Cat{S}}$ and the $\Aca$-orbifold of the trivial 3-dimensional defect TQFT: 
\be 
\zz^{\text{triv}}_{\Aca} 
\cong 
\zz^{\text{TV},\Cat{S}}
\, . 
\ee 
\end{theorem}

\section{Group extensions of modular tensor categories}
\label{sec:Gextensions}

In this section we show that for every suitable $G$-extension of a ribbon fusion category~$\Cat{B}$ there is a corresponding orbifold datum for~$\B$
in the sense of Definition~\ref{def:sodribbon}.
One type of such extensions are $G$-extensions of modular tensor categories~$\Cat{C}$. 
Another interesting situation is when we have a ribbon functor $F \colon \Cat{B} \to \Cat{C}$ and a $G$-crossed extension of~$\B$, as this gives orbifold data in~$\Cat{C}$ (by Proposition~\ref{proposition:push-forward}). 
We consider examples of this where $F$ is the embedding of a  subcategory of~$\Cat{C}$.

In fact our second main result, Theorem~\ref{thm:Gcrossed} (which is Theorem B in the introduction), also holds for certain non-fusion (e.\,g.\ non-semisimple) ribbon categories, see Remark~\ref{rem:GextNonSS}.

\medskip 

We fix a finite group~$G$, a ribbon fusion category~$\B$, and a \textsl{ribbon crossed $G$-category} $\widehat{\Cat{B}} = \bigoplus_{g\in G} \Cat{B}_g$ such that $\B = \Cat{B}_1$ and $\Cat{B}_g \neq 0$ for all $g\in G$. 
Roughly, this means that the tensor product of~$\widehat{\B}$ is compatible with the $G$-grading, there is a monoidal functor $\varphi \colon \underline{G} \to \textrm{Aut}_\otimes(\widehat{\B})$ (where~$\underline{G}$ is~$G$ viewed as a discrete monoidal category), and the twist and braiding of~$\widehat{\B}$ are ``twisted'' by the $G$-action~$\varphi$. 
For details we refer to \cite[Sect.\,VI.2]{TuraevBook2010} from which we deviate in that for us~$G$ acts from the right, i.\,e. 
\be 
\varphi(g) (\B_h) \subset \B_{g^{-1}hg} \quad \text{for all } g,h\in G \, , 
\ee 
and the $G$-twisted braiding has components 
\begin{align}
c_{X,Y} &\equiv 
\begin{tikzpicture}[very thick,scale=0.5,color=blue!50!black, baseline]
\draw[color=blue!50!black] (-1,-1) node[below] (A1) {{\scriptsize$X$}};
\draw[color=blue!50!black] (1,-1) node[below] (A2) {{\scriptsize$Y$}};
\draw[color=blue!50!black] (-1,1) node[above] (B1) {{\scriptsize$Y\vphantom{\varphi}$}};
\draw[color=blue!50!black] (1,1) node[above] (B2) {{\scriptsize$\varphi(h)(X)$}};
\draw[color=blue!50!black] (A2) -- (B1);
\draw[color=white, line width=4pt] (A1) -- (B2);
\draw[color=blue!50!black] (A1) -- (B2);
\end{tikzpicture} 
\colon X\otimes Y \stackrel{\cong}{\lra} Y \otimes \varphi(h)(X) 
\quad \text{if } Y \in \B_h \, ,
\nonumber\\
\tilde c_{Y,X} &\equiv 
\begin{tikzpicture}[very thick,scale=0.5,color=blue!50!black, baseline]
\draw[color=blue!50!black] (-1,-1) node[below] (A1) {{\scriptsize$Y$}};
\draw[color=blue!50!black] (1,-1) node[below] (A2) {{\scriptsize$X$}};
\draw[color=blue!50!black] (-1,1) node[above] (B1) {{\scriptsize$\varphi(h^{-1})(X)$}};
\draw[color=blue!50!black] (1,1) node[above] (B2) {{\scriptsize$Y\vphantom{\varphi}$}};
\draw[color=blue!50!black] (A1) -- (B2);
\draw[color=white, line width=4pt] (A2) -- (B1);
\draw[color=blue!50!black] (A2) -- (B1);
\end{tikzpicture} 
\colon Y\otimes X \stackrel{\cong}{\lra} \varphi(h^{-1})(X) \otimes Y 
\quad \text{if } Y \in \B_h \, . 
\label{eq:Gtwistedbraiding}
\end{align}
Here we wrote~$\tilde c$ for the braiding describing the opposite crossing. 
Up to coherence isomorphism from the group action, the inverse $c_{X,Y}^{-1}$ of the braiding is given by $\tilde c_{Y,\varphi(h)(X)}$.

For every $g\in G$, we now choose a simple object $m_g \in \B_g$ such that $m_1 = \one$, and we set 
\be 
d_{m_g} := \dim(m_g) \in \Bbbk^\times 
\quad \text{for all } g\in G \, . 
\ee 
We furthermore pick a square root $d_{m_g}^{1/2}$.
It is straightforward to verify that 
\be 
\label{eq:AgGextension}
A_g := m_g^* \otimes m_g 
\, , \quad 
\begin{tikzpicture}[very thick,scale=0.9,color=red!50!black, baseline=.9cm]
\draw[line width=0pt] 
(3,0) node[line width=0pt] (D) {{\scriptsize$m_g^*$}}
(2,0) node[line width=0pt] (s) {{\scriptsize$m_g\vphantom{m_g^*}$}}; 
\draw[redirectedred] (D) .. controls +(0,1) and +(0,1) .. (s);
\draw[line width=0pt] 
(3.45,0) node[line width=0pt] (re) {{\scriptsize$m_g\vphantom{m_g^*}$}}
(1.55,0) node[line width=0pt] (li) {{\scriptsize$m_g^*$}}; 
\draw[line width=0pt] 
(2.7,2) node[line width=0pt] (ore) {{\scriptsize$m_g\vphantom{m_g^*}$}}
(2.3,2) node[line width=0pt] (oli) {{\scriptsize$m_g^*$}}; 
\draw (li) .. controls +(0,0.75) and +(0,-0.25) .. (2.3,1.25);
\draw (2.3,1.25) -- (oli);
\draw (re) .. controls +(0,0.75) and +(0,-0.25) .. (2.7,1.25);
\draw (2.7,1.25) -- (ore);
\end{tikzpicture}
\!\! , \quad 
\begin{tikzpicture}[very thick,scale=1.0,color=red!50!black, baseline=-.4cm,rotate=180]
\draw[line width=0pt] 
(3,0) node[line width=0pt] (D) {{\scriptsize$m_g^*$}}
(2,0) node[line width=0pt] (s) {{\scriptsize$m_g\vphantom{m_g^*}$}}; 
\draw[directedred] (D) .. controls +(0,1) and +(0,1) .. (s);
\end{tikzpicture}
\!\!, \quad 
\frac{1}{d_{m_g}}\cdot\!\!
\begin{tikzpicture}[very thick,scale=0.9,color=red!50!black, baseline=-0.9cm, rotate=180]
\draw[line width=0pt] 
(3,0) node[line width=0pt] (D) {{\scriptsize${m_g}\vphantom{m_g^*}$}}
(2,0) node[line width=0pt] (s) {{\scriptsize$m_g^*$}}; 
\draw[directedred] (s) .. controls +(0,1) and +(0,1) .. (D);
\draw[line width=0pt] 
(3.45,0) node[line width=0pt] (re) {{\scriptsize$m_g^*$}}
(1.55,0) node[line width=0pt] (li) {{\scriptsize${m_g}\vphantom{m_g^*}$}}; 
\draw[line width=0pt] 
(2.7,2) node[line width=0pt] (ore) {{\scriptsize$m_g^*$}}
(2.3,2) node[line width=0pt] (oli) {{\scriptsize$m_g\vphantom{m_g^*}$}}; 
\draw (li) .. controls +(0,0.75) and +(0,-0.25) .. (2.3,1.25);
\draw (2.3,1.25) -- (oli);
\draw (re) .. controls +(0,0.75) and +(0,-0.25) .. (2.7,1.25);
\draw (2.7,1.25) -- (ore);
\end{tikzpicture}
\!\!, \quad 
d_{m_g}\cdot\!\!
\begin{tikzpicture}[very thick,scale=1.0,color=red!50!black, baseline=.4cm]
\draw[line width=0pt] 
(3,0) node[line width=0pt] (D) {{\scriptsize$m_g\vphantom{m_g^*}$}}
(2,0) node[line width=0pt] (s) {{\scriptsize$m_g^*$}}; 
\draw[redirectedred] (s) .. controls +(0,1) and +(0,1) .. (D);
\end{tikzpicture}
\ee 
is a $\Delta$-separable symmetric Frobenius algebra in~$\B$ for all $g \in G$. 
Moreover we have $A_{gh}$-$(A_g\otimes A_h)$-bimodules $T_{g,h} \in \Cat{B}$ given by
\begin{align}
T_{g,h} :=m_{gh}^* \otimes m_g \otimes m_h 
\quad \text{with actions} \quad 
& 
\begin{tikzpicture}[very thick,scale=0.75,color=blue!50!black, baseline]
\draw (0,-1) node[below] (X) {{\scriptsize$T_{g,h}$}};
\draw[color=green!50!black] (-0.8,-1) node[below] (A1) {{\scriptsize$A_{gh}$}};
\draw (0,1) node[right] (Xu) {};
\draw[color=green!50!black] (A1) .. controls +(0,0.5) and +(-0.5,-0.5) .. (0,0.3);
\draw (0,-1) -- (0,1); 
\fill[color=blue!50!black] (0,0.3) circle (2.9pt) node (meet2) {};
\end{tikzpicture} 
\stackrel{\textrm{def}}{=}
\begin{tikzpicture}[very thick,scale=0.75,color=red!50!black, baseline]
\draw (-1.75,-1) node[below] (mghs) {{\scriptsize$m_{gh}^*$}};
\draw (-1,-1) node[below] (mgh) {{\scriptsize$m_{gh}\vphantom{m_{gh}^*}$}};
\draw (0.75,-1) node[below] (mg) {{\scriptsize$m_g\vphantom{m_{gh}^*}$}};
\draw (1.5,-1) node[below] (mh) {{\scriptsize$m_h\vphantom{m_{gh}^*}$}};
\draw (0,-1) node[below] (A1) {{\scriptsize$m_{gh}^*$}};
\draw[directedred] (-1,-1) .. controls +(0,0.75) and +(0,0.75) .. (0,-1);
\draw (-1.75,-1) to[out=90, in=-90] (0,1);
\draw (mg) -- (0.75,1); 
\draw (mh) -- (1.5,1); 
\end{tikzpicture} 
, 
\nonumber 
\\
&  
\begin{tikzpicture}[very thick,scale=0.75,color=blue!50!black, baseline]
\draw (0,-1) node[below] (X) {{\scriptsize$T_{g,h}$}};
\draw[color=green!50!black] (0.8,-1) node[below] (A1) {{\scriptsize$A_h$}};
\draw (0,1) node[right] (Xu) {};
\draw[color=green!50!black] (A1) .. controls +(0,0.5) and +(0.5,-0.5) .. (0,0.3);
\draw (0,-1) -- (0,1); 
\fill[color=blue!50!black] (0,0.3) circle (2.9pt) node (meet2) {};
\fill[color=black] (0.2,0.5) circle (0pt) node (meet) {{\tiny$2$}};
\end{tikzpicture} 
\stackrel{\textrm{def}}{=}
\begin{tikzpicture}[very thick,scale=0.75,color=red!50!black, baseline]
\draw (1.75,-1) node[below] (mghs) {{\scriptsize$m_{h}\vphantom{m_{gh}^*}$}};
\draw (1,-1) node[below] (mgh) {{\scriptsize$m_{h}^*\vphantom{m_{gh}^*}$}};
\draw (-0.75,-1) node[below] (mg) {{\scriptsize$m_g\vphantom{m_{gh}^*}$}};
\draw (-1.5,-1) node[below] (mh) {{\scriptsize$m_{gh}^*\vphantom{m_{gh}^*}$}};
\draw (0,-1) node[below] (A1) {{\scriptsize$m_{h}\vphantom{m_{gh}^*}$}};
\draw[redirectedred] (1,-1) .. controls +(0,0.75) and +(0,0.75) .. (0,-1);
\draw (1.75,-1) to[out=90, in=-90] (0,1);
\draw (mg) -- (-0.75,1); 
\draw (mh) -- (-1.5,1); 
\end{tikzpicture} 
, 
\nonumber 
\\ 
&  
\begin{tikzpicture}[very thick,scale=0.75,color=blue!50!black, baseline]
\draw (0,-1) node[below] (X) {{\scriptsize$T_{g,h}$}};
\draw[color=green!50!black] (0.8,-1) node[below] (A1) {{\scriptsize$A_g$}};
\draw (0,1) node[right] (Xu) {};
\draw[color=green!50!black] (A1) .. controls +(0,0.5) and +(0.5,-0.5) .. (0,0.3);
\draw (0,-1) -- (0,1); 
\fill[color=blue!50!black] (0,0.3) circle (2.9pt) node (meet2) {};
\fill[color=black] (0.2,0.5) circle (0pt) node (meet) {{\tiny$1$}};
\end{tikzpicture} 
\stackrel{\textrm{def}}{=}
\begin{tikzpicture}[very thick,scale=0.75,color=red!50!black, baseline]
\draw (1.75,-1) node[below] (mghs) {{\scriptsize$m_{g}\vphantom{m_{gh}^*}$}};
\draw (1,-1) node[below] (mgh) {{\scriptsize$m_{g}^*\vphantom{m_{gh}^*}$}};
\draw (-0.75,-1) node[below] (mg) {{\scriptsize$m_g\vphantom{m_{gh}^*}$}};
\draw (-1.5,-1) node[below] (mh) {{\scriptsize$m_{gh}^*\vphantom{m_{gh}^*}$}};
\draw (0,-1) node[below] (A1) {{\scriptsize$m_{h}\vphantom{m_{gh}^*}$}};
\draw[redirectedred] (1,-1) .. controls +(0,0.75) and +(0,0.75) .. (-0.75,-1);
\draw (1.75,-1) to[out=90, in=-90] (-0.75,1);
\draw[color=white, line width=4pt] (0,-1) to[out=90, in=-90] (0,1);
\draw (0,-1) to[out=90, in=-90] (0,1);
\draw (mh) -- (-1.5,1); 
\end{tikzpicture} 
.
\label{eq:Tghactions}
\end{align}
Above we use string diagram notation for morphisms in the $G$-crossed category $\widehat{\B}$. 
By \eqref{eq:Gtwistedbraiding}, the object label attached to a string changes at crossings. For example, the first crossing in the action of $A_g$ is the inverse braiding $m_g \otimes m_h \to \varphi(g^{-1})(m_h) \otimes m_g$, and the second crossing is the braiding $\varphi(g^{-1})(m_h) \otimes m_g \to m_g \otimes m_h$ (composed with a coherence isomorphism for~$\varphi$). 
Thus the string labelled $m_h$ at the bottom is labelled by $\varphi(g^{-1})(m_h) \in \B_{ghg^{-1}}$ between the crossings and again by~$m_h$ at the top.

One checks that indeed (cf.\ \eqref{eq:A1A2comp})
\be 
\begin{tikzpicture}[very thick,scale=0.75,color=blue!50!black, baseline,xscale=-1]
\draw (0,-1) node[below] (X) {{\scriptsize$T_{g,h}$}};
\draw[color=green!50!black] (-0.75,-1) node[below] (X) {{\scriptsize$A_g$}};	
\draw[color=green!50!black] (-1.5,-1) node[below] (X) {{\scriptsize$A_h$}};
\draw (0,1) node[right] (Xu) {};
\draw (0,-1) -- (0,1); 
\draw[color=green!50!black] (-0.75,-1) .. controls +(0,0.25) and +(-0.25,-0.25) .. (0,-0.25);
\draw[color=green!50!black] (-1.5,-1) .. controls +(0,0.5) and +(-0.5,-0.5) .. (0,0.6);
\fill[color=blue!50!black] (0,-0.25) circle (2.9pt) node (meet) {};
\fill[color=blue!50!black] (0,0.6) circle (2.9pt) node (meet2) {};
\end{tikzpicture} 
=
\begin{tikzpicture}[very thick,scale=0.75,color=blue!50!black, baseline,xscale=-1]
\draw (0,-1) node[below] (X) {{\scriptsize$T_{g,h}$}};
\draw[color=green!50!black] (-0.75,-1) node[below] (A1) {{\scriptsize$A_g$}};
\draw[color=green!50!black] (-1.5,-1) node[below] (A2) {{\scriptsize$A_h$}};
\draw (0,1) node[right] (Xu) {};
\draw[color=green!50!black] (A1) .. controls +(0,0.5) and +(-0.5,-0.5) .. (0,0.6);
\draw[color=white, line width=4pt] (A2) .. controls +(0,0.5) and +(-0.25,-0.25) .. (0,-0.25);
\draw[color=green!50!black] (A2) .. controls +(0,0.5) and +(-0.25,-0.25) .. (0,-0.25);
\draw (0,-1) -- (0,1); 
\fill[color=blue!50!black] (0,-0.25) circle (2.9pt) node (meet) {};
\fill[color=blue!50!black] (0,0.6) circle (2.9pt) node (meet2) {};
\end{tikzpicture} 
. 
\ee 
Setting $A:=\bigoplus_{g_\in G} A_g$, it follows that $T := \bigoplus_{g,h\in G} T_{g,h}$ is an $A$-$(A\otimes A)$-bimodule. 

Now we define component maps 
\begin{align}
\alpha_{g,h,k} \colon T_{g,hk} \otimes T_{h,k} \lra T_{gh,k} \otimes T_{g,h} \, , \quad
\bar\alpha_{g,h,k} \colon T_{gh,k} \otimes T_{g,h} \lra T_{g,hk} \otimes T_{h,k} 
\end{align}
by 
\be
\label{eq:alphacompG}
\alpha_{g,h,k}
\; \stackrel{\textrm{def}}{=} \; 
\begin{tikzpicture}[very thick,scale=0.75,color=red!50!black, baseline]
\draw[decoration={markings, mark=at position 0.5 with {\arrow{>}},}, postaction={decorate}] 
(0,1.5) -- (0,-1.5);
\draw[decoration={markings, mark=at position 0.5 with {\arrow{<}},}, postaction={decorate}] 
(0.5,1.5) .. controls +(0,-1) and +(0,-1) .. (2,1.5);
\draw[decoration={markings, mark=at position 0.5 with {\arrow{>}},}, postaction={decorate}] 
(1,-1.5) .. controls +(0,1) and +(0,1) .. (2,-1.5);
\draw[decoration={markings, mark=at position 0.5 with {\arrow{>}},}, postaction={decorate}] 
(0.5,-1.5) to[out=90, in=-90] (2.5,1.5);	
\draw[decoration={markings, mark=at position 0.75 with {\arrow{>}},}, postaction={decorate}]
(2.5,-1.5) to[out=90, in=-90] (3,1.5);	
\draw[color=white, line width=4pt] (3,-1.5) to[out=90, in=-55] (1,1.5);
\draw[decoration={markings, mark=at position 0.5 with {\arrow{>}},}, postaction={decorate}]
(3,-1.5) to[out=90, in=-55] (1,1.5);
%
\fill (1.1,0) circle (0pt) node (meet2) {{\scriptsize$g$}};
\fill (0.4,0) circle (0pt) node (meet2) {{\scriptsize$ghk$}};
\fill (2.1,0) circle (0pt) node (meet2) {{\scriptsize$k$}};
\fill (3.05,0.35) circle (0pt) node (meet2) {{\scriptsize$h$}};
\fill (1.5,-1.1) circle (0pt) node (meet2) {{\scriptsize$hk$}};
\fill (0.9,0.55) circle (0pt) node (meet2) {{\scriptsize$gh$}};
\end{tikzpicture} 
\, , \quad 
\bar\alpha_{g,h,k}
\; \stackrel{\textrm{def}}{=} \; 
\begin{tikzpicture}[very thick,scale=0.75,color=red!50!black, baseline]
\draw[decoration={markings, mark=at position 0.5 with {\arrow{>}},}, postaction={decorate}] 
(0,1.5) -- (0,-1.5);
\draw[decoration={markings, mark=at position 0.5 with {\arrow{<}},}, postaction={decorate}] 
(1,1.5) .. controls +(0,-1) and +(0,-1) .. (2,1.5);
\draw[decoration={markings, mark=at position 0.5 with {\arrow{>}},}, postaction={decorate}] 
(0.5,-1.5) .. controls +(0,1) and +(0,1) .. (2,-1.5);
\draw[decoration={markings, mark=at position 0.75 with {\arrow{>}},}, postaction={decorate}]
(2.5,-1.5) to[out=90, in=-90] (0.5,1.5);	
\draw[decoration={markings, mark=at position 0.5 with {\arrow{>}},}, postaction={decorate}]
(3,-1.5) to[out=90, in=-90] (2.5,1.5);
\draw[color=white, line width=4pt] (1,-1.5) to[out=55, in=-90] (3,1.5);	
\draw[decoration={markings, mark=at position 0.5 with {\arrow{>}},}, postaction={decorate}] 
(1,-1.5) to[out=55, in=-90] (3,1.5);	
%
\fill (0.4,0) circle (0pt) node (meet2) {{\scriptsize$ghk$}};
\fill (0.64,0.5) circle (0pt) node (meet2) {{\scriptsize$g$}};
\fill (1.5,1.1) circle (0pt) node (meet2) {{\scriptsize$hk$}};
\fill (2.15,0.05) circle (0pt) node (meet2) {{\scriptsize$k$}};
\fill (3,0.15) circle (0pt) node (meet2) {{\scriptsize$h$}};
\fill (0.8,-0.6) circle (0pt) node (meet2) {{\scriptsize$gh$}};
\end{tikzpicture} . 
\ee 
Here and below we use the following shorthand notation in labelling string diagrams. 
A label~$g$ on a string indicates that its source and target object is~$m_g$ (or~$m_g^*$, depending on orientation). 
We stress that this is independent of the position of the label~$g$ along the string. 
For example, passing along the string labelled~$k$ in the diagram for $\bar\alpha_{g,h,k}$, the components of the string in the complement of the crossings should be labelled by the objects~$m_k$, $\varphi(h^{-1}g^{-1})(m_k)$, $\varphi(h^{-1})(m_k)$ and~$m_k$, in this order.

The components $\alpha_{g,h,k}$ and $\bar\alpha_{g,h,k}$ assemble into module maps
$\alpha := \sum_{g,h,k\in G} \alpha_{g,h,k}$ 
and 
$\bar\alpha := \sum_{g,h,k\in G} \bar\alpha_{g,h,k}$, 
as can be checked by verifying identities as in~\eqref{eq:alphaactioncomp}--\eqref{eq:alphamodulemap}. 

Finally we define $\psi\in \End_{AA}(A)$ and $\phi\in \End(\one) = \Bbbk$ by 
\be 
\label{eq:psiphiG}
\psi\big|_{A_g} 
\stackrel{\textrm{def}}{=}
d_{m_g}^{-1/2} \cdot \text{id}_{A_g} 
\, , \quad 
	 \phi^2
:= \frac{1}{|G|} \, .
\ee 

\begin{theorem}
\label{thm:Gcrossed} 
Let~$\B = \B_1$ be the neutral component of a ribbon crossed $G$-category~$\widehat{\B}$ as above. 
Then for every choice of simple objects $\{ m_g \in \B_g\}_{g\in G}$ the tuple 
$\A^m := (A, T, \alpha, \bar\alpha, \psi, \phi)$
defined in \eqref{eq:AgGextension}--\eqref{eq:psiphiG} is a special orbifold datum for~$\B$. 
\end{theorem}

\begin{proof}
We have to show that $\A^m$ satisfies the conditions \eqref{eq:347n}--\eqref{eq:351n}. 
Once the latter are written out in terms of the algebra actions~\eqref{eq:Tghactions}, the component maps~\eqref{eq:alphacompG} and the definition~\eqref{eq:psiphiG} of~$\psi$ and~$\phi$, the verification becomes a straightforward exercise in the graphical calculus for ribbon crossed $G$-categories. 
Here we provide details for only two conditions; the remainder is checked analogously. 

One of the more involved conditions is the second identity of~\eqref{eq:349n}. 
In components, its right-hand side is 
\be 
\label{eq:RHSInv2v2}
\begin{tikzpicture}[very thick,scale=0.75,color=blue!50!black, baseline=0.75cm]
\draw (0,-1.5) -- (0,3.5); 
\draw (2,-1.5) -- (2,3.5); 
%
\draw[-dot-, color=green!50!black] (0.5,2) .. controls +(0,-1) and +(0,-1) .. (1.5,2);
\draw[color=green!50!black] (0.5,2) to[out=90, in=-45] (0,3);
\draw[color=green!50!black] (1.5,2) to[out=90, in=145] (2,2); 
\draw[color=green!50!black] (1,0.75) node[Odot] (end) {}; 
\draw[color=green!50!black] (1,0.8) -- (1,1.3);
%
\fill[color=black] (0,0) circle (2.9pt) node[right] (meet) {{\scriptsize$\!\psi_2^{-2}$}};
\fill (0,3) circle (2.9pt) node[left] (meet) {};
\fill (2,2) circle (2.9pt) node[left] (meet) {};
%
\fill[color=black] (0.2,3.1) circle (0pt) node (meet) {{\tiny$2$}};
\fill[color=black] (1.8,1.9) circle (0pt) node (meet) {{\tiny$2$}};
\fill[color=green!50!black] (0.85,2.0) circle (0pt) node (M) {{\scriptsize$A_k$}};
\fill (0.55,-1.5) circle (0pt) node (M) {{\scriptsize$T_{gh,k}\vphantom{T^*_{ghk}}$}};
\fill (2.45,-1.5) circle (0pt) node (M) {{\scriptsize$T_{h,k}^*\vphantom{T^*_{ghk}}$}};
\end{tikzpicture} 
\; = \; 
\begin{tikzpicture}[very thick,scale=0.75,color=red!50!black, baseline=0.75cm]
\draw[decoration={markings, mark=at position 0.5 with {\arrow{>}},}, postaction={decorate}] 
(0,3.5) -- (0,-1.5);
\draw[decoration={markings, mark=at position 0.75 with {\arrow{>}},}, postaction={decorate}] 
(0.5,-1.5) -- (0.5,3.5);
\draw[decoration={markings, mark=at position 0.5 with {\arrow{<}},}, postaction={decorate}] 
(2,1.5) .. controls +(0,-1) and +(0,-1) .. (3,1.5);
\draw[decoration={markings, mark=at position 0.5 with {\arrow{>}},}, postaction={decorate}] 
(1.5,1.5) .. controls +(0,-1.75) and +(0,-1.75) .. (3.5,1.5);
\draw[decoration={markings, mark=at position 0.53 with {\arrow{<}},}, postaction={decorate}] 
(1.5,1.5) .. controls +(0,0.25) and +(0,0.25) .. (1,1.5);
\draw[decoration={markings, mark=at position 0.53 with {\arrow{>}},}, postaction={decorate}] 
(3.5,1.5) .. controls +(0,0.25) and +(0,0.25) .. (4,1.5);
\draw (1,1.5) -- (1,-1.5); 
\draw (4,1.5) -- (4,-1.5); 
\draw (2,1.5) to[out=90, in=-90] (1,3.5);	
\draw (3,1.5) to[out=90, in=-90] (4,3.5);	
\draw[decoration={markings, mark=at position 0.25 with {\arrow{>}},}, postaction={decorate}] 
(4.5,3.5) -- (4.5,-1.5);
\draw[decoration={markings, mark=at position 0.5 with {\arrow{>}},}, postaction={decorate}] 
(5,-1.5) -- (5,3.5);
%
\fill (-0.4,0.8) circle (0pt) node (meet2) {{\scriptsize$ghk$}};
\fill (0.8,2.4) circle (0pt) node (meet2) {{\scriptsize$gh$}};
\fill (2.5,-0.1) circle (0pt) node (meet2) {{\scriptsize$k$}};
\fill (2.5,1.15) circle (0pt) node (meet2) {{\scriptsize$k$}};
\fill (4.3,2.15) circle (0pt) node (meet2) {{\scriptsize$h$}};
\fill (5.3,1.1) circle (0pt) node (meet2) {{\scriptsize$hk$}};
\end{tikzpicture} 
\; = \; 
\begin{tikzpicture}[very thick,scale=0.75,color=red!50!black, baseline=0.75cm]
\draw[decoration={markings, mark=at position 0.5 with {\arrow{>}},}, postaction={decorate}] 
(0,3.5) -- (0,-1.5);
\draw[decoration={markings, mark=at position 0.75 with {\arrow{>}},}, postaction={decorate}] 
(0.5,-1.5) -- (0.5,3.5);
\draw[decoration={markings, mark=at position 0.5 with {\arrow{<}},}, postaction={decorate}] 
(1,3.5) .. controls +(0,-2.5) and +(0,-2.5) .. (4,3.5);
\draw[decoration={markings, mark=at position 0.5 with {\arrow{>}},}, postaction={decorate}] 
(1,-1.5) .. controls +(0,2.5) and +(0,2.5) .. (4,-1.5);
\draw[decoration={markings, mark=at position 0.25 with {\arrow{>}},}, postaction={decorate}] 
(4.5,3.5) -- (4.5,-1.5);
\draw[decoration={markings, mark=at position 0.5 with {\arrow{>}},}, postaction={decorate}] 
(5,-1.5) -- (5,3.5);
%
\fill (-0.4,0.8) circle (0pt) node (meet2) {{\scriptsize$ghk$}};
\fill (0.8,2.4) circle (0pt) node (meet2) {{\scriptsize$gh$}};
\fill (2.5,0.1) circle (0pt) node (meet2) {{\scriptsize$k$}};
\fill (2.3,1.9) circle (0pt) node (meet2) {{\scriptsize$k$}};
\fill (4.3,2.15) circle (0pt) node (meet2) {{\scriptsize$h$}};
\fill (5.3,1.1) circle (0pt) node (meet2) {{\scriptsize$hk$}};
\end{tikzpicture} 
\, , 
\ee 
while for the left-hand side we compute: 
\be 
\!\!\
\begin{tikzpicture}[very thick,scale=0.75,color=blue!50!black, baseline=0.75cm]
\draw (-0.25,2) -- (-0.25,5); 
\draw[directed] (0.25,2.2) .. controls +(0,1) and +(0,1) .. (1.75,2.2);
\draw[redirected] (0.25,1.8) .. controls +(0,-0.75) and +(0,-0.75) .. (1.25,1.8);
\draw (1.75,2.2) -- (1.75,-0.2);
\draw (-0.25,0) -- (-0.25,2); 
\draw[redirected] (0.25,-0.2) .. controls +(0,-1) and +(0,-1) .. (1.75,-0.2);
\draw[directed] (0.25,0.2) .. controls +(0,0.75) and +(0,0.75) .. (1.25,0.2);
\draw[color=white, line width=4pt] (1.25,0.2) -- (1.25,-4);

\draw (1.25,0.2) -- (1.25,-4);
\draw[color=white, line width=4pt] (1.25,1.8) -- (1.25,5);
\draw (1.25,1.8) -- (1.25,5);
\draw (-0.25,-4) -- (-0.25,0); 
%
\fill[color=white] (0,0) node[inner sep=4pt,draw, rounded corners=1pt, fill, color=white] (R2) {{\scriptsize$\alpha_{g,h,k}$}};
\draw[line width=1pt, color=black] (0,0) node[inner sep=4pt, draw, semithick, rounded corners=1pt] (R) {{\scriptsize$\bar\alpha_{g,h,k}$}};
%
%
\fill[color=white] (0,2) node[inner sep=4pt,draw, rounded corners=1pt, fill, color=white] (R2) {{\scriptsize$\alpha_{g,h,k}$}};
\draw[line width=1pt, color=black] (0,2) node[inner sep=4pt, draw, semithick, rounded corners=1pt] (R) {{\scriptsize$\alpha_{g,h,k}$}};
%
%
\fill[color=black] (-0.25,1) circle (2.9pt) node[left] (meet) {{\scriptsize$\!\psi_1^2\!$}};
%
\fill (0.3,-4) circle (0pt) node (M) {{\scriptsize$T_{gh,k}\vphantom{T^*_{gh}}$}};
\fill (1.7,-4) circle (0pt) node (M) {{\scriptsize$T^*_{h,k}\vphantom{T^*_{gh}}$}};
\end{tikzpicture} 
\!\!\!\!=
\frac{1}{d_{m_g}} \cdot
\begin{tikzpicture}[very thick,scale=0.75,color=red!50!black, baseline=0.75cm]
\draw[decoration={markings, mark=at position 0.5 with {\arrow{>}},}, postaction={decorate}] 
(0,5) -- (0,-4);
\draw (0.5,5) -- (0.5,3.5);
\draw (1,5) -- (1,3.5);
\draw (1,-2.5) -- (1,-4);
\draw (0.5,-2.5) -- (0.5,-4);
\draw[decoration={markings, mark=at position 0.1 with {\arrow{<}},}, postaction={decorate}] (0.5,1.5) -- (0.5,-0.5);
\draw (1,1.5) -- (1,-0.5);
%
%
\draw (0.5,1.5) to[out=90, in=-90] (2.5,3.5);	
\draw (2.5,1.5) to[out=90, in=-90] (3,3.5);	
\draw[decoration={markings, mark=at position 0.5 with {\arrow{<}},}, postaction={decorate}] 
(0.5,3.5) .. controls +(0,-1) and +(0,-1) .. (2,3.5);
\draw[decoration={markings, mark=at position 0.5 with {\arrow{>}},}, postaction={decorate}] 
(1,1.5) .. controls +(0,0.5) and +(0,0.5) .. (2,1.5);
\draw[color=white, line width=4pt] (3,1.5) to[out=90, in=-90] (1,3.5);
\draw (3,1.5) to[out=90, in=-90] (1,3.5);	
%
%
\draw[decoration={markings, mark=at position 0.5 with {\arrow{<}},}, postaction={decorate}] 
(1,-0.5) .. controls +(0,-0.5) and +(0,-0.5) .. (2,-0.5);
\draw[decoration={markings, mark=at position 0.5 with {\arrow{>}},}, postaction={decorate}] 
(0.5,-2.5) .. controls +(0,1) and +(0,1) .. (2,-2.5);
\draw (2.5,-2.5) to[out=90, in=-90] (0.5,-0.5);	
\draw (3,-2.5) to[out=90, in=-90] (2.5,-0.5);	
\draw[color=white, line width=4pt] (1,-2.5) to[out=90, in=-90] (3,-0.5);
\draw (1,-2.5) to[out=90, in=-90] (3,-0.5);	
%
%
\draw[decoration={markings, mark=at position 0.5 with {\arrow{>}},}, postaction={decorate}] (3,3.5) .. controls +(0,0.75) and +(0,0.75) .. (7,3.5);
\draw[decoration={markings, mark=at position 0.5 with {\arrow{>}},}, postaction={decorate}] (2.5,3.5) .. controls +(0,1.25) and +(0,1.25) .. (7.5,3.5);
\draw[decoration={markings, mark=at position 0.5 with {\arrow{<}},}, postaction={decorate}] (2,3.5) .. controls +(0,1.75) and +(0,1.75) .. (8,3.5);
\draw (8,3.5) -- (8,-2.5); 
\draw (7.5,3.5) -- (7.5,-2.5); 
\draw (7,3.5) -- (7,-2.5); 
\draw[decoration={markings, mark=at position 0.5 with {\arrow{>}},}, postaction={decorate}] (7,-2.5) .. controls +(0,-0.75) and +(0,-0.75) .. (3,-2.5);
\draw[decoration={markings, mark=at position 0.5 with {\arrow{>}},}, postaction={decorate}] (7.5,-2.5) .. controls +(0,-1.25) and +(0,-1.25) .. (2.5,-2.5);
\draw[decoration={markings, mark=at position 0.5 with {\arrow{<}},}, postaction={decorate}] (8,-2.5) .. controls +(0,-1.75) and +(0,-1.75) .. (2,-2.5);%
%
%
\draw[decoration={markings, mark=at position 0.5 with {\arrow{<}},}, postaction={decorate}] (3,1.5) .. controls +(0,-0.5) and +(0,-0.5) .. (5.5,1.5);
\draw[decoration={markings, mark=at position 0.5 with {\arrow{<}},}, postaction={decorate}] (2.5,1.5) .. controls +(0,-0.75) and +(0,-0.75) .. (6,1.5);
\draw[decoration={markings, mark=at position 0.5 with {\arrow{>}},}, postaction={decorate}] (2,1.5) .. controls +(0,-1) and +(0,-1) .. (6.5,1.5);%
\draw[color=white, line width=4pt] (5.5,1.5) -- (5.5,5); 
\draw[color=white, line width=4pt] (6,1.5) -- (6,5); 
\draw[color=white, line width=4pt] (6.5,1.5) -- (6.5,5); 
\draw[color=white, line width=4pt] (5.5,-0.5) -- (5.5,-4); 
\draw[color=white, line width=4pt] (6,-0.5) -- (6,-4); 
\draw[color=white, line width=4pt] (6.5,-0.5) -- (6.5,-4); 
\draw (5.5,1.5) -- (5.5,5); 
\draw (6,1.5) -- (6,5); 
\draw (6.5,1.5) -- (6.5,5); 
\draw (5.5,-0.5) -- (5.5,-4); 
\draw (6,-0.5) -- (6,-4); 
\draw (6.5,-0.5) -- (6.5,-4); 
\draw[decoration={markings, mark=at position 0.5 with {\arrow{>}},}, postaction={decorate}] (3,-0.5) .. controls +(0,0.5) and +(0,0.5) .. (5.5,-0.5);
\draw[decoration={markings, mark=at position 0.5 with {\arrow{>}},}, postaction={decorate}] (2.5,-0.5) .. controls +(0,0.75) and +(0,0.75) .. (6,-0.5);
\draw[decoration={markings, mark=at position 0.5 with {\arrow{<}},}, postaction={decorate}] (2,-0.5) .. controls +(0,1) and +(0,1) .. (6.5,-0.5);%
%
\fill (-0.4,0.8) circle (0pt) node (meet2) {{\scriptsize$ghk$}};
\fill (0.8,-1.65) circle (0pt) node (meet2) {{\scriptsize$gh$}};
\fill (4.3,-0.4) circle (0pt) node (meet2) {{\scriptsize$k$}};
\fill (4.05,1.4) circle (0pt) node (meet2) {{\scriptsize$k$}};
\fill (0.3,1.5) circle (0pt) node (meet2) {{\scriptsize$g$}};
\fill (5,3.7) circle (0pt) node (meet2) {{\scriptsize$h$}};
\fill (1.5,-0.5) circle (0pt) node (meet2) {{\scriptsize$hk$}};
\end{tikzpicture} 
=\!\!
\begin{tikzpicture}[very thick,scale=0.75,color=red!50!black, baseline=0.75cm]
\draw[decoration={markings, mark=at position 0.5 with {\arrow{>}},}, postaction={decorate}] 
(0,5) -- (0,-4);
\draw[decoration={markings, mark=at position 0.5 with {\arrow{<}},}, postaction={decorate}] 
(3.5,5) -- (3.5,-4);
\draw[decoration={markings, mark=at position 0.5 with {\arrow{<}},}, postaction={decorate}] 
(0.5,5) -- (0.5,-4);
\draw[decoration={markings, mark=at position 0.5 with {\arrow{<}},}, postaction={decorate}] (1,5) .. controls +(0,-2) and +(0,-2) .. (2.5,5);
\draw[decoration={markings, mark=at position 0.5 with {\arrow{>}},}, postaction={decorate}] (1,-4) .. controls +(0,2) and +(0,2) .. (2.5,-4);
\draw[decoration={markings, mark=at position 0.5 with {\arrow{<}},}, postaction={decorate}] (2.5,-1) .. controls +(0,0.75) and +(0,0.75) .. (1.5,-1);
\draw[decoration={markings, mark=at position 0.5 with {\arrow{<}},}, postaction={decorate}] (1.5,-1) .. controls +(0,-1) and +(0,-1) .. (3,-1);
\draw[color=white, line width=4pt] (3,-4) to[out=90, in=-90] (2.5,-1);	
\draw (3,-4) to[out=90, in=-90] (2.5,-1);	
\draw[decoration={markings, mark=at position 0.5 with {\arrow{>}},}, postaction={decorate}] (2.5,2) .. controls +(0,-0.75) and +(0,-0.75) .. (1.5,2);
\draw[decoration={markings, mark=at position 0.5 with {\arrow{>}},}, postaction={decorate}] (1.5,2) .. controls +(0,1) and +(0,1) .. (3,2);
\draw[color=white, line width=4pt] (3,5) to[out=-90, in=90] (2.5,2);	
\draw (3,5) to[out=-90, in=90] (2.5,2);	
\draw (3,-1) -- (3,2);
%
\fill (-0.4,0.8) circle (0pt) node (meet2) {{\scriptsize$ghk$}};
\fill (0.8,0.8) circle (0pt) node (meet2) {{\scriptsize$gh$}};
\fill (3.85,0.8) circle (0pt) node (meet2) {{\scriptsize$hk$}};
\fill (1.75,-2.8) circle (0pt) node (meet2) {{\scriptsize$k$}};
\fill (1.6,3.8) circle (0pt) node (meet2) {{\scriptsize$k$}};
\fill (1.8,-0.2) circle (0pt) node (meet2) {{\scriptsize$h$}};
\end{tikzpicture} 
\ee 
By the ribbon property, this expression is indeed equal to~\eqref{eq:RHSInv2v2}. 

We also show that the left- and right-hand side of~\eqref{eq:351n} agree: 
\begin{align}
\begin{tikzpicture}[very thick,scale=0.75,color=blue!50!black, baseline]
%
\draw[decoration={markings, mark=at position 0.5 with {\arrow{<}},}, postaction={decorate}] 
([shift=(0:1)]1,0) arc (0:180:1);
\draw[decoration={markings, mark=at position 0.5 with {\arrow{>}},}, postaction={decorate}] 
([shift=(0:1)]1,0) arc (0:-180:1);
%
\draw[color=green!50!black] (-1,-1.5) to[out=90, in=225] ([shift=(140:1)]1,0);
%
\fill[color=black] ([shift=(-180:1)]1,0) circle (2.9pt) node[right] (meet) {{\scriptsize$\!\psi_1^{2}$}};
\fill[color=black] ([shift=(-140:1)]1,0) circle (2.9pt) node[right] (meet) {{\scriptsize$\!\psi_2^{2}$}};
\fill ([shift=(140:1)]1,0) circle (2.9pt) node[right] (meet) {};
%
\fill[color=green!50!black] (-0.8,-1.5) circle (0pt) node (M) {{\scriptsize$A$}};
\fill (1.2,0.65) circle (0pt) node (M) {{\scriptsize$T$}};
\end{tikzpicture} 
& \;\; = \;
\sum_{g,h\in G} 
\begin{tikzpicture}[very thick,scale=0.75,color=blue!50!black, baseline]
%
\draw[decoration={markings, mark=at position 0.5 with {\arrow{<}},}, postaction={decorate}] 
([shift=(0:1)]1,0) arc (0:180:1);
\draw[decoration={markings, mark=at position 0.5 with {\arrow{>}},}, postaction={decorate}] 
([shift=(0:1)]1,0) arc (0:-180:1);
%
\draw[color=green!50!black] (-1,-1.5) to[out=90, in=225] ([shift=(140:1)]1,0);
%
\fill[color=black] ([shift=(-180:1)]1,0) circle (2.9pt) node[right] (meet) {{\scriptsize$\!\psi_1^{2}$}};
\fill[color=black] ([shift=(-140:1)]1,0) circle (2.9pt) node[right] (meet) {{\scriptsize$\!\psi_2^{2}$}};
\fill ([shift=(140:1)]1,0) circle (2.9pt) node[right] (meet) {};
%
\fill[color=green!50!black] (-0.6,-1.5) circle (0pt) node (M) {{\scriptsize$A_{gh}$}};
\fill (1.2,0.55) circle (0pt) node (M) {{\scriptsize$T_{g,h}$}};
\end{tikzpicture} 
\;\; = \;
\sum_{g,h\in G} d_{m_g}^{-1} d_{m_h}^{-1} \cdot \!\!\!\!
\begin{tikzpicture}[very thick,scale=0.75,color=red!50!black, baseline]
%
\draw[decoration={markings, mark=at position 0.65 with {\arrow{<}},}, postaction={decorate}] 
([shift=(0:1)]1,0) arc (0:140:1);
\draw[decoration={markings, mark=at position 0.65 with {\arrow{>}},}, postaction={decorate}] 
([shift=(0:1)]1,0) arc (0:-140:1);
%
\draw[decoration={markings, mark=at position 0.5 with {\arrow{>}},}, postaction={decorate}] 
([shift=(0:0.7)]1,0) arc (0:180:0.7);
\draw[decoration={markings, mark=at position 0.5 with {\arrow{<}},}, postaction={decorate}] 
([shift=(0:0.7)]1,0) arc (0:-180:0.7);
%
\draw[decoration={markings, mark=at position 0.5 with {\arrow{>}},}, postaction={decorate}] 
([shift=(0:0.3)]1,0) arc (0:180:0.3);
\draw[decoration={markings, mark=at position 0.5 with {\arrow{<}},}, postaction={decorate}] 
([shift=(0:0.3)]1,0) arc (0:-180:0.3);
%
\draw (-1,-1.5) to[out=90, in=225] ([shift=(140:1)]1,0);
\draw (-0.5,-1.5) to[out=90, in=140] ([shift=(-140:1)]1,0);
%
\fill (-1.3,0.-1.5) circle (0pt) node (M) {{\scriptsize$gh$}};
\fill (-0.2,0.-1.5) circle (0pt) node (M) {{\scriptsize$gh$}};
\fill (0.15,0) circle (0pt) node (M) {{\scriptsize$g\vphantom{h}$}};
\fill (0.55,0) circle (0pt) node (M) {{\scriptsize$h\vphantom{g}$}};
\end{tikzpicture} 
\nonumber 
\\
& \;\; = \;
\sum_{g\in G} \Big( \sum_{h\in G} 
\begin{tikzpicture}[very thick,scale=0.75,color=red!50!black, baseline=.2cm]
\draw[line width=0pt] 
(2.5,0.3) node[line width=0pt] (s) {{\scriptsize$h$}}; 
\draw[redirectedred] (2,0) .. controls +(0,1) and +(0,1) .. (3,0);
\end{tikzpicture}
\, \Big) 
= 
|G|\sum_{h\in G} d_{m_h}^{-1} 
\cdot 
\begin{tikzpicture}[very thick,scale=0.75,color=blue!50!black, baseline]
%
\draw[color=green!50!black] (0,-0.5) -- (0,0.5);
\draw[color=green!50!black] (0,0.5) node[Odot] (end) {}; 
%
\fill[color=green!50!black] (0.3,-0.5) circle (0pt) node (M) {{\scriptsize$A_h$}};
\end{tikzpicture} 
= 
	 \phi^{-2} 
\cdot 
\begin{tikzpicture}[very thick,scale=0.75,color=blue!50!black, baseline]
%
\draw[color=green!50!black] (0,-0.5) -- (0,0.5);
\draw[color=green!50!black] (0,0.5) node[Odot] (end) {}; 
%
\fill[color=black] (0,0) circle (2.9pt) node[left] (meet) {{\scriptsize$\!\psi^{2}$}};
%
\fill[color=green!50!black] (0.2,-0.5) circle (0pt) node (M) {{\scriptsize$A$}};
\end{tikzpicture} 
. 
\end{align}
\end{proof}

\begin{remark}
Note that in $\widehat{\Cat{B}}$, all the $A_g$, $g \in G$, are Morita equivalent to the algebra $\one$. In $\B_1$ typically only $A_1$ is Morita equivalent to $\one$, but the bimodules $T_{g,h}$ still exhibit $A_{gh}$ as Morita equivalent to $A_g \otimes A_h$ in $\B_1$. 
Thus $A_g \otimes A_{g^{-1}}$ is Morita equivalent to $\one$ in $\B_1$, that is, all $A_g$ are necessarily Azumaya algebras in $\B_1$, see \cite{VOZ} and e.\,g.\ \cite[Sect.\,10]{frsRibbon}.
Commutative separable algebras in braided tensor categories are Azumaya iff they are isomorphic to the tensor unit \cite[Thm.\,4.9]{VOZ}, and so in this sense the construction in Section~\ref{sec:commSSFR} is complementary to the one described here. 
\end{remark}

It is natural to ask to what extent Theorem~\ref{thm:Gcrossed} depends on the choice of simple objects $m_g \in \B_g$. 
To answer this question, let $\{ \widetilde m_g \in \B_g\}_{g\in G}$ be another choice of simple objects. 
Hence by setting 
\be 
\widetilde A_g := \widetilde m_g^* \otimes \widetilde m_g
\, , \quad 
\widetilde A := \bigoplus_{g\in G} \widetilde A_g 
\ee 
and similarly defining $\widetilde T, \widetilde{\alpha}, \widetilde{\bar\alpha}, \widetilde{\psi}$ in terms of $\{ \widetilde m_g\}$ along the lines of \eqref{eq:Tghactions}--\eqref{eq:psiphiG}, we have a second special orbifold datum 
\be 
\A^{\widetilde{m}} := \big( \, \widetilde{A}, \widetilde{T}, \widetilde{\alpha}, \widetilde{\bar\alpha}, \widetilde{\psi}, \phi \, \big)
\ee 
for~$\B$. 

To relate $\A^{\widetilde{m}}$ to $\A^{m}$, first note that~$\widetilde A_g$ and~$A_g$ are Morita equivalent (in $\B=\B_1$) for all $g\in G$: 
for $X_g := \widetilde m_g^* \otimes m_g$ we have 
$X_g^* \otimes_{\widetilde A_g} X_g \cong A_g$ 
and 
$X_g \otimes_{A_g} X_g^* \cong \widetilde A_g$. 
Hence 
\be 
X := \bigoplus_{g\in G} X_g
\ee 
is an $\widetilde{A}$-$A$-bimodule exhibiting a Morita equivalence between~$A$ and~$\widetilde{A}$, and Definition~\ref{def:Moritatransport} and Proposition~\ref{prop:Moritatransport} give us another special orbifold datum for~$\B$, the Morita transport of $\mathcal{A}^m$ along~$X$: 
\be 
X(\mathcal{A}^m) = \big( \, \widetilde{A}, T^X, \alpha^X, \bar\alpha^X, \psi^X, \phi \, \big) \, . 
\ee 

\begin{lemma}
	\label{lem:Tiso-for-other-choice}
	There is a $T$-compatible isomorphism from $X(\mathcal{A}^m)$ to $\A^{\widetilde{m}}$.
\end{lemma}

\begin{proof}
	The  $T$-compatible isomorphism 
	$
	\rho \colon T^X \to \widetilde{T}
	$
	can be assembled from component maps 
	$
	\rho_{g,h} \colon (T^X)_{g,h} 
	= 
	X^{*}_{gh} \otimes_{A_{gh}} T_{g,h} \otimes_{A_{g}\otimes A_h} (X_{g} \otimes X_{h}) 
	\to \widetilde{T}_{g,h}
	$ 
	obtained via the universal property from the map $X^{*}_{gh} \otimes T_{g,h} \otimes X_{g} \otimes X_{h}  \to \widetilde{T}_{g,h}$ given by
	\be 
	\frac{1}{\sqrt{d_{m_{gh}}d_{m_{g}}d_{m_{h}}}} 
	\cdot \text{id}_{\widetilde{m}_{gh}} \otimes \ev_{m_{gh}} \otimes \ev_{m_{g}} \otimes \text{id}_{\widetilde{m}_{g}} \otimes \ev_{m_{h}} \otimes \text{id}_{\widetilde m_{h}} 
	\ee 
	for all $g,h\in G$. 
	Checking the defining condition~\eqref{eq:T-iso} of $T$-compatibility for (the components of)~$\rho$ is a straightforward string-diagram manipulation of the same type as in the proof of Theorem~\ref{thm:Gcrossed}. 
\end{proof}

As one would expect, if one maps~$\B$ into a modular tensor category via a ribbon functor, all the above orbifold data lead to isomorphic orbifold TQFTs, and in this sense the construction does not depend on the choice of simple objects: 

\corollary
Let $\B,\Cat{C}$ be ribbon fusion categories, and let $\Cat{C}$ be modular. 
Let $F\colon \B \to \Cat{C}$ be a ribbon functor and let $\{ m_g \in \B_g\}_{g\in G}$ and $\{ \widetilde m_g \in \B_g\}_{g\in G}$ be two choices of simple objects as above. 
Then 
\be 
\label{eq:3isomorbi}
(\zzc)_{F(\mathcal A^m)}
\cong 
(\zzc)_{F(X(\mathcal A^m))} 
\cong 
(\zzc)_{F(\mathcal A^{\widetilde m})}
\, .
\ee 
\endcorollary

\begin{proof}
The first isomorphism in~\eqref{eq:3isomorbi} is Corollary~\ref{cor:morita-general}.
The second isomorphism follows from Lemma~\ref{lem:Tisomequi} 
and Lemma~\ref{lem:Tiso-for-other-choice}, and from the fact that~$F$ maps $T$-compatible isomorphisms to $T$-compatible isomorphisms.
\end{proof}

\begin{remark}
Recall the notions of $G$-crossed extension and $G$-equivariantisation, e.\,g.\ from \cite{ENO2, TuraevBook2010}. 
By Theorem~\ref{thm:Gcrossed} every $G$-crossed extension $\Cat{C}_G^\times$ of a modular tensor category~$\Cat{C}$ gives rise to an orbifold datum for~$\zzc$. 
We expect that the associated orbifold TQFT is isomorphic to the Reshetikhin--Turaev theory $\mathcal{Z}^{\text{RT,} (\Cat{C}_G^\times)^G}$ corresponding to the $G$-equivariantisation $(\Cat{C}_G^\times)^G$ of $\Cat{C}_G^\times$; this is in line with previous work on gauging global symmetry groups \cite{BBCW1410.4540, CGPW1510.03475} and on geometric group orbifolds of 3-2-1-extended TQFTs \cite{SchweigertWoike201802}. 
\end{remark}

\begin{remark}
\label{rem:GextNonSS}
There is a generalisation of Theorem~\ref{thm:Gcrossed} which does not need the strong assumptions of semisimplicity and finiteness inherent to fusion categories:
Suppose that~$\widehat{\B}$ is a 
$\Bbbk$-linear ribbon crossed $G$-category such that in every graded component~$\B_{g}$ there exists a simple object~$m_{g}$ with invertible quantum dimension $d_{m_g} \in \End(\one)$ which in turn has a square root. 
Then the proof of Theorem~\ref{thm:Gcrossed} still goes through to show that $\A^m = (A, T, \alpha, \bar\alpha, \psi, \phi)$ defined as in \eqref{eq:AgGextension}--\eqref{eq:psiphiG} is a special orbifold datum for~$\B=\B_{1}$.
\end{remark} 

\begin{example}
\label{example:TY} 
As a class of concrete examples of $G$-crossed extensions and their associated orbifold data, we consider \textsl{Tambara--Yamagami categories} $\mathcal{TY}_{H,\chi,\tau}$. 
Recall that Tambara and Yamagami \cite{TY} classified $\Z_2$-extensions of pointed categories, i.\,e.\ of fusion categories where all objects are invertible. 
Such extensions are constructed from tuples $(H, \chi, \tau)$, where~$H$ is a finite abelian group, $\chi\colon H \times H \rightarrow \Bbbk^{\times}$ is a nondegenerate symmetric bicharacter, and $\tau \in \Bbbk$ is a square root of $|H|^{-1}$. 
Writing $\Z_2=\{\pm 1\}$, the graded components of $\mathcal{TY}_{H,\chi,\tau}$ are the category of $H$-graded vector spaces and vector spaces, respectively: $(\mathcal{TY}_{H,\chi,\tau})_{+1}=\Vect_{H}$ and  $(\mathcal{TY}_{H,\chi,\tau})_{-1}=\Vect$. 
The fusion rules for the $+1$-component are as in $\Vect_{H}$, the single simple object in the $-1$-component is noninvertible (unless $|H| = 1$),
and the category $\mathcal{TY}_{H,\chi,\tau}$ has a canonical spherical structure such that the quantum dimensions of all objects are positive, see e.\,g.\ \cite{GNN} for details. 
\begin{enumerate}
\item 
Consider the case of a Tambara--Yamagami category where the bicharacter~$\chi$ comes from a quadratic form $q\colon H \rightarrow \Bbbk^{\times}$, i.\,e.\ it satisfies 
$\chi(a,b)=\frac{q(a \cdot b)}{q(a)q(b)}$ for all $a,b\in H$. 
Then the category $\mathcal{TY}_{H,\chi,\tau}$ is a $\Z_2$-crossed extension of
$\Vect_{H,\chi}$ 
with the braiding on simple objects $a,b \in H$ given by (where we use $a\otimes b = ab = ba = b\otimes a$ in $\Vect_{H}$)
\begin{equation}
c_{a,b}=\chi(a,b) \cdot \id_{a \cdot b}: a \otimes b \lra b \otimes a \, ,
\end{equation}
see \cite[Prop.\,5.1]{GNN}. 
From Theorem~\ref{thm:Gcrossed} we obtain orbifold data $\orb_{\tau}$ in $(\mathcal{TY}_{H,\chi,\tau})_{+1}=\Vect_{H,\chi}$ for each choice of square root~$\tau$ of $|H|^{-1}$. 
Following the construction in the proof, we see that the algebra in the orbifold datum
is $A = \one \oplus A_H$, where $A_H := \bigoplus_{h \in H} h$ corresponds to the nontrivial element $-1 \in \Z_2$. 
By inspecting the fusion rules of $\mathcal{TY}_{H,\chi,\tau}$, one finds that the bimodules $T_{g,h}$ are given by $T_{1,1}=\one$ for $1 \in \Z_2$ and $T_{g,h}=A_H$ in all other cases. 
\item 
For $H=\Z_2$ the corresponding Tambara--Yamagami categories reduce to the familiar Ising categories \cite{EGNO-book}. 
Consider, for example, the quadratic form~$q$ such that $q(+1)=1$ and $q(-1)=\textrm{i}$, and $\tau= \pm \frac{1}{\sqrt{2}}$. 
Then $\mathcal{TY}_{\Z_2, \chi,\tau}$ are  $\Z_2$-extensions of 
$\Vect_{\Z_2,\chi}$ with a symmetric braiding coming from~$\chi$: 
The simple object in degree~$-1$ has a self-braiding which is~$-1$ times the identity. 
As in the general case in part~(i) above 
we obtain orbifold data~$\orb_{\tau}$ in the ribbon category $\Vect_{\Z_2,\chi}$ for both choices of~$\tau$. 
\end{enumerate}
\end{example}

We end with an example which relates the orbifold data of Example~\ref{example:TY}(ii) to our constructions in Section~\ref{sec:orbif-data-resh}: 

\begin{example}
\label{example:sl2}  
Let $\Cat{C}_{k}$ be  the modular tensor category associated to the affine Lie algebra $\widehat{\mathfrak{sl}}(2)_k$ at a positive integer level $k$. The category $\Cat{C}_{k}$ has $k+1$ simple objects $U_{0}, U_{1}, \ldots, U_{k}$, all of which are self-dual. 
The object $U_{k}$ is invertible and has ribbon-twist $\theta_{U_k} = \textrm{i}^{k} \cdot \id_{U_k}$. 
The simple $\Delta$-separable symmetric Frobenius algebras in $\Cat{C}_{k}$ are known up to Morita equivalence from the classification of $\Cat{C}_k$-module categories \cite{Ostrik:2001} and follow an ADE pattern. Depending on the level~$k$, there are one, two or three such Morita classes:
\begin{itemize}
\item all $k$ (case A): For every value of $k$ one has the Morita class $[A_{\textrm{A}}]$ of the simple \SSFR\ $A_{\textrm{A}}:=\one$. For $k=1 \, \text{mod}\, 2$ this is furthermore the only such Morita class, so these values of $k$ do not provide interesting examples of the constructions in Section~\ref{sec:commSSFR} or in this section.
\item
$k=0 \, \text{mod}\, 4$ (case $\textrm{D}_\text{even}$): There is an up-to-isomorphism unique structure of a \SSFR\ on $A_{\textrm{D}} := U_0 \oplus U_k$. Its Morita class $[A_{\textrm{D}}]$ is different from $[A_{\textrm{A}}]$. The algebra $A_{\textrm{D}}$ is commutative and one can thus apply the construction in Section~\ref{sec:commSSFR}. 
The algebra $A_{\textrm{D}}$ is not Azumaya (and hence no algebra in $[A_{\textrm{D}}]$ is), and so it cannot appear as part of a $G$-extension as discussed in this section.
\item
$k=2 \, \text{mod}\, 4$ (case $\textrm{D}_\text{odd}$):
As above, $A_{\textrm{D}} := U_0 \oplus U_k$ is a \SSFR\ in an up-to-isomorphism unique way, and its Morita class $[A_{\textrm{D}}]$ is distinct from $[A_{\textrm{A}}]$. But this time, $A_{\textrm{D}}$ is noncommutative and in fact Azumaya. The full subcategory spanned by $U_0$ and $U_k$ is ribbon-equivalent to $\Vect_{\Z_2,\chi}$ with the symmetric braiding from~$\chi$ as in (ii) above. 
Put differently, 
there is a fully faithful ribbon functor $F \colon \Vect_{\Z_2,\chi} \rightarrow \Cat{C}_{k}$.
The two orbifold data $\orb_{\tau}$ for $\tau= \pm \frac{1}{\sqrt{2}}$ in $\Vect_{\Z_2,\chi} $ give 
by Proposition \ref{proposition:push-forward}  orbifold data $F(\orb_{\tau})$  in $\Cat{C}_{k}$. 
\item
$k \in \{ 10,28 \}$ (cases $\textrm{E}_6$, $\textrm{E}_8$): There are commutative simple \SSFRs\ $A_{\textrm{E}_6}$, $A_{\textrm{E}_8}$, which provide a third Morita class $[A_{\textrm{E}_6}]$, resp.\ $[A_{\textrm{E}_8}]$, in addition to $[A_{\textrm{A}}]$ and $[A_{\textrm{D}}]$ at these levels.
The corresponding categories of local modules are equivalent to the modular tensor categories obtained from $\widehat{\mathfrak{sp}}(4)_1$ and $\widehat{\mathfrak{g}}(2)_1$, respectively, see \cite{Ostrik:2001}. The construction in Section~\ref{sec:commSSFR} applies, and as mentioned there, we expect the orbifolds corresponding to $A_{\textrm{E}_6}$ and $A_{\textrm{E}_8}$ to be equivalent to the Reshetikhin--Turaev TQFTs obtained from these two modular tensor categories.
\item
$k=16$ (case $\textrm{E}_7$): 
There is a simple \SSFR\ $A_{\textrm{E}_7}$ which generates a third Morita class in addition to $[A_{\textrm{A}}]$ and $[A_{\textrm{D}}]$ at this level. The Morita class $[A_{E_7}]$ does not contain a commutative representative, and the algebra $A_{\textrm{E}_7}$ is not Azumaya.
We do not know if $A_{\textrm{E}_7}$ can form part of an orbifold datum.
\end{itemize}
\end{example}

\end{document}